%% file: main.tex
\documentclass[a4paper,11pt,reqno]{article}
\usepackage{relsize}
\usepackage{color}
\usepackage{mathtools}
\usepackage[dvipsnames,table]{xcolor}
\usepackage[noadjust]{cite}
\usepackage{mathrsfs}
\usepackage{float}

\usepackage{hyperref, enumitem}

\hypersetup{
  colorlinks   = true, 
  urlcolor     = blue, 
  linkcolor    = Purple, 
  citecolor   = red 
}

\usepackage{amsmath,amsthm,amssymb}
\usepackage{xurl}
\hypersetup{breaklinks=true}
\usepackage[margin=2.5cm]{geometry}
\usepackage{bm}

\usepackage{graphicx}

\usepackage{algorithmic}
\usepackage[T1]{fontenc}
\usepackage{authblk}
\usepackage[english]{babel}
\usepackage{stackrel}

\usepackage[capitalize,noabbrev]{cleveref}

\usepackage{tikz}
\usetikzlibrary{decorations.pathreplacing}
\usetikzlibrary{shapes.geometric,quotes}
\usetikzlibrary{shapes.misc}

\usepackage{tikz-cd}

\usepackage{pgf}
\usepackage{tikz}
\usepackage{tikz-3dplot}
\usepackage[outline]{contour}
\usepackage{ mathdots }
\usepackage{ytableau}
\usepackage{adjustbox}

\tikzset{cross/.style={cross out, draw=black, minimum size=2*(#1-\pgflinewidth), inner sep=0pt, outer sep=0pt}, 
cross/.default={1pt}}
\theoremstyle{definition}
\newtheorem{theorem}{Theorem}[section]

\newtheorem{proposition}[theorem]{Proposition}
\newtheorem{corollary}[theorem]{Corollary}
\newtheorem{lemma}[theorem]{Lemma}
\newtheorem{definition}[theorem]{Definition}
\newtheorem{example}[theorem]{Example}
\newtheorem{remark}[theorem]{Remark}

\newtheorem{conjecture}[theorem]{Conjecture}
\newtheorem{observation}[theorem]{Observation}

\newcommand{\F}{\mathbb F}
\newcommand{\Z}{\mathbb Z}
\newcommand{\R}{\mathbb R}

\newcommand{\mA}{\mathcal A}
\newcommand{\mB}{\mathcal B}
\newcommand{\mC}{\mathcal C}
\newcommand{\C}{\mathcal C}
\newcommand{\mD}{\mathcal D}

\newcommand{\mE}{\mathcal E}
\newcommand{\mG}{\mathcal G}
\newcommand{\mF}{\mathcal F}
\newcommand{\mR}{\mathcal R}
\newcommand{\mT}{\mathcal T}
\newcommand{\mS}{\mathcal S}

\newcommand{\card}[1]{\left| #1\right|}

\usepackage{todonotes}

\usepackage[ruled,linesnumbered]{algorithm2e}
\SetKwComment{Comment}{// }{}
\DontPrintSemicolon

\newcommand{\N}{\mathbb{N}}

\newcommand{\irr}{\ensuremath{\mathfrak{Irr}}}
\newcommand{\poly}{\ensuremath{\mathfrak{P}}}
\newcommand{\spoly}
{\ensuremath{\mathfrak{K}}}
\newcommand{\lpoly}
{\ensuremath{\mathfrak{L}}}

\usepackage{comment}
\usepackage{subcaption}

\DeclareMathOperator{\rk}{rk}

\usepackage{authblk}
\title{Irreducible Ferrers
diagrams in the Etzion-Silberstein conjecture }%(and their polytopal characterization)
\author[1]{Hugo Beeloo-Sauerbier Couv\'ee}
\author[2]{Alessandro Neri}
\affil[1]{Associate Professorship of Coding and Cryptography, Technical University of Munich, Germany, \url{hugo.sauerbier-couvee@tum.de}}
\affil[2]{Department of Mathematics and Applications ``R. Caccioppoli'', University of Naples Federico II, Italy, \url{alessandro.neri@unina.it}}
\date{ }

\rmfamily\mdseries\scshape

\begin{document}

\maketitle
%\tableofcontents

%\newpage
\begin{abstract}
\noindent The Etzion--Silberstein conjecture asserts that, for any finite field $\mathbb F$, Ferrers diagram $\mathcal D$, and integer $d$, there exists a linear matrix code supported on $\mathcal D$ with minimum rank distance $d$ that attains a natural upper bound on its dimension. Codes achieving this bound are called \emph{maximum Ferrers diagram} (MFD) codes. While the conjecture has been established for several classes of diagrams (including rectangular, monotone, and MDS-constructible cases), it remains open in general.

In this paper, we study the \emph{reducibility} of Ferrers diagrams. For a fixed distance $d$, a diagram $\mathcal D$ is said to \emph{reduce} to $\mathcal D'$ if an MFD code for $(\mathcal D,d)$ can be obtained from one for $(\mathcal D',d)$ via shortening or inclusion. Diagrams that are not reducible are called \emph{irreducible}. We show that the conjecture holds for all diagrams if and only if it holds for irreducible ones, thereby reducing the problem to this fundamental class.

Our main result provides a complete characterization of irreducible diagrams: for each $d$, they correspond exactly to the integer points of a polytope $\mathfrak{P}_d \subset \mathbb{R}^{2d-3}$. We prove that these polytopes are integral, enabling the use of Ehrhart-theoretic tools to study their structure. Finally, we formulate a new conjecture on puncturing and inclusion of maximum rank distance codes, and show that it arises as a special case of the Etzion--Silberstein conjecture.
\end{abstract}

{
  \hypersetup{linkcolor=black}
  \tableofcontents
}

\newpage

\section{Introduction}

\noindent Linear spaces of matrices with prescribed rank conditions have been widely studied in mathematics because of their relevance to several research areas. Already in the late 50's, James first \cite{james1957whitehead}, and Adams later \cite{adams1962vector}, investigated linear spaces of non-singular square matrices over the reals, in connection with the classical problem in differential geometry of the hairy ball theorem and in general of vector fields on spheres. Additional fields of research that involve linear spaces of matrices with rank properties include algebraic geometry \cite{westwick1972spaces}, Clifford algebras and mathematical physics \cite{hopf1940systeme,lounesto2001clifford}, and Young measures in PDE theory and calculus of variations \cite{bhattacharya1994restrictions}.

More recently, matrix spaces with rank constraints have attracted significant attention in coding theory, due to the growing interest in rank-metric codes, originally introduced by Delsarte \cite{delsarte1978bilinear}. These are spaces of matrices over a (typically finite) field, in which the error-correction capability is proportional to the minimum rank of a nonzero element, called the \emph{minimum rank distance}. In fact, one of the main goals is to design rank-metric codes whose dimension is  as large as possible for the given minimum rank distance. One of the main motivations for their study comes from applications in random network coding and subspace codes \cite{silvaetal2008a,koetter2008coding}.
 In the last decade, rank-metric codes have revealed deep connections to additional areas of mathematics and computer science, including finite geometry \cite{sheekey201913}, tensors \cite{byrne2019tensor},  semifields \cite{sheekey2020new}, and matroids \cite{Jurrius_Pellikaan_2018}.

The study of rank-metric codes has been generalized to spaces of matrices with prescribed support, with a particular focus on matrices supported on Ferrers diagrams. This development was largely motivated by a question posed by Etzion and Silberstein in 2009 \cite{etzion2009error}. In that work, they conjectured that, over any finite field and for any Ferrers diagram, there exists a space of matrices with prescribed minimum rank distance whose dimension attains a certain theoretical upper bound. This statement is now known as the \emph{Etzion-Silberstein conjecture}.  Because of the origins of this problem, throughout the paper we will use the terminology coming from coding theory, although the problem itself is purely algebraic and combinatorial in nature.

\subsection*{Formal setting}

For the purposes of this paper, the set of natural numbers $\N$ will be $\{1,2,3,\ldots\}$.
For a natural number $n$, the set $\{1,2,\ldots,n\}$ is denoted $[n]$. 

Let $\F$ be a finite field, $m,n$ natural numbers and $\F^{m \times n}$ the space of $(m \times n)$-matrices over $\F$. This space is naturally equipped with a distance function $d_{\operatorname{R}}$ called the \emph{rank metric}:
\[
d_{\operatorname{R}}(A,B) := \rk(B - A),
\]
where $\rk(M)$ is the $\F$-rank of a matrix $M \in \F^{m \times n}$. Classical rank-metric coding theory,  originally introduced by Delsarte in 1978 over finite fields \cite{delsarte1978bilinear}, studies (linear) codes in $\F^{m \times n}$ endowed with  $d_{\operatorname{R}}$. As a natural generalization, Etzion and Silberstein initiated in 2009 the study of \emph{Ferrers diagram rank-metric codes} \cite{etzion2009error}. These are spaces of matrices over a field $\F$ supported on a given Ferrers diagram $\mD$, i.e., a finite subset $\mD \subseteq \N^2$ satisfying the condition: if $(x,y) \in \mD$, then $(i,j) \in \mD$ for all $i \in [x]$ and all $j \in [y]$. Formally, for a given Ferrers diagram $\mD \subseteq [n]^2$, the space of matrices supported on $\mD$ is given by
\[
\F^{\mD}:=\left\{M\in \F^{n\times n}\,:\, m_{ij}=0 \mbox{ if } (i,j)\notin \mD \right\}.
\]
A $[\mD,k,d]_\F$ \emph{(Ferrers diagram rank-metric) code} $\mC$ is a $k$-dimensional $\F$-subspace of $\F^{\mD}$ such that 
\[
d:=\min\{\rk(M)\,:\, M\in \mC\setminus \{0\}\}.
\]
The original motivation for the study of these codes was to provide a tool to construct large subspace codes in random network coding thanks to the so-called \emph{multilevel construction} \cite{etzion2009error}. We will not go into details of how Ferrers diagrams are related to network coding, but we will rather give the minimum background needed to understand the main problem dealt with in this paper: the \emph{Etzion-Silberstein conjecture}. However, the interested reader is referred to \cite{silvaetal2008a,etzion2009error,silva2009metrics,etzion2011error,koetter2008coding} for the use of rank-metric and subspace codes in random network coding.

\subsection*{The Etzion-Silberstein conjecture}

For a $[\mD,k,d]_{\F}$ code $\mC$, Etzion and Silberstein derived a bound on its parameters {\cite[Theorem 1]{etzion2009error}}. In particular, for a fixed diagram $\mD$ and minimum distance $d$, the dimension of $\mC$ is upper-bounded by
\[
k \leq \nu_{\min}(\mD,d) := \min_{0\le j \le d-1}\nu_j(\mD,d),
\]
where the quantity $\nu_j(\mD,d)$ is equal to the cardinality of $\mD$ after deleting the first $j$ columns and the first $d-1-j$ rows. This bound relies on the usual arguments for Singleton-like bounds in coding theory, and at the same time, Etzion and Silberstein also conjectured this bound to be tight over \emph{every} finite field $\F$. Namely, a $[\mD,k,d]_{\F}$ code is called a \emph{maximum Ferrers diagram (MFD) code} if $k=\nu_{\min}(\mD,d)$. The \emph{Etzion-Silberstein conjecture} now states:

\medskip

{
\noindent
\textbf{Conjecture} (\cite[Conjecture 1]{etzion2009error})\textbf{.}{
For every Ferrers diagram $\mD$, every positive integer $d$ and every finite field $\F$, there exists a $[\mD,\nu_{\min}(\mD,d),d]_{\F}$ MFD code.
}}

\medskip

Of course, the Etzion-Silberstein conjecture can be restricted to some families of pairs $(\mD,d)$ and to some finite fields $\F$. In that case, we say that the \emph{Etzion-Silberstein conjecture holds for the pair $(\mD,d)$ over $\F$} if there exists a $[\mD,\nu_{\min}(\mD,d),d]_{\F}$ MFD code.

\subsection*{Some known constructions and solved cases}

Since its formulation, the Etzion--Silberstein conjecture has been proved to hold only in certain cases; see, e.g., \cite{etzion2016optimal,antrobus2019maximal,gorla2017subspace,liu2019constructions,liu2019several,neri2023proof,rakhi2023}. In general, however, it remains widely open. The case of rectangular $[m]\times [n]$ Ferrers diagrams was already settled by Delsarte \cite{delsarte1978bilinear}; in this special case, MFD codes are also known as \emph{maximum rank distance (MRD) codes}, according to classical rank metric coding theory. Since then, only a few systematic constructions of MFD codes have been obtained for other families of Ferrers diagrams. 

In \cite{etzion2016optimal}, the conjecture was proved for \emph{MDS-constructible} pairs $(\mD,d)$ over sufficiently large fields. This requirement on the field size stems from the use of MDS codes in the Hamming metric, whose existence imposes lower bounds on the field size. This restriction was recently removed in \cite{neri2023proof}: the main idea is to construct MFD codes on all triangular diagrams $\mT_n$ with column heights $(n,n-1,\ldots,2,1)$ using flags arising from a special nilpotent endomorphism, and then apply a \emph{reduction argument} (see Proposition~\ref{prop: reduction}) to prove the conjecture for all MDS-constructible pairs over arbitrary finite fields. 

Additional systematic constructions of MFD codes for special block Ferrers diagrams have been obtained in \cite{rakhi2023} using automorphisms of function fields, in \cite{neri2023proof} through special flags of subspaces, and recently generalized in \cite{calderini2026} via maximum sum-rank distance codes.

\subsection*{Our contributions}

In this paper, we investigate the notion of \emph{reducibility} of Ferrers diagrams; for a fixed integer $d$, a diagram $\mathcal{D}$ \emph{reduces} to $\mathcal{D}'$ if an MFD code for $(\mathcal{D},d)$ can be obtained by shortening or inclusion of some MFD code for $(\mathcal{D}’,d)$. If a diagram $\mathcal{D}$ does not reduce to any other diagram for this $d$, we call the pair $(\mD,d)$ \emph{irreducible}. First, we show that the Etzion-Silberstein conjecture for irreducible diagram pairs implies the conjecture for all diagrams.

\medskip

\noindent
\textbf{Theorem} \textbf{\ref{thm: es conjecture true iff irred}}\textbf{.}
The Etzion-Silberstein conjecture holds true if and only if it holds true for every irreducible pair $(\mD,d)$ and finite field $\F$.

\medskip

\noindent
As one of our main results, we characterize, for each integer $d$, all sufficiently large irreducible diagrams in terms of the quantities $\nu_j(\mD,d)$.

\medskip

%\newpage

\noindent
\textbf{Theorem} \textbf{\ref{thm:main} (Main classification theorem, informally).}
Let $(\mD,d)$ be a pair with  $(d,d)\in \mD$.
Then the following are equivalent:
\begin{enumerate}[label = (\arabic*)]
\item  $(\mD,d)$ is irreducible.
\item  $(\mD,d)$ satisfies 
\begin{itemize}
    \item $\mD \cap (\N_{\geq d-1})^2 = \{d-1,\ldots,a\} \times  \{d-1,\ldots,b\}$ for some $a,b \geq d$; 
    \item $\nu_0(\mD,d) = \nu_{d-1}(\mD,d)$;
    \item $\nu_j(\mD,d) \geq \nu_0(\mD,d) $ for all $j \in [d-2]$.
\end{itemize}
\end{enumerate}

\medskip
\noindent
Next, we use this classification to define a polytope $\poly_d$ in $\mathbb{R}^{2d-3}$ for each $d$ (see Definition \ref{def: Pd constraints}) whose integer points correspond to the irreducible diagram pairs.

\medskip
\noindent
\textbf{Theorem} \textbf{\ref{thm bijection mu with polytope} (Polytopal characterization, informally).}
For each $d \geq 3$, there is a natural bijection between the integer points of $\poly_d$ and the families of irreducible diagram pairs $(\mD,d)$ with $(d,d) \in \mD$.

\medskip
\noindent
We then show that the polytopes $\poly_d$ are integral, which makes them suitable for studying their properties with Ehrhart polynomials. Additionally, we conjecture the combinatorial structure of these polytopes.
 
\medskip
\noindent
\textbf{Theorem} \textbf{\ref{thm:main_poly} (Integrality of vertices).}
 For $d \geq 3$, the polytope $\poly_d$ is an integral polytope with $3^{d-2}$ vertices.

\medskip
\noindent
\textbf{Conjecture} \textbf{\ref{conj: product triangles} (Product of triangles).}
For all $d \geq 3$, the polytope $\poly_d$ is combinatorially equivalent to the Cartesian product of $d-2$ triangles.

\medskip
\noindent
We verified this conjecture computationally for all $3 \leq d \leq 7$. In the remainder of the paper, we investigate the Etzion-Silberstein conjecture in closer detail for $d = 3$, linking it to a conjecture on MRD codes.

\medskip
\noindent
\textbf{Theorem} \textbf{\ref{thm: MRD conj d=3} (MRD code correspondence for $d=3$).} Let $\F$ be a finite field. The following are equivalent.
\begin{enumerate}[label = (\arabic*)]
    \item  The Etzion-Silberstein conjecture holds for $d=3$ over any Ferrers diagram $\mD$.
    \item For every $n\ge 3$ there exists an
    $[n\times (n-1),n(n-3),3]_\F$ MRD code $\C$ such that its puncturing on one row is contained in an $[(n-1)\times (n-1),(n-1)(n-2),2]_\F$ MRD code.
\end{enumerate}

\medskip
\noindent
Lastly, we generalize this correspondence for larger $d$ and obtain a new conjecture on MRD codes.

\medskip
\noindent
\textbf{Conjecture} \textbf{\ref{conj:punct_ext_MRD} (Puncturing-inclusion of MRD codes).}
 Let $n,d \in \N$ with $1<d<n$, and let $\F$ be a finite field. There exists an $[n\times (n-1),n(n-d),d]_\F$ MRD code $\C$ such that its puncturing on one row  is contained in an $[(n-1)\times (n-1),(n-1)(n-d+1),d-1]_\F$ MRD code.

\subsection*{Outline of the paper}

In Section \ref{sec: prelim} we give the preliminaries on Ferrers diagrams, rank-metric codes and the Etzion-Silberstein conjecture.\\

\noindent In Section \ref{sec: irred pairs} we introduce the theory of irreducible diagram pairs and reduce the Etzion-Silberstein conjecture to this class in Theorem \ref{thm: es conjecture true iff irred}. We do this by endowing the edges of  Young's lattice with an orientation, and then studying the resulting \emph{Young digraph}. \\

\noindent In Section \ref{sec:classification} we provide the necessary lemmas and propositions to prove our main classification Theorem \ref{thm:main}, which characterizes the irreducible Ferrers diagram pairs $(\mD,d)$, under the assumption that $(d,d)\in \mD$.\\

\noindent In Section \ref{sec: polytopal char} we show the correspondence between irreducible Ferrers diagram pairs and integer points of certain polytopes in Theorem \ref{thm bijection mu with polytope}. We further investigate the structure of these polytopes and prove in Theorem \ref{thm:main_poly} that they are integral.\\

\noindent Section \ref{sec: d=3} focuses on the special case of minimum rank distance $d=3$. Here, we identify 4 families of irreducible diagram pairs $(\mD,3)$, and show in Theorem \ref{thm: MRD conj d=3}, that this case is equivalent to a puncturing-inclusion property of rectangular MRD codes of minimum rank distance $3$.\\

\noindent Finally, in Section \ref{sec:punct_ext_conjecture} we extend the argument of $d=3$ to general $d$, and obtain the more general puncturing-inclusion MRD conjecture \ref{conj:punct_ext_MRD}, which arises as a special instance of the Etzion-Silberstein conjecture. 

\subsection*{Acknowledgements}

This research was partially supported by the Italian National Group for Algebraic and Geometric Structures and their Applications (GNSAGA - INdAM). A. Neri is supported by the INdAM - GNSAGA Project CUP E53C24001950001  ``Noncommutative polynomials in coding theory''. H. Beeloo-Sauerbier Couvée is supported by the German Research Council (DFG) as an ANR-DFG project under Grant no. WA 3907/9-1.

\section{Preliminaries}\label{sec: prelim}
Recall that throughout this paper, the set of natural numbers $\N$ is $\{1,2,3,\ldots\}$, and for a natural number $n$, the set $\{1,2,\ldots,n\}$ is denoted $[n]$. 

\subsection{Ferrers diagrams}
\begin{definition}
    A \textbf{Ferrers diagram} is a \textit{finite} subset $\mD \subseteq \N^2$ satisfying the condition
\begin{itemize}
    \item if $(x,y) \in \mD$, then $(i,j) \in \mD$ for all $i \in [x]$ and all $j \in [y]$.
\end{itemize}
Furthermore, $\mD$ is said to be of \textbf{order $\boldsymbol{n}$} if $\mD \subseteq [n]^2$, and of \textbf{proper order $\boldsymbol{n}$} if $n$ is minimal in this respect, i.e. $n = \min\{i \in \N \mid \mD \subseteq [i]^2\}$. 
\end{definition}

We will represent Ferrers diagrams \textit{top-left aligned}, i.e. when representing $\N^2$ as a 2-dimensional grid, the point $(i,j)$ corresponds to an element in the $i$-th row from the top and $j$-th column from the left.

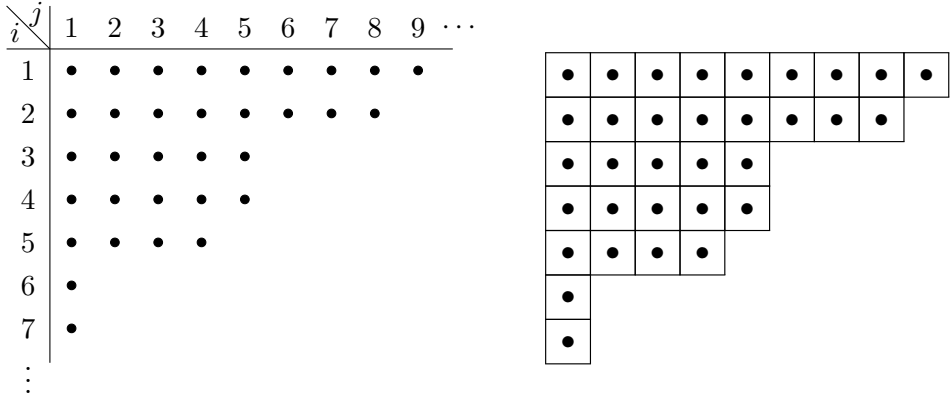
\begin{figure}[h]
\begin{center}
\begin{subfigure}[t]{0.4\textwidth}
    \begin{tikzpicture}[scale=0.57]

\draw[black,fill=black] (1,-1) circle (0.1cm);
%\draw[black] (1,-1) circle (0.25cm) node[below] {\scalebox{0.7}{$(1,1)$}};
\draw[black,fill=black]  (2,-1) circle (0.1cm);
\draw[black,fill=black]  (3,-1) circle (0.1cm);
\draw[black,fill=black] (4,-1) circle (0.1cm);
\draw[black,fill=black] (5,-1) circle (0.1cm);
\draw[black,fill=black] (6,-1) circle (0.1cm);
\draw[black,fill=black] (7,-1) circle (0.1cm);
\draw[black,fill=black] (8,-1) circle (0.1cm);
\draw[black,fill=black] (9,-1) circle (0.1cm);

\draw[black,fill=black] (1,-2) circle (0.1cm);
\draw[black,fill=black] (2,-2) circle (0.1cm);
\draw[black,fill=black] (3,-2) circle (0.1cm);
\draw[black,fill=black] (4,-2) circle (0.1cm);
\draw[black,fill=black] (5,-2) circle (0.1cm);
\draw[black,fill=black] (6,-2) circle (0.1cm);
\draw[black,fill=black] (7,-2) circle (0.1cm);
\draw[black,fill=black] (8,-2) circle (0.1cm);

\draw[fill=black] (1,-3) circle (0.1cm);
\draw[fill=black] (2,-3) circle (0.1cm);
\draw[fill=black] (3,-3) circle (0.1cm);
\draw[fill=black] (4,-3) circle (0.1cm);

\draw[fill=black] (5,-3) circle (0.1cm);

\draw[fill=black] (1,-4) circle (0.1cm);
\draw[fill=black] (2,-4) circle (0.1cm);
\draw[fill=black] (3,-4) circle (0.1cm);
\draw[fill=black] (4,-4) circle (0.1cm);
\draw[fill=black] (5,-4) circle (0.1cm);

\draw[fill=black] (1,-5) circle (0.1cm);
\draw[fill=black] (2,-5) circle (0.1cm);
\draw[fill=black] (3,-5) circle (0.1cm);
\draw[fill=black] (4,-5) circle (0.1cm);

\draw[fill=black] (1,-6) circle (0.1cm);

\draw[fill=black] (1,-7) circle (0.1cm);

\draw[-] (-0.5,0.5) -- (0.5,-0.5);
\draw[-] (0.5,0.5) -- (0.5,-7.8);
\draw[-] (-0.5,-0.5) -- (9.8,-0.5);
\draw (-0.3,0.35) node[below] {$i$};
\draw (-0.2,0.3) node[right] {$j$};

\draw (0,-1) node[] {$1$};
\draw (0,-2) node[] {$2$};
\draw (0,-3) node[] {$3$};
\draw (0,-4) node[] {$4$};
\draw (0,-5) node[] {$5$};
\draw (0,-6) node[] {$6$};
\draw (0,-7) node[] {$7$};
\draw (0,-8) node[] {$\vdots$};

\draw (1,0) node[] {$1$};
\draw (2,0) node[] {$2$};
\draw (3,0) node[] {$3$};
\draw (4,0) node[] {$4$};
\draw (5,0) node[] {$5$};
\draw (6,0) node[] {$6$};
\draw (7,0) node[] {$7$};
\draw (8,0) node[] {$8$};
\draw (9,0) node[] {$9$};
\draw (10,0) node[] {$\cdots$};
\end{tikzpicture}
\end{subfigure}\qquad\begin{subfigure}[t]{0.4\textwidth}\vspace{-4.65cm}
\ydiagram[\bullet]{9,8,5,5,4,1,1}\end{subfigure}
\end{center}
    \caption{Example of top-left aligned representation of a Ferrers diagram of proper order 9. In the paper, diagrams are either represented by dots alone (left), or dots surrounded by square boxes (right).}
    \label{fig:example}
\end{figure}

A Ferrers diagram of order $n$ can be compactly represented as a sequence of column heights 
$(c_1, c_2, \ldots, c_n)$ with $c_1 \geq c_2 \geq \ldots \geq c_n$ given by $c_i = \card{\mD \cap (\N \times \{i\}) }$, or as a sequence of row lengths $(r_1, r_2, \ldots, r_n)$ with $r_1 \geq r_2 \geq \ldots \geq r_n$ given by $r_i = \card{\mD \cap (\{i\} \times \N) }$. In the example of Figure \ref{fig:example}, we have $(c_1, c_2, \ldots, c_9) = (7,5,5,5,4,2,2,2,1)$ and $(r_1, r_2, \ldots, r_9) = (9,8,5,5,4,1,1,0,0)$. Note that the proper order of a diagram $\mD$ is equal to $\max\{c_1, r_1\}$.
It is occasionally useful to extend the column sequence or row sequence of a diagram by appending zeros on the end. With abuse of notation, when more convenient, we will identify the diagram with its sequence of column lengths, and write $\mD=(c_1,c_2,\ldots,c_n)$.

\newpage

The \textbf{removable points} of a Ferrers diagram $\mD$ are the points $P \in \mD$ such that $\mD \setminus \{P\}$ is also a Ferrers diagram. If $\mD$ is non-empty, these points constitute a non-empty subset of $\mD$ denoted by $\mR(\mD)$. It can be characterized as
\begin{align*}
\mR(\mD) &= \{(i, r_i) \in \mD \,\mid\,  i \geq 1\text{ with }  r_i > r_{i+1} \}\\
&= \{(c_j, j) \in \mD \,\mid\,   j \geq 1 \text{ with }  c_j > c_{j+1}\}.
\end{align*} 

Likewise, the \textbf{addible points} of $\mD$ are the points $P \in \N^2 \setminus \mD$ such that $\mD \cup \{P\}$ is also a Ferrers diagram. These points constitute a non-empty subset of $\mD$ denoted by $\mA(\mD)$. It can be characterized as
\begin{align*}
\mA(\mD) &= \{(i, r_i + 1) \in \N^2 \setminus\mD \,\mid\, i=1, \text{ or }  i \geq 2  \text{ with } r_{i-1} > r_{i}\}\\
&= \{(c_j+1, j) \in \N^2 \setminus \mD \,\mid\, j=1, \text{ or } j \geq 2  \text{ with } c_{j-1} > c_{j} \}.
\end{align*}

\begin{figure}[h]
    \centering
    \ydiagram[*(white!70!gray) ]
{9+1,8+1,5+1,0,4+1,1+1,0,0+1}
*[*(white!50!red) \bullet]{8+1,7+1,0,4+1,3+1,0,1}
*[*(white) \bullet]{8,7,5,4,3,1}

    \caption{The dotted boxes represents the  Ferrers diagram $\mD$  defined in Figure \ref{fig:example}.  The removable points $\mR(\mD)$ are represented with red background dotted boxes \scalebox{0.5}{\ydiagram[*(white!50!red) \bullet]{1}}, and the addible points $\mA(\mD)$ are represented with gray background empty boxes \scalebox{0.5}{\ydiagram[*(white!70!gray)]{1}}. }   
\end{figure}

The \textbf{adjoint} or \textbf{transpose} of a Ferrers diagram $\mD$ is formed by swapping rows and columns, i.e. the Ferrers diagram defined as
\[
\mD^\top := \{(y,x) \in \N^2 \mid (x,y) \in \mD\}.
\]

\subsection{Ferrers diagram rank-metric codes and the Etzion-Silberstein conjecture}

Let us fix a (finite) field $\F$. For a given Ferrers diagram $\mD$ of order $n$, we consider the Ferrers diagram matrix space
$$\F^{\mD}:=\left\{M\in \F^{n\times n}\,:\, m_{ij}=0 \mbox{ if } (i,j)\notin \mD \right\}.$$

\begin{definition}
    A $[\mD,k,d]_\F$ \textbf{(Ferrers diagram rank-metric) code} $\mC$ is a $k$-dimensional $\F$-subspace of $\F^{\mD}$ such that 
    $$d:=\min\{\rk(M)\,:\, M\in \mC\setminus \{0\}\}.$$
\end{definition}

Ferrers diagram rank-metric codes can be seen as natural generalizations of classical rank-metric codes, which were originally introduced by Delsarte in 1978 over finite fields \cite{delsarte1978bilinear}. In fact, if one considers $\mD=[m]\times [n]$, then a $[[m]\times[n],k,d]_{\F}$ code is simply an $\F$ subspace of $\F^{m\times n}$, that is, a classical rank-metric code and we will simply denote it as an $[m\times n,k,d]_{\F}$ code.

For a $[\mD,k,d]_{\F}$ code, Etzion and Silberstein derived a bound on their parameters.

\begin{theorem}[{\cite[Theorem 1]{etzion2009error}}]\label{thm:singES_bound}
    Let $\mD=(c_1,\ldots,c_n)$ be a Ferrers diagram of order $n$ and let $\C$ be an $[\mD,k,d]_{\F}$ code. Then
    $$k\le \nu_{\min}(\mD,d):=\min_{0\le j \le d-1}\nu_j(\mD,d),$$
    where
$$\nu_j(\mD,d):=\sum_{i=j+1}^n\max\{0,c_i-d+j\}.$$
\end{theorem}

\begin{remark}
    Let $\mD$ be a Ferrers diagram of order $n$. The quantity $\nu_j(\mD,d)$ is equal to the cardinality of $\mD$ after deleting the first $j$ columns and the first $d-1-j$ rows.

\end{remark}

    \begin{example}
Let $\mD=(5,4,4,1,1)$ be a Ferrers diagram of order $5$ and let $d=3$.     Then 
    \[
    \nu_{\min}(\mD,d)=\min\left\{\sum_{i=j+1}^{5}\max\{0,c_i-3+j\}\ :\ j\in\{0,1,2\}\right\}=\min\{7, 6, 6\}=6,
    \]
    which is the same as the minimum among the number of dots remaining after performing the following deletions:

      \bigskip 
            \centering\reflectbox
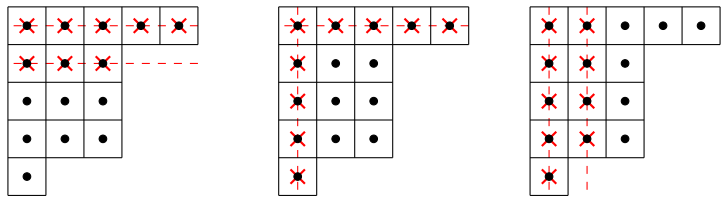
{\begin{tikzpicture}[scale=0.5]
 \draw[help lines, very thick, white] (0,1) -- (0,-4) -- (5,-4)--(5,1)--(0,1);
\draw[red, dashed] (0,0.5) -- (5,0.5);
\draw[red, dashed] (0,-0.5) -- (5,-0.5);
\draw (0.5,0.5) node [cross=3.5pt,red,thick] {};
\draw (1.5,0.5) node [cross=3.5pt,red,thick] {};
\draw (2.5,0.5) node [cross=3.5pt,red,thick] {};
\draw (3.5,0.5) node [cross=3.5pt,red,thick] {};
\draw (4.5,0.5) node [cross=3.5pt,red,thick] {};
\draw (2.5,-0.5) node [cross=3.5pt,red,thick] {};
\draw (3.5,-0.5) node [cross=3.5pt,red,thick] {};
\draw (4.5,-0.5) node [cross=3.5pt,red,thick] {};
\draw[fill=black] (0.5,0.5) circle (0.1cm);
\draw[fill=black] (1.5,0.5) circle (0.1cm);
\draw[fill=black] (2.5,0.5) circle (0.1cm);
\draw[fill=black] (3.5,0.5) circle (0.1cm);
\draw[fill=black] (4.5,0.5) circle (0.1cm);
\draw[fill=black] (2.5,-0.5) circle (0.1cm);
\draw[fill=black] (3.5,-0.5) circle (0.1cm);
\draw[fill=black] (4.5,-0.5) circle (0.1cm);
\draw[fill=black] (2.5,-1.5) circle (0.1cm);
\draw[fill=black] (3.5,-1.5) circle (0.1cm);
\draw[fill=black] (4.5,-1.5) circle (0.1cm);
\draw[fill=black] (2.5,-2.5) circle (0.1cm);
\draw[fill=black] (3.5,-2.5) circle (0.1cm);
\draw[fill=black] (4.5,-2.5) circle (0.1cm);
\draw[fill=black] (4.5,-3.5) circle (0.1cm);
%%%%%%%%%%%%%%%
\draw[black] (0,0) -- (5,0);
\draw[black] (0,1) -- (5,1);
\draw[black] (2,-1) -- (5,-1);
\draw[black] (2,-2) -- (5,-2);
\draw[black] (2,-3) -- (5,-3);
\draw[black] (4,-4) -- (5,-4);
%%%%%%%%%%%
\draw[black] (5,1) -- (5,-4);
\draw[black] (4,1) -- (4,-4);
\draw[black] (3,1) -- (3,-3);
\draw[black] (2,1) -- (2,-3);
\draw[black] (1,1) -- (1,0);
\draw[black] (0,1) -- (0,0);

\end{tikzpicture}} \quad \quad
       \reflectbox{\begin{tikzpicture}[scale=0.5]
 \draw[help lines, very thick, white] (0,1) -- (0,-4) -- (5,-4)--(5,1)--(0,1);
\draw[red, dashed] (0,0.5) -- (5,0.5);
\draw[red, dashed] (4.5,1) -- (4.5,-4);
\draw (0.5,0.5) node [cross=3.5pt,red,thick] {};
\draw (1.5,0.5) node [cross=3.5pt,red,thick] {};
\draw (2.5,0.5) node [cross=3.5pt,red,thick] {};
\draw (3.5,0.5) node [cross=3.5pt,red,thick] {};
\draw (4.5,0.5) node [cross=3.5pt,red,thick] {};
\draw (4.5,-0.5) node [cross=3.5pt,red,thick] {};
\draw (4.5,-1.5) node [cross=3.5pt,red,thick] {};
\draw (4.5,-2.5) node [cross=3.5pt,red,thick] {};
\draw (4.5,-3.5) node [cross=3.5pt,red,thick] {};

\draw[fill=black] (0.5,0.5) circle (0.1cm);
\draw[fill=black] (1.5,0.5) circle (0.1cm);
\draw[fill=black] (2.5,0.5) circle (0.1cm);
\draw[fill=black] (3.5,0.5) circle (0.1cm);
\draw[fill=black] (4.5,0.5) circle (0.1cm);
\draw[fill=black] (2.5,-0.5) circle (0.1cm);
\draw[fill=black] (3.5,-0.5) circle (0.1cm);
\draw[fill=black] (4.5,-0.5) circle (0.1cm);
\draw[fill=black] (2.5,-1.5) circle (0.1cm);
\draw[fill=black] (3.5,-1.5) circle (0.1cm);
\draw[fill=black] (4.5,-1.5) circle (0.1cm);
\draw[fill=black] (2.5,-2.5) circle (0.1cm);
\draw[fill=black] (3.5,-2.5) circle (0.1cm);
\draw[fill=black] (4.5,-2.5) circle (0.1cm);
\draw[fill=black] (4.5,-3.5) circle (0.1cm);

%%%%%%%%%%%%%%%
\draw[black] (0,0) -- (5,0);
\draw[black] (0,1) -- (5,1);
\draw[black] (2,-1) -- (5,-1);
\draw[black] (2,-2) -- (5,-2);
\draw[black] (2,-3) -- (5,-3);
\draw[black] (4,-4) -- (5,-4);
%%%%%%%%%%%
\draw[black] (5,1) -- (5,-4);
\draw[black] (4,1) -- (4,-4);
\draw[black] (3,1) -- (3,-3);
\draw[black] (2,1) -- (2,-3);
\draw[black] (1,1) -- (1,0);
\draw[black] (0,1) -- (0,0);
\end{tikzpicture} }\quad \quad
                   \reflectbox{\begin{tikzpicture}[scale=0.5]
 \draw[help lines, very thick, white] (0,1) -- (0,-4) -- (5,-4)--(5,1)--(0,1);
\draw[red, dashed] (3.5,1) -- (3.5,-4);
\draw[red, dashed] (4.5,1) -- (4.5,-4);
\draw (3.5,0.5) node [cross=3.5pt,red,thick] {};
\draw (3.5,-0.5) node [cross=3.5pt,red,thick] {};
\draw (3.5,-1.5) node [cross=3.5pt,red,thick] {};
\draw (3.5,-2.5) node [cross=3.5pt,red,thick] {};
\draw (4.5,0.5) node [cross=3.5pt,red,thick] {};
\draw (4.5,-0.5) node [cross=3.5pt,red,thick] {};
\draw (4.5,-1.5) node [cross=3.5pt,red,thick] {};
\draw (4.5,-2.5) node [cross=3.5pt,red,thick] {};
\draw (4.5,-3.5) node [cross=3.5pt,red,thick] {};
\draw[fill=black] (0.5,0.5) circle (0.1cm);
\draw[fill=black] (1.5,0.5) circle (0.1cm);
\draw[fill=black] (2.5,0.5) circle (0.1cm);
\draw[fill=black] (3.5,0.5) circle (0.1cm);
\draw[fill=black] (4.5,0.5) circle (0.1cm);
\draw[fill=black] (2.5,-0.5) circle (0.1cm);
\draw[fill=black] (3.5,-0.5) circle (0.1cm);
\draw[fill=black] (4.5,-0.5) circle (0.1cm);
\draw[fill=black] (2.5,-1.5) circle (0.1cm);
\draw[fill=black] (3.5,-1.5) circle (0.1cm);
\draw[fill=black] (4.5,-1.5) circle (0.1cm);
\draw[fill=black] (2.5,-2.5) circle (0.1cm);
\draw[fill=black] (3.5,-2.5) circle (0.1cm);
\draw[fill=black] (4.5,-2.5) circle (0.1cm);
\draw[fill=black] (4.5,-3.5) circle (0.1cm);

%%%%%%%%%%%%%%%
\draw[black] (0,0) -- (5,0);
\draw[black] (0,1) -- (5,1);
\draw[black] (2,-1) -- (5,-1);
\draw[black] (2,-2) -- (5,-2);
\draw[black] (2,-3) -- (5,-3);
\draw[black] (4,-4) -- (5,-4);
%%%%%%%%%%%
\draw[black] (5,1) -- (5,-4);
\draw[black] (4,1) -- (4,-4);
\draw[black] (3,1) -- (3,-3);
\draw[black] (2,1) -- (2,-3);
\draw[black] (1,1) -- (1,0);
\draw[black] (0,1) -- (0,0);
\end{tikzpicture}}

    \end{example}

The idea behind  Theorem \ref{thm:singES_bound} is pretty standard, and it relies on the usual arguments for Singleton-like bounds in coding theory.  However, at the same time, Etzion and Silberstein also conjectured this bound to be tight over \emph{every} finite field $\F$. 

\begin{definition}
    A $[\mD,k,d]_{\F}$ code is called a \textbf{maximum Ferrers diagram (MFD) code} if $k=\nu_{\min}(\mD,d)$.
\end{definition}

When $\mD=[m]\times [n]$, it is easy to verify that $\nu_{\min}([m]\times[n],d)=\max(m,n)(\min(m,n)-d+1)$. In this case, an $[m\times n,\max(m,n)(\min(m,n)-d+1),d]_{\F}$ MFD code is also called \textbf{maximum rank distance (MRD)}. This terminology comes from the theory of classical rank-metric codes \cite{gabidulin1985theory}.

\begin{conjecture}[Etzion-Silberstein Conjecture {\cite[Conjecture 1]{etzion2009error}}]\label{con: ES conjecture}
    For every Ferrers diagram $\mD$, every positive integer $d$ and every finite field $\F$, there exists a $[\mD,\nu_{\min}(\mD,d),d]_{\F}$ MFD code.
\end{conjecture}

Of course, the Etzion-Silberstein conjecture can be restricted to some families of pairs $(\mD,d)$ and to some finite fields $\F$. In that case, we say that the \textbf{Etzion-Silberstein conjecture holds for the pair $(\mD,d)$ over $\F$} if there exists a $[\mD,\nu_{\min}(\mD,d),d]_{\F}$ MFD code. Similarly, we will say that  the \textbf{Etzion-Silberstein conjecture holds for the pair $(\mD,d)$} if there exists a $[\mD,\nu_{\min}(\mD,d),d]_{\F}$ MFD code over every finite field $\F$. Last, we will say that the \textbf{Etzion-Silberstein conjecture holds for the diagram $\mD$} if there exists a $[\mD,\nu_{\min}(\mD,d),d]_{\F}$ MFD code over every finite field $\F$ and for every value $d$.

\medskip

Since 2009, the Etzion-Silberstein conjecture has been proved to hold for some cases (see \cite{etzion2016optimal,antrobus2019maximal,gorla2017subspace,liu2019constructions,liu2019several,neri2023proof,rakhi2023}), but it is in general widely open, and no counterexample has been derived so far. It should be noted that most of the systematic constructions of MFD codes -- that is, those that work for every finite field and all the minimum distances for the given family of Ferrers diagram -- rely on the general construction of classical MRD codes over the rectangular diagrams $[m] \times [n]$ in \cite{delsarte1978bilinear,gabidulin1985theory}, and more recently, on the construction of MFD codes over the triangular diagrams $\mT_n=(n,n-1,\ldots,2,1)$ in \cite{neri2023proof}. 
In fact, starting from these general constructions, many more instances of the Etzion-Silberstein conjecture have been proved using a \emph{reduction argument} based on the upcoming Proposition \ref{prop: reduction}. We will explain this method in detail and analyze it combinatorially in the rest of the paper.

\section{Irreducible Ferrers diagram pairs}\label{sec: irred pairs}

The aim of this section is to formalize the notion of reducibility of Ferrers diagram pairs and to provide a theoretical study on irreducible Ferrers diagrams, using a graph theoretical approach.

\subsection{Reduction of Ferrers diagram pairs}

The following easy result can be considered as the stepping stone for this section, and it is directly obtained from \cite{antrobus2019maximal}.

\begin{proposition}\label{prop: reduction}
	Let $\mD$ be a Ferrers diagram of order $n$ with $|\mD|\ge 2$ and let $d$ be a positive integer with $1\le d \le n$. Let $P \in \mR(\mD)$ and $\mD' :=\mD\setminus\{P\}$. Then
	\[
	 \nu_{\min}(\mD',d) \in \{\nu_{\min}(\mD,d),\, \nu_{\min}(\mD,d)-1 \}.
	\]
	In particular:
	\begin{enumerate}[label = (\arabic*)]
		\item If $\nu_{\min}(\mD',d)=\nu_{\min}(\mD,d)$, then every $[\mD',\nu_{\min}(\mD',d),d]_\F$ code is also a $[\mD,\nu_{\min}(\mD,d),d]_\F$ code.
		\item If $\nu_{\min}(\mD',d)=\nu_{\min}(\mD,d)-1$, then for every $[\mD,\nu_{\min}(\mD,d),d]_\F$ code $\mC$, $\mC\cap \F^{\mD'}$ is a $[\mD',\nu_{\min}(\mD',d),d]_\F$ code.
	\end{enumerate}
\end{proposition}

Observe that Proposition \ref{prop: reduction} gives a reduction relation between two Ferrers diagrams which are obtained one from the other by removing just an element. 

\begin{definition}
    Let $\mD$ be a Ferrers diagram of order $n$ and let $d$ be a positive integer. Let $P \in \mR(\mD)$ and $\mD' :=\mD\setminus\{P\}$. We say that $(\mD,d)$ \textbf{is reduced from} $(\mD',d)$, and denote by $\mD'\stackrel{d}{\longrightarrow} \mD$, if $\nu_{\min}(\mD',d)=\nu_{\min}(\mD,d)$. Otherwise we say that $(\mD',d)$ \textbf{is reduced from} $(\mD,d)$, and denote it by $\mD\stackrel{d}{\longrightarrow} \mD'$. Similarly, we say that $(\mD_\ell,d)$ is \textbf{reachable from} $(\mD_0,d)$ if there exist diagrams $\mD_1,\ldots,\mD_{\ell-1}$ such that $\mD_0\stackrel{d}{\longrightarrow} \mD_1 \stackrel{d}{\longrightarrow} \mD_2 \stackrel{d}{\longrightarrow} \cdots \stackrel{d}{\longrightarrow} \mD_{\ell-1} \stackrel{d}{\longrightarrow} \mD_\ell$. Note that reachability is a transitive relation.
\end{definition}

The reduction and reachability relations defined above formalize the concept of constructing MFD codes of a given minimum rank on a Ferrers diagram from an MFD code on another Ferrers diagram. Indeed, according to Proposition \ref{prop: reduction},  if
$$
\mD\stackrel{d}{\longrightarrow} \cdots \stackrel{d}{\longrightarrow} \mD'
$$
then one can construct a $[\mD',\nu_{\min}(\mD',d),d]_{\F}$ code from any $[\mD,\nu_{\min}(\mD,d),d]_{\F}$ code.

\begin{observation}\label{obs: one reducible from other}
Let $\mD_1, \mD_2$ be two Ferrers diagrams satisfying $|\mD_1\triangle\mD_2|=1$. Then
\begin{enumerate}[label = (\arabic*)]
    \item $\mD_1\stackrel{d}{\longrightarrow} \mD_2  \ \text{ or } \ \mD_2\stackrel{d}{\longrightarrow} \mD_1$.
    \item  $\mD_1\stackrel{d}{\longrightarrow} \mD_2$ holds if and only if $\mD_1^\top\stackrel{d}{\longrightarrow} \mD_2^\top$. This follows directly from noting that $\nu_j(\mD,d) = \nu_{d-1-j}(\mD^\top,d)$ for all $0 \leq j \leq d-1$.
\end{enumerate}
\end{observation}

\begin{definition}\label{def:reducible}
    Let $\mD$ be a Ferrers diagram of any order and let $d$ be a positive integer. We say that $(\mD,d)$ is \textbf{reducible} if there exists a Ferrers diagram $\mD'$ (of any order) such that $\mD'\stackrel{d}{\longrightarrow} \mD$. Moreover, we say that $(\mD,d)$ is \textbf{irreducible} if it is not reducible.
\end{definition}

The following result is straightforward, and it comes directly from the definition of irreducible pairs. We state it properly because we will make a wide use of it throughout the paper as it provides a more usable characterization of irreducibility.

\begin{proposition}\label{prop:characterization_irreducible}
Let $\mD$ be a Ferrers diagram  and let $d$ be a positive integer. The pair $(\mD,d)$ is irreducible if and only if it satisfies these two conditions:
\begin{enumerate}[label = (\arabic*)]
    \item For every addible point $P \in \mA(\mD)$ we have $\nu_{\min}(\mD \cup \{P\},d) = \nu_{\min}(\mD,d)$;
    \item For every removable point $P \in \mR(\mD)$ we have $\nu_{\min}(\mD \setminus \{P\},d) = \nu_{\min}(\mD,d) -1$.
\end{enumerate}
\end{proposition}

For the rest of the paper, we introduce the following notation. We use $\mathfrak D$ to denote the set of all Ferrers diagrams, and $\mathfrak D_n$ to the restriction of $\mathfrak D$ to the set of Ferrers diagrams of order $n$.

\subsection{The Young digraph}

We can graphically represent the reduction relations using the Hasse diagram of Young's lattice on integer partitions. 

\begin{definition}
     The \textbf{Young's lattice} is the poset $(\mathfrak D,\subseteq)$, where $\mathfrak D$ is the set of all the Ferrers diagrams, and the partial order is  the standard containment $\mD\subseteq \mD'$.
\end{definition}

The Hasse diagram of the Young's lattice $(\mathfrak D,\subseteq)$ is a representation as a graph, in which the set of vertices is $\mathfrak D$, and there is an edge between $\mD$ and $\mD'$ if and only if $|\mD\triangle\mD'|=1$. Thus, this is compatible with the reducibility notion given in Definition \ref{def:reducible}, provided that we give a direction to each edge. This motivates the following definition.

\begin{definition}
    Let $d$ be a positive integer. The \textbf{$d$-Young (reducibility) digraph} is a directed graph $\mathfrak G_d=(V,E)$ such that $V=\mathfrak D$, and $(\mD,\mD')\in E$ if and only if $\mD \stackrel{d}{\longrightarrow}\mD'$. The \textbf{$d$-Young (reducibility) digraph of order $n$} is the subgraph of $\mathfrak G_d$ induced from the set of vertices $\mathfrak D_n$, and it is denoted by $\mathfrak G_{d,n}$.
\end{definition}

\begin{example}\label{exa:n=d=3}
  Let us fix $n=d=3$. Below, we have a graphical representation of the $3$-Young digraph of order $3$. In this representation, we use square boxes instead of dots. Note that this digraph possesses a natural mirror symmetry, due to the involution $\mD \mapsto \mD^\top$, which preserves the reducibility relation, due to Observation \ref{obs: one reducible from other}(2).

 \bigskip 

\noindent\begin{tabular}{p{\textwidth}}
\adjustbox{scale=0.35,center}{
\input{TikZ_pictures/BulletsDigraph_n=3d=3}
} \\ 
\\
\hline 
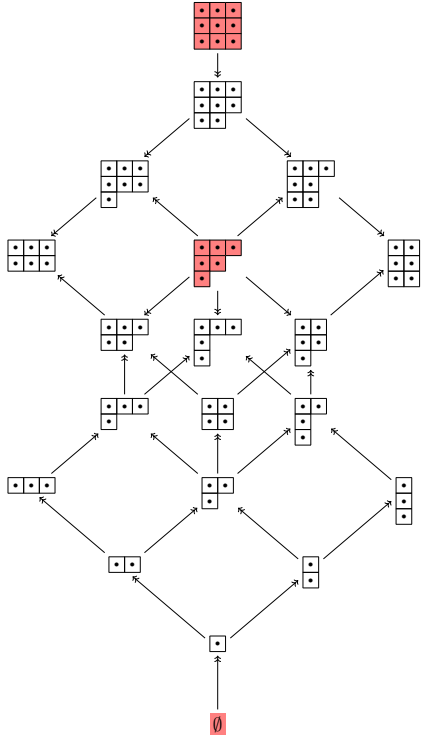
 \footnotesize{The irreducibility digraph for $n=d=3$. The coloured diagrams are irreducible.} \\
\hline
\end{tabular}

\end{example}

\begin{example}\label{exa:n=4}
    For $n=4$, we provide below a graphical illustration of the $4$-Young digraph and of the $4$-Young digraph of order $4$. Clearly, the vertices of the two digraphs are the same, but the orientation of the edges are sometimes different. Note that in this case the graphical representation does not reflect the mirror symmetry, due to the two central Ferrers diagrams of size $8$ which are self-adjoint. %\todo[]{reformulate?}

 \bigskip 

\adjustbox{scale=0.8,center}{\begin{tabular}{c|c}
\input{TikZ_pictures/BulletsDigraph_n=4_d=3} 
 & 
\hspace*{-2cm}\input{TikZ_pictures/BulletsDigraph_n=4_d=4}

\\ 
\\
\hline 
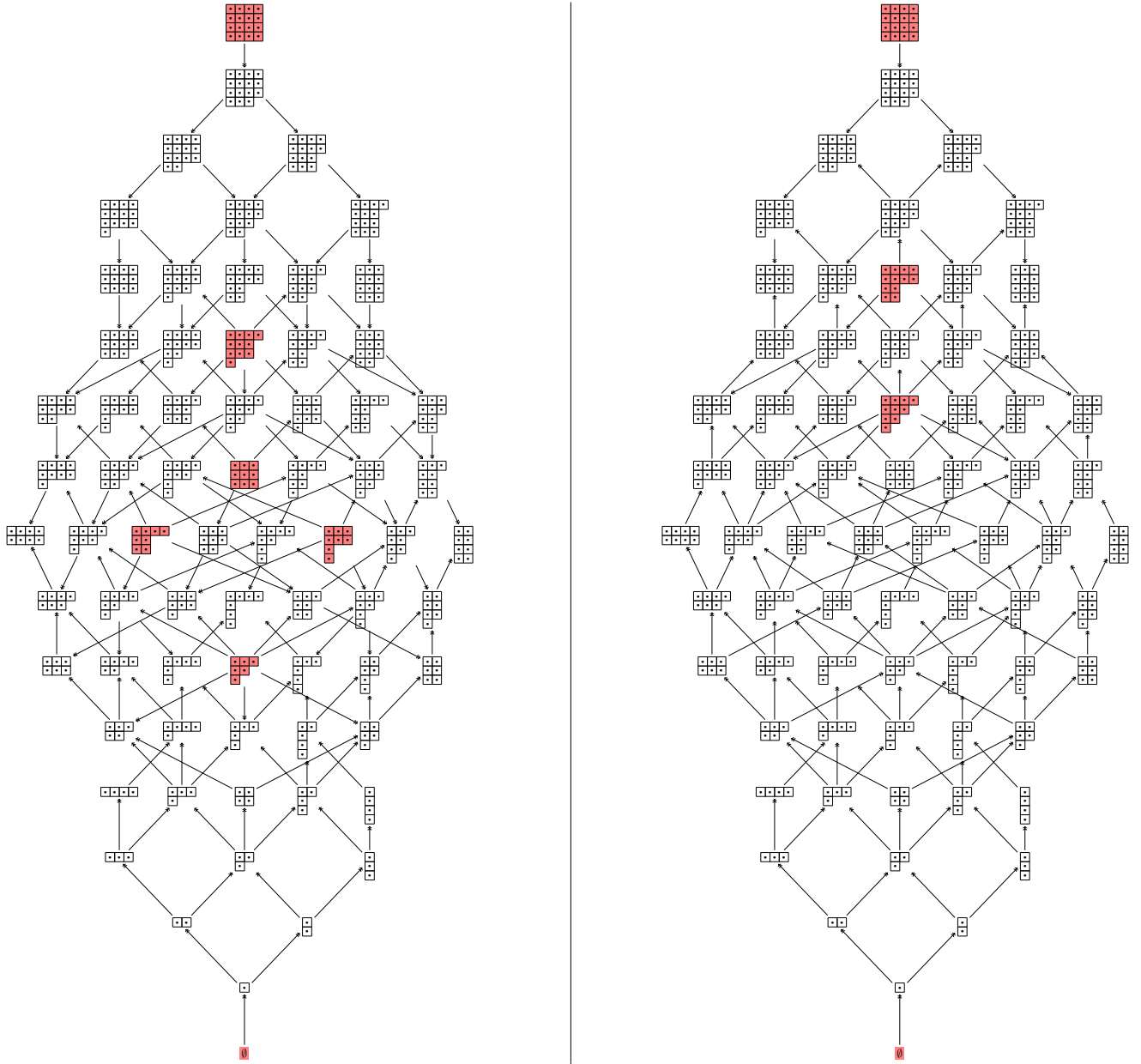
 \multicolumn{2}{c}{\footnotesize{The irreducibility digraphs for $\mathfrak G_{3,4}$ and $\mathfrak G_{4,4}$. The coloured diagrams are irreducible.}} \\
\hline
\end{tabular}}

\end{example}

\begin{proposition}\label{prop: digraph no cycles}
For every $d \in \N$, the $d$-Young reducibility digraph $\mathfrak G_d = (V,E)$ is a directed acyclic graph (DAG), i.e. it does not contain directed cycles. As an immediate consequence, the subgraphs $\mathfrak G_{d,n}$ are DAG's.
\end{proposition}

\begin{proof}
 First note that if $(\mD,\mD') \in E$ is any directed edge, then $\card{D'} = \card{D} \pm 1$ and
 \[
\nu_{\min}(\mD',d) = \begin{cases}
\nu_{\min}(\mD,d)
  & \text{ if } \card{D'} = \card{D}+1; \\
\nu_{\min}(\mD,d)-1
    & \text{ if } \card{D'} = \card{D}-1.
\end{cases}
 \]
Now suppose there exists a directed cycle in $\mathfrak G_d$ with  vertex sequence $(\mD_0, \mD_1, \mD_2, \ldots, \mD_{\ell-1}, \mD_\ell = \mD_0)$.  Since $\nu_{\min}(\mD_i,d) \geq \nu_{\min}(\mD_{i+1},d)$ for all $0 \leq i < \ell$ but also $\nu_{\min}(\mD_0,d) = \nu_{\min}(\mD_\ell,d)$, we necessarily get $\nu_{\min}(\mD_0,d) =  \nu_{\min}(\mD_1,d) = \ldots = \nu_{\min}(\mD_{\ell-1},d) = \nu_{\min}(\mD_\ell,d)$. But then it follows that $\card{\mD_\ell} = \card{\mD_{\ell-1}} + 1 = \ldots = \card{\mD_0}+\ell$, a contradiction.
\end{proof}

\begin{definition}\label{def:n-reducible}
    Let $\mD$ be a Ferrers diagram of order $n$ and let $d$ be a positive integer. We say that $(\mD,d)$ is \textbf{$\bm{n}$-reducible} if there exists a Ferrers diagram $\mD'$ of order $n$ such that $\mD'\stackrel{d}{\longrightarrow} \mD$. Moreover, we say that $(\mD,d)$ is \textbf{$\bm{n}$-irreducible} if $\mD$ is of order $n$ and $(\mD,d)$ is not $n$-reducible.
\end{definition}

The irreducible pairs are exactly the \textit{sources} in the entire digraph $\mathfrak G_d$ (i.e. the vertices without incoming edges). Likewise, the $n$-irreducible pairs are the sources in the sub-digraph $\mathfrak G_{d,n}$. 
It is immediate that any irreducible pair $(\mD,d)$ with $\mD$ of order $n$ is also $n$-irreducible. Interestingly, the converse also holds, thereby showing that these definitions are equivalent.

\begin{proposition}\label{prop: irred and n-irred equivalent}
Let $\mD$ be a Ferrers diagram of order $n$ and let $d$ be a positive integer. Then $(\mD,d)$ is irreducible if and only if it is $n$-irreducible.    
\end{proposition}

\begin{proof}
%\Ale{This proofs must be adapted to top-left alignment}
We will prove this by contradiction: suppose a pair $(\mD,d)$ (with $\mD$ of order $n$) is $n$-irreducible but not irreducible. 
Then the reducibility of  $(\mD,d)$ implies $\mD'\stackrel{d}{\longrightarrow} \mD$ for some diagram $\mD'$ of order $n+1$. Without loss of generality, we can assume $\mD' = \mD \cup \{(n+1,1)\}$ and $c_{1}(\mD) = n$ (the only other possibility would be the transposed case $\mD' = \mD \cup \{(1,n+1)\}$).

Since $\mD'\setminus \mD = \{(n+1,1)\}$, we have $\nu_0(\mD',d) = \nu_0(\mD,d) + 1$ and $\nu_j(\mD',d) = \nu_j(\mD,d)$ for every $j \geq 1$. As $\mD'\stackrel{d}{\longrightarrow} \mD$ together with $\card{\mD'} = \card{\mD}+1$ means that $\nu_{\min}(\mD',d) = \nu_{\min}(\mD,d)+1$, we deduce that $\nu_j(\mD,d) > \nu_0(\mD,d)$ for every $j \geq 1$.

Next we make the following two observations:
\begin{enumerate}[label = (\arabic*)]
    \item For every $y \in \{2,\ldots,n\}$, $c_y(\mD) \notin \{d-1,d,\ldots,n-1\}$. Otherwise it would be possible to add a point $P \in \{d,\ldots,n\} \times \{2,\ldots,n\}$ such that $\mD'':= \mD \cup \{P\}$ is a Ferrers diagram of order $n$ and $\mD''\stackrel{d}{\longrightarrow} \mD$ (since $\nu_0(\mD'',d) = \nu_0(\mD,d)+1 \leq \nu_j(\mD,d) \leq \nu_j(\mD'',d)$ for $j \geq 1$), contradicting $n$-irreducibility. 
    \item For every $y \in \{2,\ldots,n\}$, $c_y(\mD) \notin \{1,2,\ldots,d-1\}$. Otherwise it would be possible to remove a point $Q \in (\{1,\ldots,d-1\} \times \{2,\ldots,n\}) \cap \mD$ such that $\mD''':= \mD \setminus \{Q\}$ is a Ferrers diagram of order $n$  and $\mD'''\stackrel{d}{\longrightarrow} \mD$ (since $\nu_0(\mD''',d) = \nu_0(\mD,d) \leq \nu_j(\mD,d)-1 \leq \nu_j(\mD''',d)$ for $j \geq 1$), contradicting $n$-irreducibility. 
\end{enumerate}

Hence we have that for every $y \in \{1,\ldots,n+1\}$, $c_y(\mD) = 0$ or $c_y(\mD) = n$. The problem is now that for such a Ferrers diagram $\mD$ it is easy to see that $\nu_{\min}(\mD,d) = \nu_{d-1}(\mD,d)$, a contradiction.
\end{proof}

\begin{corollary}\label{cor: path from irred to red}
Let $\mD$ be a Ferrers diagram of order $n$ and let $d$ be any positive integer with $d\le n$. If $(\mD,d)$ is reducible, then in the $d$-Young reducibility digraph $\mathfrak G_{d,n}$ of order $n$   there is a directed path from some irreducible $(\mD',d)$ to $(\mD,d)$.
\end{corollary}

\begin{proof}
First, because of Proposition \ref{prop: irred and n-irred equivalent}, $(\mD,d)$ is $n$-reducible with $n$ the order of $\mD$. As $\mathfrak G_{d,n}$ is a DAG and moreover finite, there must be a directed path in $\mathfrak G_{d,n}$ from some $n$-irreducible $(\mD',d)$ to $(\mD,d)$ (if we look at all directed paths ending in $(\mD,d)$, these paths cannot be arbitrarily long because of finiteness; a path of maximum length has to start at an $n$-irreducible because of acyclicness).  
\end{proof}

From the notion of irreducibility and from Corollary \ref{cor: path from irred to red} we obtain a statement equivalent to the Etzion-Silberstein conjecture (Conjecture \ref{con: ES conjecture}).

\begin{theorem}\label{thm: es conjecture true iff irred}
    The Etzion-Silberstein conjecture holds true if and only if it holds true for every irreducible pair $(\mD,d)$ and finite field $\F$. 
\end{theorem}

\begin{remark}
    In \cite{neri2023proof} the Etzion-Silberstein conjecture was proved to hold for every MDS-constructible pair over any finite field. The main idea to achieve this proof was to prove the conjecture for all triangular Ferrers diagrams, that is, diagrams of the form $\mT_{n}:=(n,n-1,\ldots,2,1)$, and then use the reduction argument formalized here. Indeed, in \cite[Theorem 4.22]{neri2023proof} it was shown that every MDS-constructible pair can be reduced from a triangular diagram.
\end{remark}

It is now natural to ask how one could determine the sources of the digraph $\mathfrak G_d$, for a given $d$. First, one could restrict to the finite subgraphs $\mathfrak G_{d,n}$, and then run an algorithm for finding the sources of $\mathfrak G_{d,n}$. The computational cost of this algorithm will depend on the number of vertices and edges of $\mathfrak G_{d,n}$, which we exhibit in the following combinatorial result.

\begin{proposition}\label{prop:vertices_edges_Gdn}
    The number of vertices of $\mathfrak G_{d,n}$ is equal to $N=\binom{2n}{n}$, while the number of edges is $M=n\binom{2n-1}{n}=\frac{n}{2}\binom{2n}{n}$.
\end{proposition}

\begin{proof}
    The number of Ferrers diagram of order $n$ coincides with the number of right-top paths from $(0,0)$ to $(n,n)$, which is equal to $N=\binom{2n}{n}$. For the number of edges, one needs to count for how many Ferrers diagrams of order $n$ the point $(i+1,j+1)$ is removable, and then sum over all the $(i,j) \in \{0,\ldots,n-1\}^2$. This number coincides with
    $$M=\sum_{i=0}^{n-1}\sum_{j=0}^{n-1}\binom{i+j}{i}\binom{2n-i-j-2}{n-i-1}.$$
    Let us call $s=i+j$, and rewrite the sum as
    $$M=\sum_{i=0}^{n-1}\sum_{s=i}^{n-1+i}\binom{s}{i}\binom{2n-2-s}{n-i-1}=\sum_{i=0}^{n'}\sum_{s\ge0}\binom{s}{i}\binom{2n'-s}{n'-1},$$
    where we have substituted $n'=n-1$ and extended the inner sum to all $s\ge 0$, since outside $\{i,\ldots,n'-i\}$ the quantity in the sum is $0$. By the Chu-Vandermonde identity, the inner sum is equal to $\binom{2n'+1}{n'+1}=\binom{2n-1}{n}$ and thus the whole sum is
    $$M=\sum_{i=0}^{n-1}\binom{2n-1}{n}=n\binom{2n-1}{n},$$
    completing the proof.
\end{proof}

\begin{remark} Algorithmically, determining all the sources of a digraph requires $\mathcal{O}(M+N)$ operations, where $N$ is the number of vertices and $M$ the number of edges. In this case, if we restrict ourselves to study the finite DAG $\mathfrak G_{d,n}$, by Proposition \ref{prop:vertices_edges_Gdn} we have $N=\binom{2n}{n}$ and $M=\frac{n}{2}\binom{2n}{n}$. Moreover, to construct the graph $\mathfrak G_{d,n}$, one has to  compute $\nu_{\min}(\mD,d)$, for every Ferrers diagram $\mD$ of order $n$, which requires $d$ operations. The total cost can be estimated, asymptotically in $n$ as
$$M=\frac{\sqrt{n}4^n}{2\sqrt{\pi}}(1+\mathcal{O}(n^{-1})).$$ However, in the next section, we will provide a combinatorial characterization of irreducible pairs (see Theorem \ref{thm:main}), which allows us to skip the algorithmic computation of sources in $\mathfrak G_{d,n}$. In particular, our combinatorial characterization does not depend at all on $n$.
\end{remark}

\begin{remark}
    Although not of interest for  our scope, it is worth to mention that one can easily characterize the sinks of the $d$-Young digraphs of order $n$. Indeed, the reader can verify that they consists exactly of the Ferrers diagrams of the form
     $\mathcal L_{n,d,j}=(c_1,\ldots, c_n)$,  for $0\le j <d$, given by 
    $$c_i =\begin{cases}
        n & \mbox{ if } 1\le i \le d-j-1, \\
        j & \mbox{ if } d-j \le i \le n.
    \end{cases}$$
\end{remark}

\section{Classifying irreducible pairs}\label{sec:classification}

\subsection{A classification theorem for large irreducible Ferrers diagrams}

By the nature of the $d$-Young digraph, we have that a Ferrers diagram $\mD$ is a source of the $d$-Young digraph $\mathfrak G_d$ if and only if $(\mD,d)$ is irreducible. The aim of this section to characterize the irreducible pairs $(\mD,d)$. 

 We start with three easy results, which are helpful for the reader to get familiar with the notion of irreducible pairs.

\begin{lemma}\label{lem:adjoint_irred}
    Let $\mD$ be a Ferrers diagram and let $d \in \N$. Then, $(\mD,d)$ is irreducible if and only if $(\mD^\top,d)$ is irreducible.
\end{lemma}

\begin{proof}
    It directly follows from Observation \ref{obs: one reducible from other}(2).
\end{proof}

\begin{lemma}\label{lem:irreducible_empty}
    Let $(\mD,d)$ be an irreducible pair with $\nu_{\min}(\mD,d)=0$. Then $\mD=\emptyset$.
\end{lemma}

\begin{proof}
    If $(\mD,d)$ is an irreducible pair with $\nu_{\min}(\mD,d)=0$, then $\emptyset \subset \mD$ and $\nu_{\min}(\mD,d)=0=\nu_{\min}(\emptyset,d)$. Thus, we can find a chain of Ferrers diagrams $\emptyset =\mD_0 \subset \mD_1\subset \ldots \subset \mD_{s-1}\subset \mD_s= \mD$, where $|\mD_i|=i$ and $|\mD_i\setminus \mD_{i-1}|=1$ for every $i \in [s]$. Moreover, $\nu_{\min}(\mD_i,d)=0$, and hence
    $$ \emptyset=\mD_0\stackrel{d}{\longrightarrow} \mD_1 \stackrel{d}{\longrightarrow} \mD_2 \stackrel{d}{\longrightarrow} \cdots \stackrel{d}{\longrightarrow} \mD_{s-1} \stackrel{d}{\longrightarrow} \mD_s=\mD.$$
\end{proof}

\begin{lemma}\label{lem:rectangle_irred}
    Let $a,b,d \in \N$. Then, $([a]\times [b],d)$ is irreducible if and only if $a=b\ge d$.
\end{lemma}
\begin{proof}
    Let $\mD=[a]\times[b]$. We have $\mA(\mD)=\{(1,b+1),(a+1,1)\}$ and $\mR(\mD)=\{(a,b)\}$. 
        It is easy to see that if $a=b\ge d$, then $(\mD,d)$ is irreducible, since \begin{align*}
        \nu_{\min}(\mD\setminus\{(a,a)\},d)&=a(a-d+1)-1=\nu_{\min}(\mD,d)-1, \\
        \nu_{\min}(\mD\cup\{(1,a+1)\},d)&=\nu_0(\mD\cup\{(1,a+1)\},d)=a(a-d+1)=\nu_{\min}(\mD,d), \\
        \nu_{\min}(\mD\cup\{(a+1,1)\},d)&=\nu_{d-1}(\mD\cup\{(a+1,1)\},d)=a(a-d+1)=\nu_{\min}(\mD,d)
    \end{align*} 

    On the other hand, if $a<d$, then $\nu_0(\mD,d)=0=\nu_{\min}(\mD,d)$, and hence $(\mD,d)$ cannot be irreducible. The same holds if $b<d$. Hence, assume $a,b\ge d$ but $a\neq b$, and without loss of generality suppose that $a>b$.  In this case one has
    $\nu_{\min}(\mD,d)=\nu_{d-1}(\mD,d)=a(b-d+1)$. However, it also holds that $\nu_{d-1}(\mD\cup\{(a+1,1)\},d)=a(b-d+1)+1=\nu_{\min}(\mD,d)+1$, and thus $(\mD,d)$ is not irreducible, showing the claim. 
\end{proof}

Before proceeding with the main results, we first need to fix some notation, which we will need in the following, as done in \cite[Section 4.3]{neri2023proof}.
% \begin{align*}
% \nu_j(\mD,d) & = \sum_{i=1}^{n-j}\max\{0,c_i-d+1+j\}, 
% \end{align*}
Let us consider the following subsets of $[n]^2$: 
\begin{align*}
\mS_{n,d,j}& =\{d-j,\ldots, n\}\times\{j+1,\ldots,n\} , \\
 \mB_{n,d-1}&=\{d-1,\ldots,n\}\times \{d-1,\ldots,n\}.
\end{align*}
\begin{observation}\label{obs: nu is D cap S}
With the notation above, observe that, for any Ferrers diagram $\mD$ of order $n$, one has
$$\nu_j(\mD,d)=| \mD\cap \mS_{n,d,j}|.$$  
\end{observation}

{\begin{example}[See {\cite[Example 4.18]{neri2023proof}}]\label{ex:STL}
    For $n=8$ and $d=4$, the orange areas in the first figure below represents the sets $\mS_{8,4,j}$ for $j \in \{0,\ldots,3\}$.  
    The green area in the second figure represents instead the set $\mB_{8,3}$.
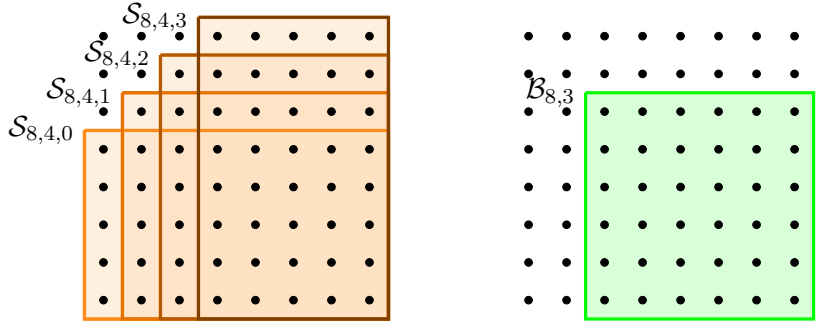
\begin{center}
    \begin{tikzpicture}[scale=0.5]

\draw[help lines, very thick, orange, fill=orange!25, fill opacity=0.5] (0,-2) -- (0,-7) -- (8,-7)--(8,-2)--(0,-2);
\draw[help lines, very thick, orange, fill=orange!25, fill opacity=0.5] (1,-1) -- (1,-7) -- (8,-7)--(8,-1)--(1,-1);
\draw[help lines, very thick, orange, fill=orange!25, fill opacity=0.5] (2,0) -- (2,-7) -- (8,-7)--(8,0)--(2,0);
\draw[help lines, very thick, orange, fill=orange!25, fill opacity=0.5] (3,1) -- (3,-7) -- (8,-7)--(8,1)--(3,1);

\draw[help lines,  very thick, draw=orange!90!white] (0,-2) -- (0,-7) -- (8,-7)--(8,-2)--(0,-2);
\draw[help lines,  very thick, draw=orange!90!black] (1,-1) -- (1,-7) -- (8,-7)--(8,-1)--(1,-1);
\draw[help lines,  very thick, draw=orange!70!black] (2,0) -- (2,-7) -- (8,-7)--(8,0)--(2,0);
\draw[help lines,  very thick, draw=orange!50!black] (3,1) -- (3,-7) -- (8,-7)--(8,1)--(3,1);

\draw[]  (0,-2) node[left] {$\mS_{8,4,0}$};
\draw[]  (1,-1) node[left] {$\mS_{8,4,1}$};
\draw[]  (2,0) node[left] {$\mS_{8,4,2}$};
\draw[]  (3,1) node[left] {$\mS_{8,4,3}$};

\foreach \x in {0,...,7}
    \foreach \y in {0,...,7} 
        \draw[black,fill=black] (\x+0.5,0.5-\y) circle (0.1cm);
\end{tikzpicture} \qquad \qquad 
    \begin{tikzpicture}[scale=0.5]
\draw[help lines, very thick, green, fill=green!15] (2,-1) -- (2,-7) -- (8,-7)--(8,-1)--(2,-1);
\draw[]  (2,-1) node[left] {$\mB_{8,3}$};

\foreach \x in {0,...,7}
    \foreach \y in {0,...,7} 
        \draw[black,fill=black] (\x+0.5,0.5-\y) circle (0.1cm);
\end{tikzpicture}
\end{center}
\end{example}
}
\begin{lemma}\label{lem: if }
    Let $\mD$ be a Ferrers diagram of order $n$, and assume that $(\mD,d)$ is irreducible. If $(i_1,\ell_1),(i_2,\ell_2) \in \mD\cap \mB_{n,d-1}$ with $i_1<i_2$, $\ell_2<\ell_1$. Then $(i_2,\ell_1)\in\mD$.
\end{lemma}

\begin{proof}
   Assume on the contrary that $(i_2,\ell_1)\notin \mD$. Then, there must be an addible point $P\in \mA(\mD)$ belonging to $\{i_1+1,i_1+2,\ldots,i_2\}\times \{\ell_2+1,\ell_2+2,\ldots,\ell_1\}$. Since $P\in \mB_{n,d}$, then for every $j\in\{0,\ldots,d-1\}$ we have
   $$\nu_j(\mD\cup\{P\},d)=\nu_j(\mD,d)+1,$$
   and hence $\nu_{\min}(\mD\cup\{P\},d)=\nu_{\min}(\mD,d)+1$,
   implying that $\mD\cup\{P\}\stackrel{d}{\longrightarrow} \mD$  and contradicting the hypothesis of irreducibility.
\end{proof}

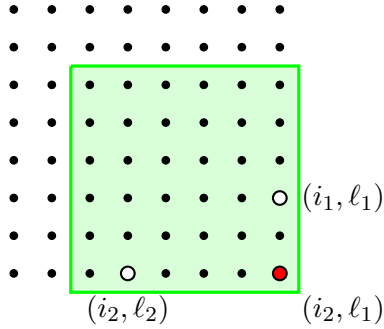
\begin{figure}[h]
\centering
    \begin{tikzpicture}[scale=0.5]
\draw[help lines, very thick, green, fill=green!15] (2,-1) -- (2,-7) -- (8,-7)--(8,-1)--(2,-1);
\foreach \x in {0,...,7}
    \foreach \y in {0,...,7} 
        \draw[black,fill=black] (\x+0.5,0.5-\y) circle (0.1cm);

\draw[black, thick, fill=white] (3.5,-6.5) circle (1.8mm); 
\draw[black, thick, fill=white] (7.5,-4.5) circle (1.8mm); 
\draw[black, thick, fill=red] (7.5,-6.5) circle (1.8mm); 

 \draw[] (3.5,-6.8) node[below] {$(i_2,\ell_2)$};
 \draw[] (7.8,-4.5) node[right] {$(i_1,\ell_1)$};
 \draw[] (7.8,-6.8) node[below right] {$(i_2,\ell_1)$};
\end{tikzpicture}
  \centering
    \caption{Visualization of Lemma \ref{lem: if }, with $\mB_{n,d-1}$ represented by the green area, the points $(i_1,\ell_1)$, $(i_2,\ell_2)$ in white and $(i_2,\ell_1)$ in red.}
 %   \label{fig:placeholder}
\end{figure}

As a consequence, if $(\mD,d)$ is irreducible and $\mD\cap \mB_{n,d-1} \neq \emptyset$, then \[
\mD\cap \mB_{n,d-1} = \{d-1,\ldots,a\} \times \{d-1,\ldots,b\}
\]
for some $a \geq d-1$ and $b \geq d-1$. This motivates us to define the  following form.

\begin{definition}
We say that a Ferrers diagram pair $(\mD,d)$, with $\mD$ of order $n$, is in $\boldsymbol{(a,b)}$\textbf{-standard form} with integers $a \geq d-1$ and $b \geq d-1$, if 
\[
\mD\cap \mB_{n,d-1} = \{d-1,\ldots,a\} \times \{d-1,\ldots,b\}.
\]
\begin{minipage}{0.6\textwidth} Note that if a Ferrers diagram pair with $\mD\cap \mB_{n,d-1} \neq \emptyset$ is in $(a,b)$-standard form, these $a$ and $b$ are clearly unique. For a pair $(\mD,d)$ in $(a,b)$-standard form we define the (transposed) subsets
\begin{itemize}
    \item $X = X(\mD,d) := (\mD \cap ([d-2] \times \{b+1,\ldots,n\}))^\top$;
	\item $Y = Y(\mD,d) := \mD \cap (\{a+1,\ldots,n\} \times [d-2])$.  
\end{itemize}
With these subsets defined, we have the decomposition $\mD = ([a] \times [b]) \sqcup X^\top \sqcup Y$.
\end{minipage}
\hfill \begin{minipage}{0.35\textwidth}
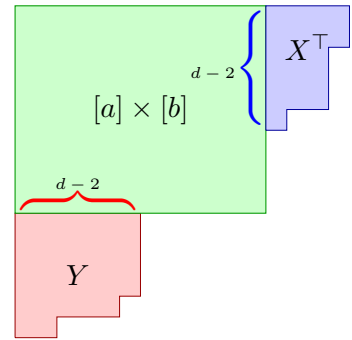
\begin{tikzpicture}[scale=0.55]
				
				%\draw[fill=white!80!red,draw=black!40!red] (0,0) rectangle ++(3,3);

                \draw[help lines, draw=black!40!red, fill=white!80!red] (0,0) -- (1,0) --(1,0.5) -- (2.5,0.5) -- (2.5,1) -- (3,1)--(3,3)--(0,3)--(0,0);
				\draw[fill=white!80!green,draw=black!40!green] (0,3) rectangle ++(6,5);
			
             \draw[help lines, draw=black!40!blue, fill=white!80!blue] (6,5) -- (6.5,5) --(6.5,5.5) -- (7.5,5.5) -- (7.5,7) -- (8,7)--(8,8)--(6,8)--(6,5);
				
				\draw (1.5,1.5) node[] {$Y$};
				\draw (3,5.5) node[] {$[a] \times [b]$};
				\draw (7,7) node[] {$X^\top$};
                \draw (5.1,6.5) node[] {\tiny{$d-2$}\Large{\textcolor{blue}{$\Bigg\{$}}};
                 \draw (1.5,3.7) node[] {\tiny{$d-2$}};
                \draw (1.5,3.3) node[] {\Large{\rotatebox[origin=c]{-90}{\textcolor{red}{$\Bigg\{$}}}};

				\end{tikzpicture} 
\end{minipage}
\end{definition}

\begin{remark}
Note that $X$ and $Y$ are shifts of Ferrers diagrams, i.e. a set of the form $\{v + p \ \mid \ p \in \mD\}$ for some $v \in \N^2$ and Ferrers diagram $\mD$. The definitions of addible and removable points are naturally extended to shifted Ferrers diagrams. Also, for every $i \in [d-1]$, we can define $c_i(X)$ (respectively, $c_i(Y)$) as for classical Ferrers diagrams, that is, $c_i(X)=|X\cap (\N \times \{i\})|$ (respectively, $c_i(Y)=|Y\cap (\N \times \{i\})|$), and we have $c_{d-1}(X) = c_{d-1}(Y) = 0$ by definition.
   
\end{remark}

\begin{figure}[h]
	\begin{center}
		\begin{tikzpicture}[scale=0.5]

\draw[fill=white!80!gray,draw=black!40!gray] (3.5,-3.5) rectangle ++(7,-7); 
\draw[fill=white!80!green,draw=black!40!green] (0.5,-0.5) rectangle ++(6,-5);    \draw[help lines, fill=white!80!blue,draw=black!40!blue] (6.5,-0.5) -- (10.5,-0.5) -- (10.5,-1.5) -- (9.5, -1.5) -- (9.5, -2.5) -- (7.5, -2.5) -- (7.5, -3.5) -- (6.5, -3.5) -- (6.5, -0.5);    \draw[fill=white!80!red,draw=black!40!red] (0.5,-5.5)--(3.5,-5.5)--(2.5,-5.5)--(2.5,-6.5)--(1.5,-6.5)--(1.5,-8.5)--(0.5,-8.5)--(0.5,-5.5);

\draw[draw=black!40!gray] (3.5,-3.5) rectangle ++(7,-7); 

\foreach \y in {1,...,8} 
    \draw[black,fill=black] (1,-\y) circle (0.1cm);
\foreach \y in {1,...,6} 
    \draw[black,fill=black] (2,-\y) circle (0.1cm);
\foreach \y in {1,...,5} 
    \draw[black,fill=black] (3,-\y) circle (0.1cm);
\foreach \y in {1,...,5} 
    \draw[black,fill=black] (4,-\y) circle (0.1cm);
\foreach \y in {1,...,5} 
    \draw[black,fill=black] (5,-\y) circle (0.1cm);
\foreach \y in {1,...,5} 
    \draw[black,fill=black] (6,-\y) circle (0.1cm);
\foreach \y in {1,...,3} 
    \draw[black,fill=black] (7,-\y) circle (0.1cm);
\foreach \y in {1,...,2} 
    \draw[black,fill=black] (8,-\y) circle (0.1cm);
\foreach \y in {1,...,2} 
    \draw[black,fill=black] (9,-\y) circle (0.1cm);
\draw[black,fill=black] (10,-1) circle (0.1cm);
\draw[] (2.3,-7.5) node[] {$Y$};
\draw[] (9.8,-2.9) node[] {$X^\top$};
\draw[] (3.5,-2.5) node[] {\scalebox{0.8}{$[a]\times[b]$}};
\draw[] (5,-4.55) node[] {\scalebox{0.8}{$\mD \cap \mB_{n,d-1}$}};
\draw[] (7,-7) node[] {\scalebox{0.8}{$\mB_{n,d-1}$}};

\end{tikzpicture}\text{}\qquad  \begin{tikzpicture}[scale=0.5]   \draw[fill=white!80!red,draw=black!40!red] (0.5,-5.5)--(3.5,-5.5)--(2.5,-5.5)--(2.5,-6.5)--(1.5,-6.5)--(1.5,-8.5)--(0.5,-8.5)--(0.5,-5.5);
\foreach \y in {6,...,8} 
    \draw[black,fill=black] (1,-\y) circle (0.1cm);
\foreach \y in {6,...,6} 
    \draw[black,fill=black] (2,-\y) circle (0.1cm);
    
\draw[] (2.3,-7.5) node[] {$Y$};

\end{tikzpicture}\text{}\qquad \begin{tikzpicture}[scale=0.5]   \draw[fill=white!80!blue,draw=black!40!blue] (6.5,-0.5) -- (9.5,-0.5) -- (9.5,-1.5) -- (8.5,-1.5) -- (8.5,-3.5) -- (7.5,-3.5) -- (7.5,-4.5) -- (6.5,-4.5) -- (6.5,-0.5);

\foreach \y in {1,...,4} 
    \draw[black,fill=black] (7,-\y) circle (0.1cm);
\foreach \y in {1,...,3} 
    \draw[black,fill=black] (8,-\y) circle (0.1cm);
\foreach \y in {1,...,1} 
    \draw[black,fill=black] (9,-\y) circle (0.1cm);

\draw[] (9.3,-3) node[] {$X$};
		\end{tikzpicture} 
        
	\end{center}
	\caption{On the left a representation of the Ferrers diagram pair $(\mD,d)$, with  $\mD = (8,6,5,5,5,5,3,2,2,1)$ of order $n=10$ and $d=5$. The pair is in $(a,b)$-standard form with $(a,b)=(5,6)$. The rectangle $[a] \times [b]$ represented in green, $\mB_{n,d-1}$ in gray, and diagrams $Y,X^\top$ in red and blue. $Y$ and $X$ are shifts of the two Ferrers diagrams on the right, each at most $d-2$ columns, with $(c_1(Y), c_2(Y), c_3(Y))=(3,1,0)$ and $(c_1(X), c_2(X), c_3(X)) = (4,3,1)$.}
\end{figure}
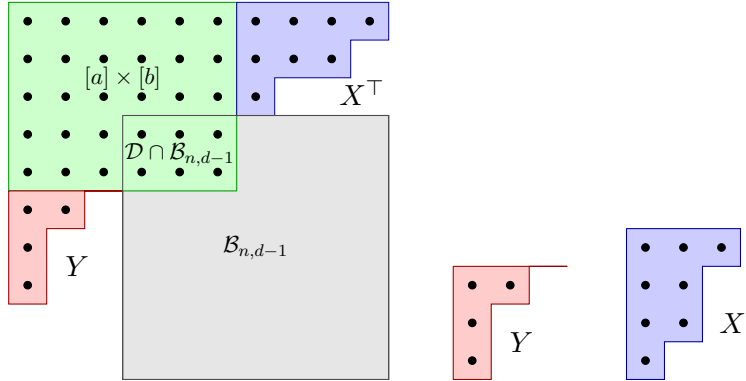

The first thing that we must observe is that we can identify addible and removable points of a Ferrers diagram in standard form, in terms of the addible and removable points of its $X$ and $Y$ components.

\begin{lemma}\label{lem: addible and removible standard form}
Let $(\mD,d)$ be a pair with $\mD\cap \mB_{n,d-1} \neq \emptyset$ in $(a,b)$-standard form. Then
\begin{enumerate}[label = (\arabic*)]
    \item $\mR(\mD)=\mR(X(\mD,d)^\top)\sqcup\mR(Y(\mD,d))\sqcup\{(a,b)\}$.
    \item $\mA(\mD)=\mA(X(\mD,d)^\top)\sqcup\mA(Y(\mD,d))$.
\end{enumerate}
\end{lemma}

\begin{proof}
Let $X = X(\mD,d)$ and $Y = Y(\mD,d)$.
  Statement (1) follows directly from the decomposition $\mD = X^\top \sqcup Y \sqcup ([a] \times [b]) $, noting that $\mR([a] \times [b]) = \{(a,b)\}$. We have $(a,b) \in \mR(\mD)$ because $a,b \geq d-1$ and $c_{d-1}(X) = c_{d-1}(Y) = 0$.

  Statement (2) is shown as follows: let $P \in \mA(\mD)$. First, since $a,b \geq d-1$ and $c_{d-1}(X) = c_{d-1}(Y) = 0$, we have $P \in \{a+1,\ldots,n+1\} \times [d-1]$ or $P \in [d-1] \times \{b+1,\ldots,n+1\}$, so $P \in \mA(Y)$ or $P \in \mA(X^\top)$. 
  Conversely, if $P \in \mA(Y)$, then $P$ is always an addible point of $\mD$. Note that this also holds in the edge case $Y = \emptyset$ and $P = (a+1,1)$, and in the edge case $c_{d-2}(Y)>0$ and $P=(a+1,d-1)$, because $b \geq d-1$. This holds analogously for $P \in \mA(X^\top)$. 
\end{proof}

We now present one of the main results of this paper, which classifies the irreducible Ferrers diagram pairs $(\mD,d)$ under the assumption that $\mD\cap \mB_{n,d-1}\neq \emptyset$, that is, when we have an $(a,b)$-standard form for $(\mD,d)$. This classification states that we only need to look at the values $\nu_{j}(\mD,d)$ in order to determine if $(\mD,d)$ is irreducible. Or, equivalently, we only need that $c_i(X)$ and $c_i(Y)$ satisfy certain inequalities.

\newpage

\begin{theorem}[Main classification theorem]\label{thm:main}
Let $(\mD,d)$ be a pair with $\mD\cap \mB_{n,d-1} \neq \emptyset$. Then the following are equivalent:
\begin{enumerate}[label = (\arabic*)]
\item  $(\mD,d)$ is irreducible.
\item  $(\mD,d)$ is in $(a,b)$-standard form and it satisfies 
\begin{itemize}
    \item $\nu_0(\mD,d) = \nu_{d-1}(\mD,d)$; %equivalent nu_0 = nu_{d-1} 
    % \textcolor{blue}{in general case, add $\geq 1$?}
    \item $\nu_j(\mD,d) \geq \nu_0(\mD,d) $ for all $j \in [d-2]$; 
    \item (Only when $a=b=d-1$): there exists a $j \in [d-2]$ for which  $\nu_j(\mD,d) = \nu_0(\mD,d)$.
\end{itemize}
\item  $(\mD,d)$ is in $(a,b)$-standard form with $X:=X(\mD,d)$ and $Y:=Y(\mD,d)$, and it satisfies
\begin{itemize}
    \item $\card{Y} - \card{X}= (b-a)(d-1)$; 
     \item  For all $j \in [d-2]$:
     \[
     j(b-a+d-1-j) +  \sum_{i = 1}^{j} c_{d-i} (X)  - \sum_{i = 1}^{j}c_i(Y) \geq 0;
     \]
     \item (Only when $a=b=d-1$): there exists a $j \in [d-2]$ for which equality holds above.
\end{itemize}
\end{enumerate}
\end{theorem}

\subsection{Proof of the main classification theorem}

For the proof of Theorem \ref{thm:main}, we start out by proving a few auxiliary results that hold for irreducible pairs $(\mD,d)$ with $\mD\cap \mB_{n,d-1} \neq \emptyset$ in $(a,b)$-standard form. In these proofs we often use the characterization of irreducible pairs from Proposition $\ref{prop:characterization_irreducible}$.

\begin{proposition}\label{prop:nui-nui-1}
Let $(\mD,d)$ be a pair with $\mD\cap \mB_{n,d-1} \neq \emptyset$ in $(a,b)$-standard form. Then for all $i \in [d-1]$:
\[
\nu_{i}(\mD,d)-\nu_{i-1}(\mD,d) = b-a + d - 2i + c_{d-i}(X)-c_{i}(Y).
\]
In particular, for all $j \in [d-1]$:
\[
\nu_{j}(\mD,d)-\nu_{0}(\mD,d) = j(b-a+d-1-j) +  \sum_{i = 1}^{j} c_{d-i} (X)  - \sum_{i = 1}^{j}c_i(Y).
\]
\end{proposition}
\begin{proof}
First, let us recall that $\nu_i(\mD,d)=| \mD\cap \mS_{n,d,i}|$ where $\mS_{n,d,i} =\{d-i,\ldots, n\}\times\{i+1,\ldots,n\}$ (Observation \ref{obs: nu is D cap S}). Then also $\nu_i(\mD,d)=| \mD\cap  (\mS_{n,d,i} \setminus  \mS_{n,d,i-1})| + | \mD\cap  (\mS_{n,d,i} \cap  \mS_{n,d,i-1})|$, and thus
\[
\nu_{i}(\mD,d)-\nu_{i-1}(\mD,d) = |\mD\cap (\mS_{n,d,i} \setminus  \mS_{n,d,i-1})| - |\mD\cap (\mS_{n,d,i-1} \setminus  \mS_{n,d,i})|.
\]
The terms on the right-hand side are equal to
\begin{align*}
|\mD\cap (\mS_{n,d,i} \setminus  \mS_{n,d,i-1})| &= |\mD\cap (\{d-i\} \times \{i+1,\ldots,n\})| = (b - i) + c_{d-i}(X),\\
|\mD\cap (\mS_{n,d,i-1} \setminus  \mS_{n,d,i})| &= |\mD\cap (\{d-i+1,\ldots,n\})\times\{i\})| = (a - d + i) + c_{i}(Y).
\end{align*}
And thus $\nu_{i}(\mD,d)-\nu_{i-1}(\mD,d) = b-a + d - 2i + c_{d-i}(X)-c_{i}(Y)$. The last equation in the Proposition follows directly from the sum $\sum_{i=1}^j (\nu_{i}(\mD,d)-\nu_{i-1}(\mD,d)) = \nu_{j}(\mD,d)-\nu_{0}(\mD,d)$.
%\todo[]{maybe re-use picture if not too much work}
\end{proof}

\begin{lemma}\label{lem:cx=0} Let $(\mD,d)$ be an irreducible pair with $\mD\cap \mB_{n,d-1} \neq \emptyset$ in $(a,b)$-standard form.
Let $j \in \{0,\ldots,d-2\}$. If for all $i \in \{0,\ldots,j-1\}$ the inequality $\nu_i(\mD,d) > \nu_j(\mD,d)$ holds, then $c_{d-2-j+1}(X) = \ldots = c_{d-2}(X) = 0$.   
\end{lemma}
		
\begin{proof}
Suppose not. Let $\ell \in \{d-2-j+1,\ldots,d-2\}$ the largest index such that $c_{\ell}(X) > 0$. Then the point $P = (\ell+1, a + 1)$  is an addible point for $\mD$.  Then for $\mD' = \mD \cup \{P\}$, we have 
\[
\nu_{k}(\mD',d) = \nu_{k}(\mD,d) + 1
\]
for all $k \geq d-2-\ell+1$. In particular, $\nu_{k}(\mD',d) = \nu_{k}(\mD,d) + 1$ for all $k \geq j$. By the assumption on $\nu_j(\mD,d)$, this implies that $\nu_{\min}(\mD',d) = \nu_{\min}(\mD,d) + 1$, contradicting the irreducibility hypothesis.
\end{proof}
		
\begin{lemma}\label{lem:cy=} Let $(\mD,d)$ be an irreducible pair with $\mD\cap \mB_{n,d-1} \neq \emptyset$ in $(a,b)$-standard form. Let $j \in \{0,\ldots,d-2\}$.  If for all $i \in \{0,\ldots,j-1\}$ the inequality $\nu_i(\mD,d) > \nu_j(\mD,d)$ holds, then $c_{1}(Y) = \ldots = c_{j+1}(Y)$.   
\end{lemma}
		
\begin{proof}
Suppose not. If $Y = \emptyset$ we are done, so consider $Y \neq \emptyset$.  Let $\ell \in [j]$ be an index such that $c_{\ell}(Y) > c_{\ell+1}(Y)$. Then the point $P = (b+c_{\ell}(Y), \ell)$  is a removable point for $\mD$.  Therefore, for $\mD' = \mD \setminus \{P\}$, we have 
\[
\nu_{k}(\mD',d) = \nu_{k}(\mD,d)
\]
for all $k \geq \ell$. In particular, $\nu_{k}(\mD',d) = \nu_{k}(\mD,d)$ for all $k \geq j$.  By the assumption on $\nu_j(\mD,d)$, this implies that  $\nu_{\min}(\mD',d) = \nu_{\min}(\mD,d)$, contradicting the irreducibility hypothesis.
\end{proof}
		
\begin{proposition}\label{prop: inf decending chain nu_j}
Let $(\mD,d)$ be an irreducible pair with $\mD\cap \mB_{n,d-1} \neq \emptyset$ in $(a,b)$-standard form.  Let $j \in \{0,\ldots,d-2\}$.  If for all $i \in \{0,\ldots,j-1\}$ the inequality $\nu_i(\mD,d) > \nu_j(\mD,d)$ holds, then also
\[
\nu_j(\mD,d) > \nu_{j+1}(\mD,d).
\]
In particular, for all $i \in \{0,\ldots,j\}$ the inequality $\nu_i(\mD,d) > \nu_{j+1}(\mD,d)$ holds.
\end{proposition}
\begin{proof}
 By Lemmata \ref{lem:cx=0} and \ref{lem:cy=} we have $c_{d-i}(X) = 0$  and  $c_{i}(Y) = c_1(Y)$ for all $i \in [j+1]$.
Hence, by Proposition \ref{prop:nui-nui-1}, we get
\[
\nu_{j+1}(\mD,d)-\nu_{j}(\mD,d) = b-a + d - 2(j+1) + c_{1}(Y) = 	\nu_{j}(\mD,d)-\nu_{j-1}(\mD,d) - 2 < 0.
\]
\end{proof}
		
\begin{corollary}\label{cor:irreducible implies nu_0 = nu_min}
Let $(\mD,d)$ be an irreducible pair with $\mD \cap \mB_{n,d-1} \neq \emptyset$.  Then 
\[
\nu_0(\mD,d) = \nu_{\min}(\mD,d) = \nu_{d-1}(\mD,d).
\]
Moreover, if $(\mD,d)$ is in $(d-1,d-1)$-standard form, then there exists a $j \in [d-2]$ for which  $\nu_j(\mD,d) = \nu_0(\mD,d)$. 
\end{corollary}
\begin{proof}
 It is enough to show that $\nu_{\min}(\mD,d)=\nu_0(\mD,d)$. Indeed, by taking the adjoint diagram, we get also $\nu_{\min}(\mD,d)=\nu_{\min}(\mD^\top,d)=\nu_0(\mD^\top,d)=\nu_{d-1}(\mD,d)$, where the second inequality follows from Lemma \ref{lem:adjoint_irred}.

Assume on the contrary that  $\nu_{\min}(\mD,d)<\nu_0(\mD,d)$, and let $j$ be the minimum index in $[d-1]$ such that $\nu_{\min}(\mD,d)=\nu_j(\mD,d)$. Then, by Proposition \ref{prop: inf decending chain nu_j}, we must have $j = d-1$ and thus
$\nu_{\min}(\mD,d) = \nu_{d-1}(\mD,d)<\nu_k(\mD,d)$ for every $k \in \{0,\ldots,d-2\}$. Hence, by Lemmata \ref{lem:cx=0} and \ref{lem:cy=} we get $c_1(X)=0$ and $c_{1}(Y)=c_{d-1}(Y) = 0$, so $Y=X=\emptyset$ and $\mD$ must be a rectangle, that is, $\mD=[a]\times[b]$. However, by Lemma \ref{lem:rectangle_irred} we must have $a=b$, but in this case it holds $\nu_0([a]\times [a],d)=a(a-d+1)=\nu_{d-1}([a]\times[a])$, leading to a contradiction, showing the first part of the statement.

\medskip 

Now assume that $(\mD,d)$ is in $(d-1,d-1)$-standard form. Recall that by Lemma \ref{lem: addible and removible standard form}, $P:=(d-1,d-1)\in \mR(\mD)$. Suppose on the contrary that for every $j \in [d-2]$ we have $\nu_j(\mD,d)>\nu_{\min}(\mD,d)$. Then, we have
$$\nu_j(\mD\setminus\{P\},d)=\nu_j(\mD,d)-1 \geq \nu_{\min}(\mD,d), \quad \mbox{ for each } j \in [d-2],$$
and 
$$\nu_i(\mD\cup \{P\},d)=\nu_i(\mD,d)=\nu_{\min}(\mD,d),\quad \mbox{ for each } i \in \{0,d-1\}.$$
By Proposition \ref{prop:characterization_irreducible}, this contradicts the hypothesis of irreducibility of $(\mD,d)$.
\end{proof}

\begin{proposition} \label{prop: nu_0 = nu_min implies irreducible}
Let $(\mD,d)$ be a pair with $\mD\cap \mB_{n,d-1}\neq \emptyset$ in $(a,b)$-standard form. 
\begin{enumerate}[label = (\arabic*)]
    \item Assume that $(a,b)\neq(d-1,d-1)$. If $\nu_0(\mD,d) = \nu_{\min}(\mD,d) = \nu_{d-1}(\mD,d)$, then $(\mD,d)$ is irreducible.
    \item Assume that $(a,b)=(d-1,d-1)$.
    If $\nu_0(\mD,d) = \nu_{\min}(\mD,d) = \nu_{d-1}(\mD,d)$ and there exists $j\in[d-2]$ such that $\nu_j(\mD,d)=\nu_{\min}(\mD,d)$, then $(\mD,d)$ is irreducible.
\end{enumerate}

\end{proposition}

\begin{proof}
Let us first assume that $(a,b)\neq(d-1,d-1)$ and, without loss of generality, that $a \geq d$. 

Let $P \in \mA(\mD)$ be an addible point. By Lemma \ref{lem: addible and removible standard form}, we can assume without loss of generality that $P \in \mA(X)$. Then $\nu_0(\mD \cup \{P\},d) = \nu_0(\mD,d)$ and thus  $\nu_{\min}(\mD \cup \{P\},d) = \nu_{\min}(\mD,d)$. 

Similarly, let $Q \in \mR(\mD)$ be a removable point. If $Q = (a,b)$, then since $a \geq d$ we have $\nu_0(\mD \setminus \{Q\}, d) = \nu_0(\mD, d) - 1$.   Else, by Lemma \ref{lem: addible and removible standard form} we can assume w.l.o.g. that $Q \in \mR(X)$ and we find $\nu_{d-1}(\mD \setminus \{Q\}, d) = \nu_{d-1}(\mD, d) - 1$. For both assumptions on $Q$ we thus get $\nu_{\min}(\mD \setminus \{Q\}, d) = \nu_{\min}(\mD, d) - 1$. Irreducibility now follows from Proposition \ref{prop:characterization_irreducible}.\\ 

Secondly, let $(a,b) = (d-1,d-1)$. By the same arguments as above we have   $\nu_{\min}(\mD \cup \{P\},d) = \nu_{\min}(\mD,d)$ for any addible point $P \in \mA(\mD)$ and $\nu_{\min}(\mD \setminus \{Q\}, d) = \nu_{\min}(\mD, d) - 1$ for any removable point  $Q \in \mR(X)\cup \mR(Y)$. In the remaining case when $Q = (d-1,d-1)$, we need the existence of $j \in [d-2]$ with $\nu_j(\mD,d) = \nu_{\min}(\mD,d)$. Then  $\nu_{\min}(\mD \setminus \{Q\},d) = \nu_j(\mD \setminus \{Q\},d) =  \nu_j(\mD,d)-1 = \nu_{\min}(\mD,d)-1$. Again irreducibility follows from Proposition \ref{prop:characterization_irreducible}.
\end{proof}

We are now ready to put everything together and give a proof of Theorem \ref{thm:main}.

\begin{proof}[Proof of Theorem \ref{thm:main}]
The equivalence $(1) \Leftrightarrow (2)$ follows directly from Corollary \ref{cor:irreducible implies nu_0 = nu_min} and Proposition \ref{prop: nu_0 = nu_min implies irreducible}. It remains to show the equivalence $(2) \Leftrightarrow (3)$. This equivalence is a direct consequence of Proposition \ref{prop:nui-nui-1}, since it implies that
$\nu_j(\mD,d)\geq \nu_0(\mD,d)$ if and only if
$$j(b-a+d-1-j) +  \sum_{i = 1}^{j} c_{d-i} (X)  - \sum_{i = 1}^{j}c_i(Y)\ge 0,$$
and $\nu_0(\mD,d)=\nu_{d-1}(\mD,d)$ if and only if 
$$0=(d-1)(b-a) +  \sum_{i = 1}^{d-1} c_{d-i} (X)  - \sum_{i = 1}^{d-1}c_i(Y)=(d-1)(b-a) +|X|-|Y|.$$
% \todo[inline]{show 2 equiv 3}
\end{proof}

\subsection{Considerations about small diagrams}
We want to formalize some considerations when we want to study irreducible Ferrers diagrams such that $\mD\cap \mB_{n,d-1}=\emptyset$. In this case, Theorem \ref{thm:main} is not valid, and indeed there might be some Ferrers diagrams satisfying condition 2. but which are not irreducible. In the sequel we try to study this case. 

Consider triangular diagrams. Let $\mT_n$ be the triangular $n\times n$ Ferrers diagram, given by $\mT_n=(n,n-1,\ldots,2,1)$, and let $d \in \mathbb N$.

\begin{lemma}\label{lem:irreducible_contains_triangle}
    Let $\mD$ be a Ferrers diagram. Then, $\nu_{\min}(\mD,d)=0$ if and only if $\mD\not\supseteq \mT_d$. In particular, if $\mD\neq \emptyset$ and the  pair $(\mD,d)$ is irreducible, then, $\mD\supseteq \mT_{d}$. 
\end{lemma}

\begin{proof}
     The first part of the statement follows from the fact that $\nu_j(\mD,d)=0$ if and only if $(d-j,j+1)\notin\mD$. This is due to the fact that $\nu_j(\mD,d)$ coincides with the size of $\mD\cap\mS_{n,d,j}$, which is empty if and only if $(d-j,j+1)\notin\mD$. Moreover, by Lemma \ref{lem:irreducible_empty}, the only irreducible pair $(\mD,d)$ with $\nu_{\min}(\mD,d)=0$ is $(\emptyset,d)$, contradicting the hypothesis that $\mD\neq \emptyset$.
\end{proof}

A first consequence of Lemma \ref{lem:irreducible_contains_triangle} is that, when constructing the $d$-Young digraph $\mathfrak G_d$ and the one $\mathfrak G_{d,n}$ of order $n$ for determining the irreducible pairs, we can restrict ourselves to the sublattices $\mathfrak D/\mT_d$ and $\mathfrak D/\mT_d$ of the Young lattice consisting of all the Ferrers diagrams containing $\mT_d$. More precisely, define
$$\mathfrak D/\mT_d:=\{\mD \in \mathfrak D\,:\, \mD\supseteq \mT_d\}$$
$$\mathfrak D_n/\mT_d:=\{\mD \in \mathfrak D_n\,:\, \mD\supseteq \mT_d\},$$
and let $\mathfrak G_d/\mT_d$ and $\mathfrak G_{d,n}/\mT_d$ be 
the induced subdigraphs of $\mathfrak G_d$ and $\mathfrak G_{d,n}$ on $\mathfrak D$ and $\mathfrak D_n$, respectively.

\begin{proposition}
    Let $\mD$ be a nonempty Ferrers diagram of order $n$. Then $(\mD,d)$ is irreducible if and only if it is a source in $\mathfrak G_{d,n}/\mT_d$.
\end{proposition}

\begin{proof}
    Assume that $(\mD,d)$ is irreducible. Since $\mD\neq\emptyset$, by Lemma \ref{lem:irreducible_contains_triangle}, $\mD\supseteq \mT_d$, and thus $\mD \in \mathfrak D_n/\mT_d$. Moreover,  since $(\mD,d)$ is irreducible, $\mD$ is a source in $\mathfrak G_{d,n}$, and thus it must also be a source in the induced subgraph $\mathfrak G_{d,n}/\mT_d$.

    Let $\mD\in \mathfrak D_n/mT_d$, and $\mD'=\mD\setminus\{P\}$ be such that $\mD'\notin \mathfrak D_n/\mT_d$. Then, by Lemma \ref{lem:irreducible_contains_triangle}, we have $\nu_{\min}(\mD',d)=0$ and $\nu_{\min}(\mD,d)>0$, implying $\nu_{\min}(\mD,d)=1$. Therefore, $\mD\stackrel{d}{\longrightarrow}\mD'$. This shows that from all the edges connecting a Ferrers diagram $\mD\in\mathfrak D_n/\mT_d$ with a Ferrers diagram $\mD'\notin\mathfrak D_n/\mT_d$ are outgoing. Hence, any source $\mD$ of $\mathfrak G_{d,n}/\mT_d$ is also a source of $\mathfrak G_{d,n}$, which means that $(\mD,d)$ is irreducible.
\end{proof}

\begin{example}
    We consider the same setting of Example \ref{exa:n=d=3}, with $n=d=3$. The induced $d$-Young digraph $\mathfrak G_{3,3}/\mT_3$ obtained only restricting to the sublattice $\mathfrak D_3/\mT_3$ is given below.

     \bigskip 

\noindent\begin{tabular}{p{\textwidth}}
\adjustbox{scale=0.35,center}{
\input{TikZ_pictures/BulletsSubDigraph_n=3_d=3}
} \\ 
\\
\hline 
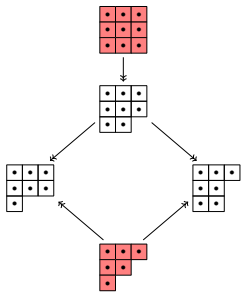
 \footnotesize{The irreducibility subdigraph $\mathfrak G_{3,3}/\mT_3$. The coloured diagrams are all the nonempty irreducible.} \\
\hline
\end{tabular}
\end{example}

\medskip

\begin{example}
We consider the same setting of Example \ref{exa:n=4}, with $n=4$ and $d\in\{3,4\}$. The induced $d$-Young digraphs $\mathfrak G_{3,4}/\mT_3$ and $\mathfrak G_{4,4}/\mT_4$ obtained only restricting to the sublattices $\mathfrak D_4/\mT_3$ and $\mathfrak D_4/\mT_4$ are given below.

 \bigskip 

\adjustbox{scale=0.8,center}{\begin{tabular}{c|c}
\input{TikZ_pictures/BulletsSubDigraph_n=4_d=3} 
 & 
\hspace*{-2cm}\input{TikZ_pictures/BulletsSubDigraph_n=4_d=4}

\\ 
\\
\hline 
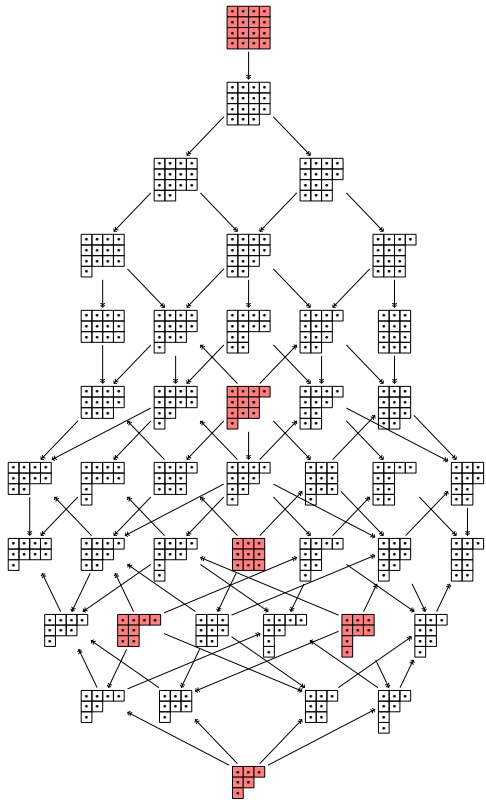
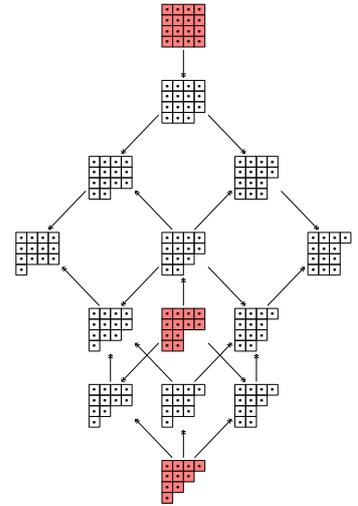
 \multicolumn{2}{c}{\footnotesize{The irreducibility subdigraphs $\mathfrak G_{4,3}/\mT_3$ and $\mathfrak G_{4,4}/\mT_4$. The coloured diagrams are all the nonempty irreducible.}} \\
\hline
\end{tabular}}

\end{example}

\begin{example}\label{Exa:n=d=5}
    We consider the case $n=d=5$. The induced $5$-Young digraph $\mathfrak G_{5,5}/\mT_5$ obtained only restricting to the sublattice $\mathfrak D_5/\mT_5$ reveals all the  Ferrers diagrams $\mD$ of order $5$ such that $(\mD,5)$ is irreducible, as illustrated below.

     \bigskip 

\noindent\begin{tabular}{p{\textwidth}}
\adjustbox{scale=0.35,center}{
\input{TikZ_pictures/BulletsSubDigraph_n=5_d=5}

} \\ 
\\
\hline 
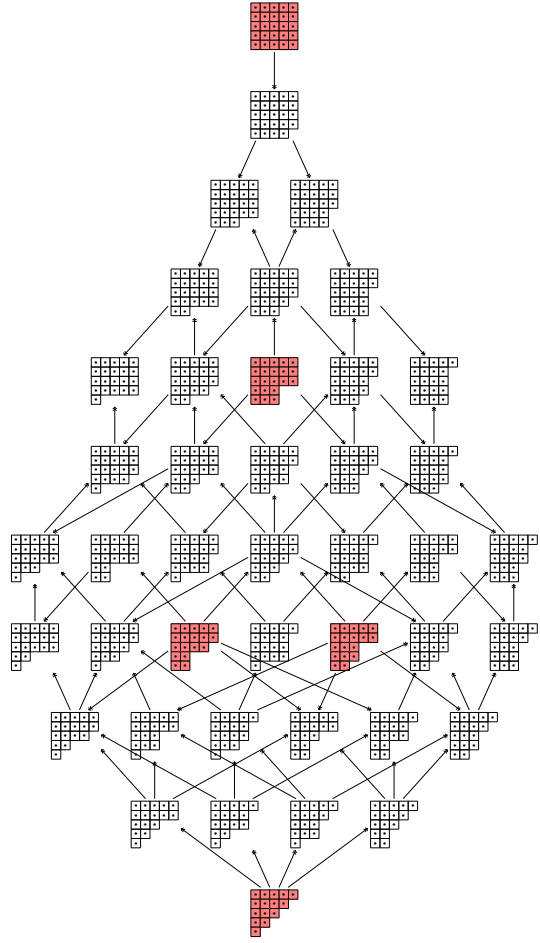
 \footnotesize{The irreducibility subdigraph $\mathfrak G_{5,5}/\mT_5$. The coloured diagrams are all the nonempty irreducible.} \\
\hline
\end{tabular}
\end{example}

As a consequence of Lemma \ref{lem:irreducible_contains_triangle}, the Ferrers diagrams $\mD$ such that $(\mD,d)$ is irreducible and not satisfying Theorem \ref{thm:main} -- that is, with $\mD\cap \mB_{n,d-1}=\emptyset$ -- are those such that $\mT_{d}\subseteq \mD$ and $\mD\not\supseteq [d-1]^2$. Among them, we have some of the triangular Ferrers diagrams, as the next result shows.

\begin{proposition}\label{prop:irreducible_triangular}
    The pair $(\mT_n,d)$ is irreducible if and only if 
    $d\leq n \leq 2d-3$. 
\end{proposition}

\begin{proof} 
    First, observe that if $n\ge 2d-2$, then $\mT_n\cap \mB_{n,d-1}=\{((j,n+1-i) )\,:\,d-1\le j\le  i\le n-d+2\}$. Hence,  $(\mT_n,d)$ is not in $(a,b)$-standard form, and thus $(\mT_n,d)$ cannot be irreducible, due to Lemma \ref{lem: if }. 
    
    Thus, assume now that $d\le n \le 2d-3$.   We use Proposition \ref{prop:characterization_irreducible} to show that in this case $(\mT_n,d)$ is irreducible. The addible and removable points of $\mT_n$ are given, respectively, by
    $$\mA(\mT_n)=\{(n+2-i,i)\,:\, i \in [n+1] \}, \quad \mR(\mT_n)=\{(n+1-i,i)\,:\, i \in [n]\}.$$
    Note that, for every triangular Ferrers diagram $\mT_n$ it holds that
    $$\nu_j(\mT_n,d)=\frac{(n-d+1)(n-d+2)}{2}.$$
    In particular, for every $P=(n+1-i,i)\in \mR(\mT_n)$, we have
    $$\nu_{\min}(\mT_n\setminus\{P\},d)=\nu_{i-1}(\mT_n\setminus\{P\},d)=\nu_i(\mT_n,d)-1=\nu_{\min}(\mT_n,d)-1.$$
    Moreover, let $P=(n+2-i,i)\in \mA(\mT_n)$. If $i\le d-1$, then $$\nu_{i}(\mT_n\cup\{P\},d)=|\mS_{n,d,i}\cap(\mT_n\cup\{P\})|=|\mS_{n,d,i}\cap \mT_n|=\nu_i(\mT_n,d)=\nu_{\min}(\mT_n,d).$$ If instead $i\ge d$, then $n+2-i\le n+2-d\le d-1$, and 
    $$\nu_{0}(\mT_n\cup\{P\},d)=|\mS_{n,d,0}\cap(\mT_n\cup\{P\})|=|\mS_{n,d,0}\cap \mT_n|=\nu_i(\mT_n,d)=\nu_{\min}(\mT_n,d).$$
    Thus, by Proposition \ref{prop:characterization_irreducible} we conclude.
\end{proof}

Note that the triangular Ferrers diagram $\mT_{2d-3}$ is the only one among those listed in Proposition \ref{prop:irreducible_triangular} that is included in Theorem \ref{thm:main}. Moreover, we remark that all the triangular Ferrers diagrams are \emph{strictly monotone}, and the Etzion-Silberstein conjecture has been proved to hold for them in \cite{neri2023proof}. 

\begin{remark}
     At this point, the reader might wonder whether an analogue of Theorem \ref{thm:main} holds also when $\mD\cap \mB_{n,d-1}=\emptyset$. First of all, in this case, it is difficult to uniquely define a notion of \emph{standard form}, since an analogue of Lemma \ref{lem: if } does not hold. If we try to omit this notion, we can also verify that condition (2) in Theorem \ref{thm:main} is not sufficient. If $\nu_{\min}(\mD,d)=0$, then this is clear. So, we might restrict ourselves to the diagrams for which $\nu_{\min}(\mD,d)>0$ (which is guaranteed when $\mD\cap \mB_{n,d-1}\neq \emptyset$). By observing the values $\nu_i$'s of the irreducible pairs in the digraph $\mathfrak G_{5,5}$ given in Example \ref{Exa:n=d=5}, we have that they are represented by\\
     \begin{table}[h]
       \centering
\begin{tabular}{|c||c|c|c|c|c|}
\hline
Irreducible $\mD$ ($d=5$)& $(5,5,5,5,5)$ & $(5,5,5,3,3)$ & $(5,5,3,3,2)$ & $(5,5,4,2,2)$ & $(5,4,3,2,1)$  \\ \hline
$(\nu_{0}, \ldots, \nu_4)$  & $(5,8,9,8,5)$ & $(3,4,5,4,3)$ & $(2,2,2,3,2)$ & $(2,3,2,2,2)$ & $(1,1,1,1,1)$ \\ \hline
\end{tabular} .
\end{table}

     While in this case, one might think that the condition $\nu_0(\mD,d)=\nu_{d-1}(\mD,d)=\nu_{\min}(\mD,d)$ is still necessary, it is clear that it cannot be sufficient, since the diagram $\mD=(5,5,4,3,2)$ with $d=5$ has $\nu_i$'s equal to $(2,3,3,3,2)$ but it is not irreducible. For larger order $n$, the condition is also not necessary, as exemplified by the irreducible diagram pair $\mD = (6,6,3,3,3)$ of order $n=6$ and $d = 5$, having $\nu_j$'s equal to $(4,3,3,4,3)$.
\end{remark}

\section{Polytopal characterization of irreducible Ferrers diagram pairs}\label{sec: polytopal char}

For given natural numbers $a,b,d$ with $\min(a,b) \geq d-1$, let us denote by $\irr_d^{(a,b)}$ the (finite) set of irreducible diagram pairs $(\mD,d)$ in $(a,b)$-standard form. In Section \ref{sec:classification} we gave a classification of the elements of this set in terms of the columns $X$ and $Y$ of their $(a,b)$-standard form. In this section, we complement this classification by providing an alternative description as sets of integer points in $\R^{2d-4}$ and $\R^{2d-3}$ satisfying systems of linear inequalities, in particular as polytopes. 
\begin{definition}[Polytope]
A subset $\wp \subset \R^N$ is called a \emph{bounded convex polytope} (or \emph{polytope} for short) if one of the following two
statements holds: 
\begin{enumerate}[label = (\arabic*)]
    \item $\wp$ is the convex hull of a finite set of points $V \subset \R^N$ ($V$-polytope);
\item  $\wp$ is the bounded intersection of finitely many closed halfspaces ($H$-polytope).
\end{enumerate}
By the main theorem on polytopes \cite[Thm 1.1]{ziegler1993lectures}, these two statements are equivalent.
\end{definition}

This original bridge between irreducible diagram pairs and integer points of polytopes is first explained in Subsection \ref{subsec: poly def}. Afterwards, we show in Subsection \ref{subsec: integral poly} that these classifying polytopes in $\R^{2d-3}$ are moreover \emph{integral} (i.e. all vertices have integer coordinates), making them suitable to study their integer points using Ehrhart polynomials (see \cite{beck2007computing,rehberg2025}). In particular, we show in that when $a \geq d$ or $b \geq d$, the set $\irr_d^{(a,b)}$ is in natural bijection with the integer points of a polytope $\poly_d^{(a,b)} \subset \R^{2d-4}$. Similarly, the set
\[
\irr_d^{\mu} = \bigcup_{\substack{(a,b)\in \N^2\\ \min(a,b) =\mu}} \irr_d^{(a,b)}
\]
with $3 \leq d \leq \mu$ is in natural bijection with the integer points of an \emph{integral} polytope $\poly_d \subset \R^{2d-3}$.

\subsection{Irreducible pairs as integer points of polytopes}\label{subsec: poly def}

\begin{definition}\label{def: Pdab constraints}
     Let $d$, $a$, $b$ be natural numbers satisfying $\min(a,b) \geq d-1$. We define the set $\poly_d^{(a,b)}\subseteq \R^{2d-4}$ as the set of $(x_1,\ldots,x_{d-2},y_1,\ldots,y_{d-2})\in \R^{2d-4}$ satisfying the constraints
     \begin{equation}\label{eq:poly1}\begin{cases} x_j-x_{j+1}\geq 0  &  j \in [d-3]\\
  y_j-y_{j+1}\geq 0  &  j \in [d-3]\\
x_j,y_j \geq 0 &  j \in [d-2]\\
		j(b-a+d-1-j) +  \sum\limits_{i = 0}^{j-2} x_{d-2-i}  - \sum\limits_{i = 1}^{j}y_i \geq 0 \;\; & j \in [d-2] \\
\sum\limits_{i=1}^{d-2}y_i - \sum\limits_{i=1}^{d-2}x_i = (b-a)(d-1).
 \end{cases}
 \end{equation}
Moreover, when $a=b=d-1$, we impose the extra constraint that there exists a $j \in [d-2]$ for which the equality $j(d-1-j) +  \sum\limits_{i = 0}^{j-2} x_{d-2-i}  - \sum\limits_{i = 1}^{j}y_i = 0$ holds. We will use the short-hand notation $\poly_d^{(a,b)}(\Z) := \poly_d^{(a,b)} \cap \Z^{2d-4}$ for the subset of integer points of $\poly_d^{(a,b)}$. 
\end{definition}

\begin{theorem}\label{thm bijection ab with polytope}
Let $d$, $a$, $b$ be natural numbers satisfying $\min(a,b) \geq d-1$. There is a bijection between the set $\irr_d^{(a,b)}$ and the set $\poly_d^{(a,b)}(\Z) $ of integer points given by
\begin{align*}
   \Psi_d^{(a,b)}: \irr_d^{(a,b)} \; &\to \qquad\qquad\qquad\qquad \poly_d^{(a,b)}(\Z) \\
    (\mD,d) \quad&\mapsto (c_1(X),\ldots,c_{d-2}(X),c_1(Y),\ldots,c_{d-2}(Y))
\end{align*}
with $X = X(\mD,d)$ and  $Y = Y(\mD,d)$.
\end{theorem}

\begin{proof}
 First, for any irreducible diagram pair $(\mD,d) \in \irr_d^{(a,b)}$, Theorem \ref{thm:main} ensures that $\Psi_d^{(a,b)}(\mD,d)$ satisfies the constraints of Definition \ref{def: Pdab constraints}, so $\Psi_d^{(a,b)}$ is well-defined. 

In the other direction, any vector  $v = (x_1,\ldots,x_{d-2},y_1,\ldots,y_{d-2}) \in \poly_d^{(a,b)} (\Z) $ determines uniquely two shifted Ferrers diagrams $X,Y$ by setting their column heights as $c_i(X) = x_i$ and $c_i(Y) = y_i$ for each $i \in [d-2]$. We naturally obtain a Ferrers diagram $\mD := ([a] \times [b]) \sqcup X^\top \sqcup Y$, and again by Theorem \ref{thm:main}, $(\mD,d)$ is irreducible. Hence we have a well-defined inverse map $(\Psi_d^{(a,b)})^{-1}: \poly_d^{(a,b)}(\Z) \to \irr_d^{(a,b)} : v \mapsto (\mD,d)$. 
\end{proof}

\begin{observation}\label{obs: min a b only depends on difference}
Note that, when $a\geq d$ or $b \geq d$, the sets  $\poly_d^{(a,b)}$ and $\poly_d^{(a,b)}(\Z)$ only depend on the difference $\Delta := b-a$. So for fixed $\Delta \in \Z$, all sets of the form $\irr_d^{(a, a + \Delta)}$ with $a\geq d$ or $a + \Delta \geq d$ (and of course $\min(a,a+\Delta) \geq d-1$) are naturally in bijection. More precisely, we have that $\poly_d^{(a',a'+\Delta)}= \poly_d^{(a,a+\Delta)}$, the map $$\Psi^{(a,a+\Delta)}_d\circ (\Psi^{(a',a'+\Delta)}_d)^{-1}: \poly_d^{(a',a'+\Delta)}(\Z)\longrightarrow \poly_d^{(a,a+\Delta)}(\Z)$$ is the identity map, and $$ (\Psi^{(a',a'+\Delta)}_d)^{-1}\circ \Psi^{(a,a+\Delta)}_d: \irr_d^{(a, a + \Delta)} \longrightarrow \irr_d^{(a', a' + \Delta)} $$
maps bijectively any irreducible pair $(\mD,d)$ in $(a,a+\Delta)$-standard form to an irreducible pair $(\mD',d)$ in $(a',a'+\Delta)$-standard form with $X(\mD,d) = X(\mD',d)$ and $Y(\mD,d) = Y(\mD',d)$.
\end{observation}

\begin{example}\label{ex: d=4 and delta=1}
Let $d = 4$ and $\Delta = 1$. The maps $\Psi_4^{(3,4)}$ and $\Psi_4^{(4,5)}$ give natural bijections between the sets of irreducible pairs $\irr_4^{(3,4)}$, $\irr_4^{(4,5)}$ and the set $\poly_4^{(3,4)}=\poly_4^{(4,5)}$ of integer points. The induced bijection between $\irr_4^{(3,4)}$ and $\irr_4^{(4,5)}$ is represented in the following figure, with corresponding diagram pairs and integer points below each other.\\

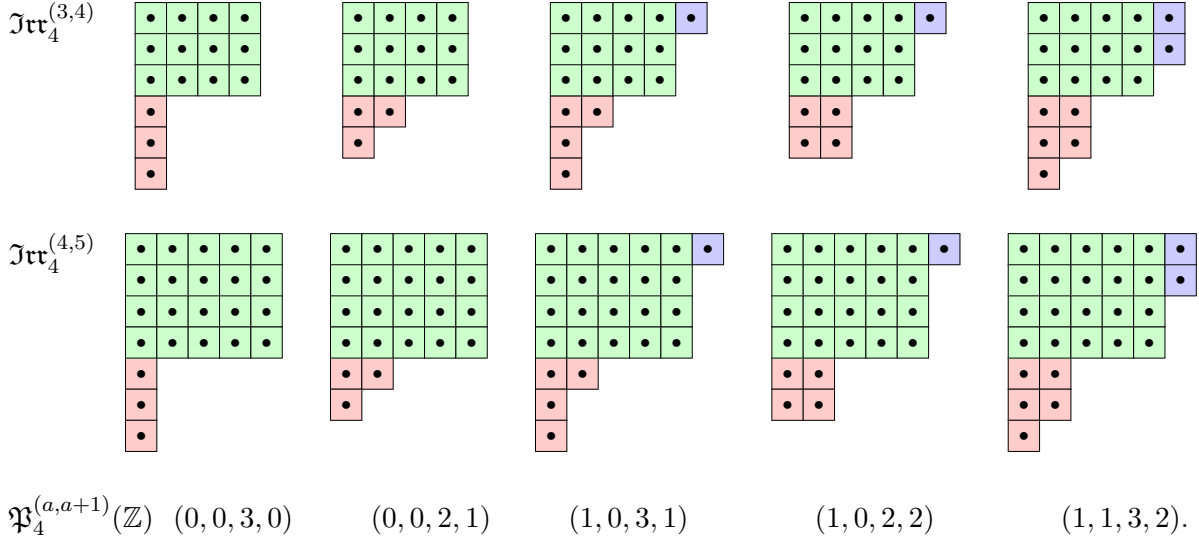
\ytableausetup{baseline}
\noindent
$\irr_4^{(3,4)}$ \quad  \scalebox{0.7}{\ydiagram[*(white!80!green) \bullet]{4,4,4}  *[*(white!80!red) \bullet] {0,0,0,1,1,1}
\qquad \qquad  
\ydiagram[*(white!80!green) \bullet]{4,4,4}  *[*(white!80!red) \bullet] {0,0,0,2,1}
\qquad \qquad
\ydiagram[*(white!80!blue) \bullet]
{4+1} *[*(white!80!green) \bullet]{4,4,4} *[*(white!80!red) \bullet] {0,0,0,2,1,1}
\qquad \qquad   
\ydiagram[*(white!80!blue) \bullet]
{4+1} *[*(white!80!green) \bullet]{4,4,4} *[*(white!80!red) \bullet] {0,0,0,2,2} 
\qquad \qquad  
\ydiagram[*(white!80!blue) \bullet]
{4+1, 4+1} *[*(white!80!green) \bullet]{4,4,4} *[*(white!80!red) \bullet] {0,0,0,2,2,1} 
} 
\text{ }\\
\text{ }\\
\text{ }\\
$\irr_4^{(4,5)}$ \ \ \scalebox{0.7}{\ydiagram[*(white!80!green) \bullet]{5,5,5,5}  *[*(white!80!red) \bullet] {0,0,0,0,1,1,1}
\qquad  \ 
\ydiagram[*(white!80!green) \bullet]{5,5,5,5}  *[*(white!80!red) \bullet] {0,0,0,0,2,1}
\qquad \
\ydiagram[*(white!80!blue) \bullet]
{5+1} *[*(white!80!green) \bullet]{5,5,5,5} *[*(white!80!red) \bullet] {0,0,0,0,2,1,1}
\qquad \    
\ydiagram[*(white!80!blue) \bullet]
{5+1} *[*(white!80!green) \bullet]{5,5,5,5} *[*(white!80!red) \bullet] {0,0,0,0,2,2} 
\qquad \ 
\ydiagram[*(white!80!blue) \bullet]
{5+1, 5+1} *[*(white!80!green) \bullet]{5,5,5,5} *[*(white!80!red) \bullet] {0,0,0,0,2,2,1} 
} 
\text{ }\\
\text{ }\\
\text{ }\\
$\poly_4^{(a,a+1)}(\Z)$ \ $(0,0,3,0)$ \qquad  \ $(0,0,2,1)$ \qquad \  $(1,0,3,1)$ \qquad \qquad  $(1,0,2,2)$ \qquad \qquad   $(1,1,3,2)$.\\ 
  
   % \label{fig:AGEF}

\end{example}

\begin{observation}
The transposition $\mD \mapsto \mD^\top$ naturally induces an involution $\irr_d^{(a,b)} \to  \irr_d^{(b,a)} : (\mD,d) \mapsto (\mD^\top,d)$. Under the map $\Psi_d^{(a,b)}$ of Theorem \ref{thm bijection ab with polytope}, this gives us an involution $\poly_d^{(a,b)}(\Z) \to  \poly_d^{(b,a)}(\Z): (x_1,\ldots,x_{d-2},y_1,\ldots,y_{d-2}) \mapsto (y_1,\ldots,y_{d-2},x_1,\ldots,x_{d-2})$ .    
\end{observation}

Note that when $a \geq d$ or $b \geq d$ (and $\min(a,b)\geq d-1$), the set $\poly_d^{(a,b)}$ is the intersection of finitely many closed half-spaces as it is the solution space of a system of linear (in)equalities. We will now show that $\poly_d^{(a,b)}$ is bounded, making it a convex polytope.

\begin{lemma}\label{lem: b-a bounded}
Let  $\min(a,b)\geq d-1$ and  $(x_1,\ldots,x_{d-2},y_1,\ldots,y_{d-2}) \in \poly_d^{(a,b)}$. Then $0 \leq x_1 \leq a-b + d-2$ and $0 \leq y_1 \leq b-a + d-2$.  
\end{lemma}

\begin{proof}
The linear inequality $j(b-a+d-1-j) +  \sum_{i = 0}^{j-2} x_{d-2-i}  - \sum_{i = 1}^{j}y_i \geq 0$ for $j = 1$ reduces to $b-a+d-2  - y_1 \geq 0$. With the same argument for  the `transpose'  $(y_1,\ldots,y_{d-2},x_1,\ldots,x_{d-2}) \in \poly_d^{(b,a)}$ we obtain $a-b+d-2  - x_1 \geq 0$.
\end{proof}

\begin{corollary}\label{cor: b-a bounded}
If $\poly_d^{(a,b)}$ is non-empty, then $|b-a| \leq d-2$.    
\end{corollary}

\begin{corollary}
For any natural numbers $d$, $a$, $b$ with $\min(a,b)\geq d-1$, the set $\poly_d^{(a,b)}$ is contained in the hypercube $[0,2d-4]^{2d-4}$ and hence bounded. As direct consequences, $\poly_d^{(a,b)}(\Z)$ is finite, and $\poly_d^{(a,b)}$ is a bounded convex polytope when $a\geq d$ or $b\geq d$.
\end{corollary}

\begin{proof}
    Combining the results above we deduce that if $(x_1,\ldots,x_{d-1},y_1,\ldots,y_{d-2}) \in \poly_d^{(a,b)}$, then $x_1 \leq a-b +d-2 \leq 2d-4$ and $y_1 \leq b-a + d-2 \leq 2d-4$. Since $x_1 \geq x_2 \geq \ldots \geq x_{d-2}$ and $y_1 \geq y_2 \geq \ldots \geq y_{d-2}$, this means that $v \in [0,2d-4]^{2d-4}$.
\end{proof}

\begin{example}
Let us look at the polytope $\poly_4^{(a,a)}$ (i.e. $d=4$ and $b-a=0$) with $a \geq 4$. This polytope lies in the 3-dimensional hyperplane $H = \{(x_1,x_2,y_1,y_2) \in \R^4 \ \mid \ x_1+x_2-y_1-y_2=0\}$, and thus we can make a 3-dimensional plot of $\poly_4^{(a,a)}$ using an othogonal basis $(1,1,1,1), (1,-1,0,0),(0,0,1,-1)$ of $H$.
 
\begin{figure}[H]
	\centering
		\tdplotsetmaincoords{70}{120} %70,125
		\begin{tikzpicture}[tdplot_main_coords, xscale =2.5, yscale = 2.5]
		\draw[thick,->] (0,0,0)--(1.2,0,0) node[anchor=north east]{\scalebox{0.8}{$(1,-1,0,0)$}};
		\draw[thick,->] (0,0,0)--(0,1.2,0) node[anchor=north west]{\scalebox{0.8}{$(0,0,1,-1)$}};
		\draw[thick,->, opacity=0.5] (0,0,0)--(0,0,2.5) ;

        \coordinate (A) at (0,0,0);
        \coordinate (B) at (0.5,0.5,0.5);
        \coordinate (C) at (1,1,1);
        \coordinate (D) at (0,0,1);
        \coordinate (E) at (0,0,2);
        \coordinate (F) at (1,0,1);
        \coordinate (G) at (0,1,1);
        \coordinate (H) at (0.5,0.5,1.5);

        \fill[fill=black] (D) circle (1pt) node[above left] {(1,1,1,1)};
        
        \draw[opacity=0.3] (D)--(C);
        
	\filldraw[        draw=blue,        fill=blue!50, opacity=0.1    ] (A)-- (E)   -- (C)-- cycle;
       \filldraw[        draw=blue,        fill=blue!50, opacity=0.1    ]          (D)            -- (F)            -- (C)            -- (G)-- cycle;
        \filldraw[draw=black, fill=blue!50, opacity=0.2 ]          (A)-- (C) -- (F) -- cycle;
        \filldraw[draw=black, fill=purple!50, opacity=0.2]          (A)-- (C) -- (G) -- cycle;
	\filldraw[draw=black, fill=green!50, opacity=0.0 ]          (A)-- (E) -- (F) -- cycle;
	\filldraw[draw=black, fill=orange!50, opacity=0.0 ]          (A)-- (E) -- (G) -- cycle;
        \filldraw[draw=black, fill=red!50, opacity=0.3 ]          (E)-- (C) -- (F) -- cycle; %purple
        \filldraw[draw=black, fill=orange!50, opacity=0.2 ]          (E)-- (C) -- (G) -- cycle;

        \draw[] (A) --(C);
        \draw[] (A) --(F);
        \draw[] (A) --(G);
        \draw[] (E) --(C);
        \draw[] (E) --(F);
        \draw[] (E) --(G);
        \draw[] (C) --(F);
        \draw[] (C) --(G);

\fill[fill=black] (A) circle (1pt) node[below] {(0,0,0,0)};
        \fill[fill=black] (B) circle (1pt) node[below right] {(1,0,1,0)};
        \fill[fill=black] (C) circle (1pt) node[below right] {(2,0,2,0)};

        \fill[fill=black] (E) circle (1pt) node[above] {(2,2,2,2)};

        \fill[fill=black] (F) circle (1pt) node[left] {(2,0,1,1)};
        \fill[fill=black] (G) circle (1pt) node[right] {(1,1,2,0)};
        \fill[fill=black] (H) circle (1pt) node[above right] {(2,1,2,1)};
        
\end{tikzpicture}
\text{}\\
\caption{The polytope $\poly_4^{(a,a)}$ with $a \geq 4$ and integer points indicated by black dots.}
\end{figure}
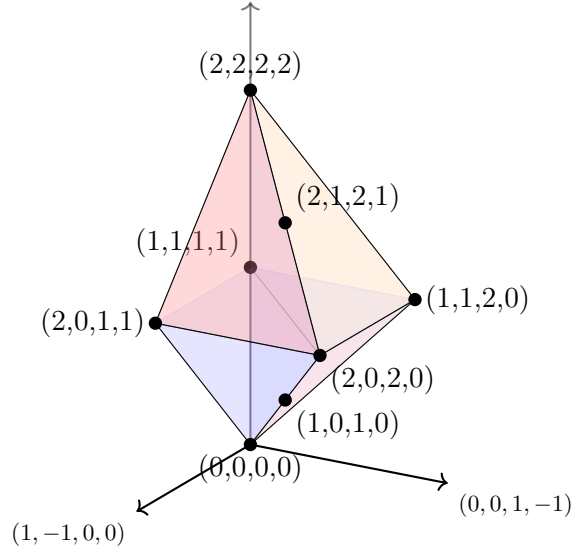

    \begin{figure}[H]
	\centering
		\tdplotsetmaincoords{70}{120} %70,125
		\begin{tikzpicture}[tdplot_main_coords, xscale =2.5, yscale = 2.5]
		\draw[thick,->] (0,0,0)--(1.2,0,0) node[anchor=north east]{\scalebox{0.8}{$(1,-1,0,0)$}};
		\draw[thick,->] (0,0,0)--(0,1.2,0) node[anchor=north west]{\scalebox{0.8}{$(0,0,1,-1)$}};
		\draw[thick,->, opacity=0.5] (0,0,0)--(0,0,2.5) ;

        \coordinate (C) at (1,1,1);
        \coordinate (E) at (0,0,2);
        \coordinate (F) at (1,0,1);
        \coordinate (G) at (0,1,1);
        \coordinate (H) at (0.5,0.5,1.5);

        \filldraw[draw=black, fill=red!50, opacity=0.3 ]          (E)-- (C) -- (F) -- cycle; %purple
        \filldraw[draw=black, fill=orange!50, opacity=0.2 ]          (E)-- (C) -- (G) -- cycle;

        \draw[] (E) --(C);
        \draw[] (E) --(F);
        \draw[] (E) --(G);
        \draw[] (C) --(F);
        \draw[] (C) --(G);

        \fill[fill=black] (C) circle (1pt) node[below right] {(2,0,2,0)};

        \fill[fill=black] (E) circle (1pt) node[above] {(2,2,2,2)};

        \fill[fill=black] (F) circle (1pt) node[left] {(2,0,1,1)};
        \fill[fill=black] (G) circle (1pt) node[right] {(1,1,2,0)};
        \fill[fill=black] (H) circle (1pt) node[above right] {(2,1,2,1)};
        
\end{tikzpicture}
\text{}\\
\caption{The set $\poly_4^{(3,3)}$ with integer points indicated by black dots. Since $a = b = d-1$, this set is not a convex polytope.}
\end{figure}
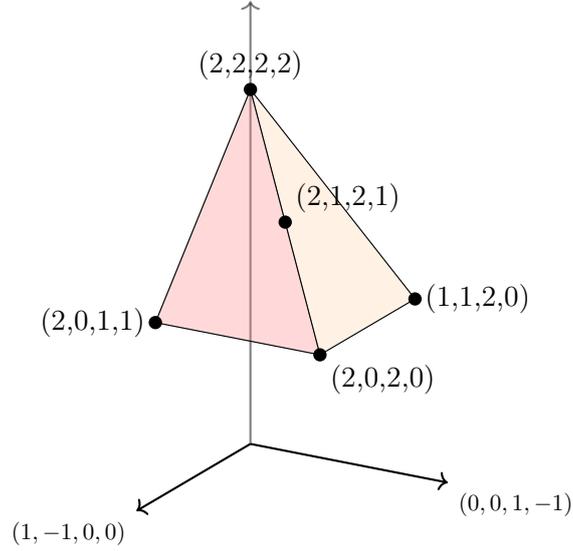
\end{example}

For given natural numbers $d, \mu$ satisfying $3 \leq d \leq \mu$, let us denote by $\irr_d^{\mu}$ the union
\[
\bigcup_{\substack{(a,b)\in \N^2\\ \min(a,b) =\mu}} \irr_d^{(a,b)},
\]
i.e. the set of irreducible diagram pairs $(\mD,d)$ in $(a,b)$-standard form for some $a,b$ satisfying $\min(a,b)=\mu$. By Corollary \ref{cor: b-a bounded}, $\irr_d^{(a,b)}$ is only non-empty whenever $\max(a,b) \leq \min(a,b)+(d-2)=\mu+d-2$, so $\irr_d^{\mu}$ is the union of a finite number of non-empty finite sets. 

 \newpage
 
\begin{definition}\label{def: Pd constraints}
     Let $d \geq 3$ be a natural number. We define the set $\poly_d\subseteq \R^{2d-3}$ as the set of $(x_1,\ldots,x_{d-2},y_1,\ldots,y_{d-2}, z)\in \R^{2d-3}$ satisfying the constraints
     \begin{equation}\label{eq:poly2}\begin{cases} x_j-x_{j+1}\geq 0  &  j \in [d-3]\\
  y_j-y_{j+1}\geq 0  &  j \in [d-3]\\
x_j,y_j \geq 0 &  j \in [d-2]\\
		j(z+d-1-j) +  \sum\limits_{i = 0}^{j-2} x_{d-2-i}  - \sum\limits_{i = 1}^{j}y_i \geq 0 \;\; & j \in [d-2] \\
\sum\limits_{i=1}^{d-2}y_i - \sum\limits_{i=1}^{d-2}x_i = z(d-1)
 \end{cases}.
 \end{equation}
Similarly as before, we will use the short-hand notation $\poly_d(\Z) := \poly_d \cap \Z^{2d-3}$ for the subset of integer points of $\poly_d$. 
\end{definition}

\begin{observation}\label{obs: polytope is union slices constant z}
According to the definitions of $\poly_d$ and $\poly_d^{(a,b)}$, one can derive 
$$\poly_d=\bigcup_{\substack{(a,b)\in\R^2\\ \min(a,b)=d}}\poly_d^{(a,b)}\times \{b-a\} = \left(\bigcup_{z \in\R_{\geq 0}}\poly_d^{(d,d+z)}\times \{z\}\right) \cup \left(\bigcup_{z \in\R_{< 0}}\poly_d^{(d-z,d)}\times \{z\}\right),$$
where we have extended the definition of $\poly_d^{(a,b)}$ to pairs $(a,b)\in \R^2$. Thus, restricting to its integer set $\poly_d(\Z)$, one has
 \[
\poly_d(\Z) =  \bigcup_{\substack{(a,b)\in \N^2\\ \min(a,b) = d}} \poly_d^{(a,b)}(\Z) \times \{b-a\} = \left(\bigcup_{z \in\Z_{\geq 0}}\poly_d^{(d,d+z)}\times \{z\}\right) \cup \left(\bigcup_{z \in\Z_{< 0}}\poly_d^{(d-z,d)}\times \{z\}\right).
\]
\end{observation}

Using Theorem \ref{thm bijection ab with polytope}, we can also link the polytope $\poly_d$ with the set of irreducible Ferrers diagram pairs $\irr_d^{\mu}$.

\begin{theorem}\label{thm bijection mu with polytope}
Let $d$, $\mu$ be natural numbers satisfying $3 \leq d \leq \mu$. There is a bijection between the set $\irr_d^{\mu}$ and the set $\poly_d(\Z) $ of integer points given by
\begin{align*}
   \Psi_d^{\mu}: \irr_d^{\mu} \; &\longrightarrow \qquad\qquad\qquad\qquad \poly_d(\Z) \\
    (\mD,d) &\longmapsto (c_1(X),\ldots,c_{d-2}(X),c_1(Y),\ldots,c_{d-2}(Y), b-a)
\end{align*}
with $X = X(\mD,d)$,  $Y = Y(\mD,d)$ and $(\mD,d)$ in $(a,b)$-standard form.
\end{theorem}

\begin{proof}
    By definition of $\irr_d^{\mu}$, it consists of a union of sets $\irr_d^{(a,b)}$ with $\min(a,b) = \mu$.    Similarly, by Observations \ref{obs: polytope is union slices constant z} and \ref{obs: min a b only depends on difference}, $\poly_d(\Z)$ consists of a union of sets $\poly_d^{(a,b)}(\Z) \times \{b-a\}$ with $\min(a,b) = \mu$. By Theorem \ref{thm bijection ab with polytope}, the map
    \begin{align*}
    \irr_d^{(a,b)} \; &\longrightarrow \poly_d^{(a,b)}(\Z)\times \{b-a\} \\
    (\mD,d) &\longmapsto (\Psi_d^{(a,b)}(\mD,d), b-a)
    \end{align*}
    is a bijection equal to the restriction of $\Psi_d^{\mu}$ to the set $\irr_d^{(a,b)}$, and thus $\Psi_d^{\mu}$ is a bijection.
\end{proof}

Furthermore, via a proof similar to that of Lemma \ref{lem: b-a bounded}, we deduce that the set $\poly_d$ is bounded and hence a convex polytope, as happens for $\poly_d^{(a,b)}$. 

\begin{lemma}
Let $3 \leq d \leq \mu$ and  $(x_1,\ldots,x_{d-2},y_1,\ldots,y_{d-2},z) \in \poly_d$. Then $0 \leq x_1 \leq -z + d-2$ and $0 \leq y_1 \leq z + d-2$, and so $|z| \leq d-2$.  
\end{lemma}

\begin{corollary}\label{cor: poly d bounded}
For any natural number $d \geq 3$, the set $\poly_d$ is contained in $[0, 2d-4]^{2d-4}\times [-d+2,d-2]$ and is therefore a bounded convex polytope. This also implies that that the integer point set $\poly_d(\Z)$ is finite.
\end{corollary}

\begin{example}
    Let $d = 3$. The polytope $\poly_{3} \subset \R^3$ is defined as the set of vectors $(x_1,y_1,z)=(x,y,z)\in\R^3$ satisfying the constraints
$$\begin{cases}
x\geq 0 \\
y\geq 0 \\ 
		 z - x +1\geq 0 &  \\
y  - x - 2z = 0.
 \end{cases}$$

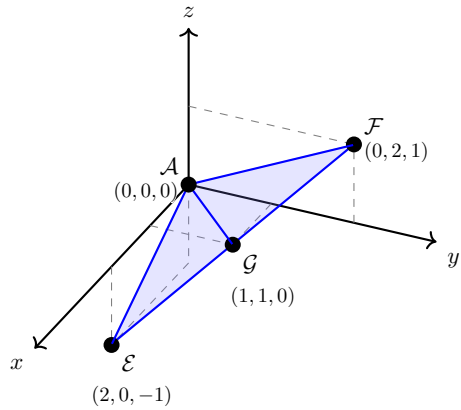
\begin{figure}[H]
\input{TikZ_pictures/Polytope_P3}
\caption{The polytope $\poly_3$, with integer points indicated with black dots, and labeled corresponding to the irreducible diagrams in $\irr_d^\mu$ for $\mu \geq d$.}
\end{figure}
The integer point set $\poly_3(\Z)$ consists of the four points $\{(0,0,0),(1,1,0), (2,0,-1),(0,2,1) \}$.  For $\mu=3$, the map $(\Psi_{3}^3)^{-1}$ sends these four points to the set of irreducible diagrams $\irr_{3}^3=\{\mA_3,\mG_4,\mE_5,\mF_5 \}$ respectively, shown here below.

\begin{figure}[H]
    \centering
\ytableausetup{baseline}
  \scalebox{0.7}{\ydiagram[*(white!80!green) \bullet]{3,3,3} 
\qquad \qquad  \ydiagram[*(white!80!blue) \bullet]
{3+1}
*[*(white!80!green) \bullet]{3,3,3} *[*(white!80!red) \bullet] {0,0,0,1}
\qquad \qquad      \ydiagram[*(white!80!blue) \bullet]
{3+2}
*[*(white!80!green) \bullet]{3,3,3,3} 
\qquad \qquad     \ydiagram[*(white!80!red) \bullet]
{0,0,0,0+1,0+1}
*[*(white!80!green) \bullet]{4,4,4}} 
    \caption{A graphical illustration of the $4$ Ferrers diagrams $\mA_3,\mG_4,\mE_5,\mF_5$, with their standard forms highlighted.}
   % \label{fig:AGEF}
\end{figure}

 Likewise, for $\mu=4$ the map $(\Psi_{3}^4)^{-1}$ sends the four points to the irreducible diagrams $\irr_{3}^4=\{\mA_4,\mG_5,\mE_6,\mF_6 \}$ respectively. 
\begin{figure}[H]
    \centering
\ytableausetup{baseline}
  \scalebox{0.7}{\ydiagram[*(white!80!green) \bullet]{4,4,4,4} 
\qquad \qquad  \ydiagram[*(white!80!blue) \bullet]
{4+1}
*[*(white!80!green) \bullet]{4,4,4,4} *[*(white!80!red) \bullet] {0,0,0,0,1}
\qquad \qquad      \ydiagram[*(white!80!blue) \bullet]
{4+2}
*[*(white!80!green) \bullet]{4,4,4,4,4} 
\qquad \qquad     \ydiagram[*(white!80!red) \bullet]
{0,0,0,0,0+1,0+1}
*[*(white!80!green) \bullet]{5,5,5,5}} 
    \caption{A graphical illustration of the $4$ Ferrers diagrams $\mA_4,\mG_5,\mE_6,\mF_6$, with their standard forms highlighted.}
   % \label{fig:AGEF}
\end{figure}

\end{example}

In Section \ref{sec: d=3} we will further investigate the existence of MFD codes for the above irreducible diagrams for $d=3$.\\

We end this section by connecting $\poly_d^{(a,b)}(\Z)$, for the special case $|b-a| = d-2 $, to the sequence \href{https://oeis.org/A007747}{A007747} in the OEIS \cite{oeis} counting the number of score sequences in a chess tournament.
Considering, without loss of generality, the case $a-b = d-2 $, the set $\poly_d^{(2d-2,d)}(\Z)$ is given by
$(x_1,\ldots,x_{d-2},0,\ldots,0)\in \Z^{2d-4}$ satisfying the constraints
     \begin{equation}\label{eq:poly chess x}\begin{cases} x_j-x_{j+1}\geq 0  &  j \in [d-3]\\
x_j \geq 0 &  j \in [d-2]\\
		\sum\limits_{i = 0}^{j-2} x_{d-2-i}  \geq j(j-1) \;\; & j \in [d-2] \\
\sum\limits_{i=1}^{d-2}x_i = (d-1)(d-2).
 \end{cases}
 \end{equation}
Let us now define $p_j := \left(\sum\limits_{i = 0}^{j-2} x_{d-2-i}\right)  - j(j-1) = \left(\sum\limits_{i = 2}^{j} x_{d-i}\right)  - j(j-1) \geq 0$  for $j \in [d-1]$. Then $p_{j+1}-p_{j} = x_{d-j-1}-2j$, and thus
\[
2p_{j}-(p_{j+1}-p_{j-1}) = x_{d-j}-x_{d-j-1}+2j-2(j-1) \leq 2.\\
\]
Therefore the tuples $(p_1,\ldots,p_{d-1})$ satisfy the constraints
\begin{equation}\label{eq:poly chess p}\begin{cases} 
2p_{j}-(p_{j+1}+p_{j-1}) \leq 2 &  j \in \{2,3,\ldots,d-2\}\\
		p_j \geq 0 \;\; & j \in [d-1] \\
p_1 = 0\\
p_{d-1} = 0.
 \end{cases}
 \end{equation}
The equations $p_j = \left(\sum\limits_{i = 0}^{j-2} x_{d-2-i}\right)  - j(j-1)$ for all $j$ imply a unimodular map between the polytope of tuples $(x_1, \ldots, x_{d-2}, 0) \in \R^{d-1}$ satisfying \eqref{eq:poly chess x}, and the polytope of tuples $(p_1, \ldots, p_{d-1}) \in \R^{d-1}$ satisfying \eqref{eq:poly chess p}. This unimodular map implies a bijection between the integer points of these polytopes. 
The sequence \href{https://oeis.org/A007747}{A007747} is defined as the number of tuples $(p_1,\ldots,p_{d-1}) \in \Z^{d-1}$ satisfying the system of equations \eqref{eq:poly chess p}, for $d\in\mathbb N$. For the reasons above, such sequence coincides with $\{|\poly_d^{(2d-2,d)}(\Z)|\}_{d\in\mathbb N}$. 
Since there does not yet seem to exist a closed-form formula for this sequence, we do not expect an easily obtainable closed-form formula for general $(a,b)$ for the number of points in $\poly_d^{(a,b)}(\Z)$ in terms of $d$.

\subsection{Integrality and vertex description of $\poly_d$}\label{subsec: integral poly}

In this section, we focus on the polytope $\poly_d$, for $d\geq 3$. We start by giving the following definitions, and refer the reader to \cite{ziegler1993lectures} for the general theory of polytopes.

\begin{definition}
Let $\wp\subset \R^n$ be an $H$-polytope defined by 
\[
\wp = \{x \in \R^n \mid H x\geq b\}
\]
for some $H \in \R^{m \times n}$ and $b \in \R^n$ with $m \geq n$. Then the set $\mathcal{V}(\wp)$ of \textbf{vertices} of $\wp$ is the set of points $v \in \wp$ satisfying
\[
H'v = b
\]
for some invertible submatrix $H' \in \R^{n \times n}$ of $H$.
\end{definition}

 \begin{definition}
 	Two polytopes $\wp_1 \subset \R^n$, $\wp_2 \subset \R^m$ are called \textbf{combinatorially isomorphic} or  \textbf{combinatorially equivalent} if they have isomorphic face lattices, i.e. there is a bijection between 
    $\mathcal{V}(\wp_1)$ and $\mathcal{V}(\wp_2)$, in such a way that the vertex sets of the faces of $\wp_1$ correspond via this bijection to the vertex sets of the faces of $\wp_2$, and induces an inclusion- and dimension-preserving bijection between the faces of the polytopes.
 	
 	The polytopes $\wp_1$ and $\wp_2$ are called \textbf{affinely/linearly/unimodularly isomorphic} if there is an affine/linear/unimodular map $\R^n \to \R^m$ that is a bijection between the points of $\wp_1$ and $\wp_2$. Here, a unimodular map is a linear map in $\operatorname{GL}_n(\Z)$ (so $n=m$), thus preserving $\Z^n \subset \R^n$.
 \end{definition}

\begin{definition}
   An \textbf{integral polytope} in $\R^n$ is a convex polytope whose vertices belong to $\Z^n$. In other words, an integral polytope coincides with the convex hull of points in $\Z^n$. 
\end{definition}

The importance of integral polytopes relies on the fact that some of its properties can be studied via Ehrhart theory: for instance, their volume,  number of vertices and number of integer points, are encoded by its Ehrhart polynomial; see e.g. \cite{beck2007computing, rehberg2025}. We state here our main result about the polytope $\poly_d$. 
\begin{theorem}\label{thm:main_poly}
    Let $d\geq 3$. The polytope $\poly_d$ is an integral polytope with $3^{d-2}$ vertices.
\end{theorem}

The rest of the section will be dedicated to prove Theorem \ref{thm:main_poly}. We start by defining some auxiliary matrices and vectors, together with their properties, which will be useful for the proof.

\begin{definition}\label{def:Ak_Bk_bk} Let $k\in \N$. We define the following two integer-valued $k\times k$ matrices
\[
A_k := \begin{pmatrix}
2 & -1 & 0 & \cdots & 0 \\
-1 & 2 & -1 & \ddots & 0 \\
0 & -1 & 2 & \ddots & 0 \\
\vdots & \ddots & \ddots & \ddots & -1 \\
0 & \cdots & 0 & -1 & 2 \\
\end{pmatrix}, \qquad B_k:=\begin{pmatrix}
     i(k+1-j)
 \end{pmatrix}_{i,j\in[k]},
\]
and the integer-valued vector
\[b_{k}:=\begin{pmatrix}
        1\cdot k\\
    2\cdot(k-1)\\
    3\cdot(k-2)\\
    \vdots\\
    (k-1)\cdot 2\\
    k\cdot 1
\end{pmatrix}\in\Z^k.\]
\end{definition}

We start by highlighting some useful properties of the above mentioned objects that are easily proven by elementary calculations.

\begin{lemma}\label{lem:AkBkbk}
  Let $k\in \N$, and let $A_k,B_k,b_k$ be as in Definition \ref{def:Ak_Bk_bk}. The following hold.
\begin{enumerate}[label = (\arabic*)]
    \item $\det(A_k) = k+1$;
    \item $B_kA_k=(k+1)I_k=A_kB_k$;
    \item $B_k \scalebox{0.8}{$\begin{pmatrix}
    2 \\
\vdots\\2
\end{pmatrix}$} = (k+1)b_k$;\vspace{-0.3cm}
\item $A_k b_k =  \scalebox{0.8}{$\begin{pmatrix}
    2 \\
\vdots\\2
\end{pmatrix}$}$;\vspace{-0.1cm}
\item $\begin{pmatrix}
    1 & 2 &\cdots & k 
\end{pmatrix}A_k = \begin{pmatrix}
    0 &\cdots & 0 & k+1 
\end{pmatrix}$.
\end{enumerate}
\end{lemma}

Let $d\geq 3$ and let us fix, from now on, $k:=d-2$. Recall that, with this notation, the polytope $\poly_d=\poly_{k+2}$ is the set of $(x_1,\ldots,x_{k},y_1,\ldots,y_{k},z)\in\R^{2k+1}$ satisfying the following system of inequalities.

\begin{equation}\label{eq:description_poly}
\text{}\poly_d: \hspace{-0cm}{\small\left(
\begin{array}{ccccc|ccccc|c}
1 & -1 & 0 & \cdots & 0 &  &  &  &  &  & 0 \\
0 & 1 & -1 & \ddots & \vdots &  &  &  &  &  & 0 \\
0 & 0 & 1 & \ddots & 0 &  &  & \scalebox{1.2}{{0}} &  &  & 0 \\
\vdots & \vdots & \ddots & \ddots & -1 &  &  &  &  &  & \vdots  \\
0 & 0 & \cdots & 0 & 1 &  &  &  &  &  &  0 \\
\hline
 &  &  &  &  & 0 & \cdots & 0 & 0 & 1 & 0 \\
 &  &  &  &  & 0 & \cdots & 0 & 1 & -1 & 0 \\
 &  & \scalebox{1.2}{{0}} &  &  & \vdots & \iddots & 1 & -1 & 0 & 0 \\
 &  &  &  &  & 0 & \iddots & \iddots & \iddots & \vdots & \vdots \\
 &  &  &  &  & 1 & -1 & 0 & \cdots & 0 & 0 \\
\hline
0 & 1 & 1 & \cdots & 1 & -1 & -1 & -1 & \cdots & -1 & k \\
0 & 0 & 1 & \ddots & \vdots & -1 & -1 & \iddots & \iddots & 0 & k-1 \\
0 & 0 & 0 & \ddots & 1 & -1 & \iddots & \iddots & \iddots & 0 & k-2 \\
\vdots & \vdots & \ddots & \ddots & 1 & \vdots & \iddots & \iddots & \iddots & \vdots & \vdots \\
0 & 0 & \cdots & 0 & 0 & -1 & 0 & 0 & \cdots & 0 & 1 \\
\hline
-1 & -1 & \cdots & -1 & -1 & 1 & 1 & 1 & \cdots & 1 & -(k+1) \\
1 & 1 & \cdots & 1 & 1 & -1 & -1 & -1 & \cdots & -1 & k+1 
\end{array} 
\right)
\left(
\begin{array}{c}
    x_1\\
    x_2\\
    \vdots\\
    x_{k}\\
    \hline
    y_1\\
    y_2\\
    \vdots\\
    y_{k}\\
    \hline
    z
\end{array}
\right) 
\geq 
\left(
\begin{array}{c}
    0\\
    0\\
    0\\
    \vdots\\
    0\\
    \hline
    0\\
    0\\
    0\\
    \vdots\\
    0\\
    \hline
    \\
    \\
    \\
    -b_{k}\\
    \\
    \\
    \\
    \hline
    0\\
    0\\
\end{array}
\right) }
\end{equation}

To prove Theorem \ref{thm:main_poly}, we first apply a unimodular map (i.e. a linear map in $\operatorname{GL}_{2k+1}(\Z)$) to the polytope $\poly_d$ and the corresponding system of inequalities.
Let $U \in \operatorname{GL}_{2k+1}(\Z)$ be the unimodular map  
\[
U := \left( 
\begin{array}{ccc|ccc|c}
1 & \cdots & 1 &  &  &  & 0 \\
 & \ddots & \vdots &  & \scalebox{1.2}{{0}} &  & \vdots \\
0 &  & 1 &  &  &  & 0 \\
\hline
 &  &  & 1 & \cdots & 1 & 0 \\
 & \scalebox{1.2}{{0}} &  & \vdots & \iddots &  & \vdots \\
 &  &  & 1 &  & 0 & 0 \\
 \hline
0 & \cdots & 0 & 0 & \cdots & 0 & 1 
\end{array}\right) .
\]
The transformed polytope $\poly_d'  := U^{-1}(\poly_d) = \{U^{-1}x \mid x \in \poly_d\}$ is then the solution set of the new system of inequalities \eqref{eq: new system after U}, obtained by multiplying the matrix in system \eqref{eq:description_poly} with $U$ on the right.

The map $U^{-1}$ preserves the integer points of $\poly_d$, in particular the vertices with integer coordinates. This is stated more precisely in the following lemma (see also \cite{beck2007computing, rehberg2025}).

\begin{lemma}
   The polytopes $\poly_d$ and $\poly_d'$ are unimodularly isomorphic. In particular:
   \begin{enumerate}[label = (\arabic*)]
       \item $\poly_d$ and $\poly_d'$ are linearly and combinatorially isomorphic, where $U$ induces an isomorphism between face lattices and thus a bijection between vertices;
       \item  $U$ induces a bijection $\poly_d'(\Z) \to \poly_d(\Z)$ between integer points of the polytopes;
       \item $\poly_d$ is an integral polytope if and only if $\poly_d'$ is.
   \end{enumerate}
\end{lemma}

\begin{equation}\label{eq: new system after U}
\text{} \hspace{-1.4cm}\poly_d':
{\scriptsize\left(
\begin{array}{ccccc|ccccc|c}
 &  &  &  &  &  &  &  &  &  & 0 \\
 &  &  &  &  &  &  &  &  &  & 0 \\
 &  & \multicolumn{2}{c}{{I_{k}}}   &  &  &   \multicolumn{2}{c}{\scalebox{1.2}{{0}}}  & &  & 0 \\
 &  &  &  &  &  &  &  &  &  & \vdots  \\
 &  &  &  &  &  &  &  &  &  &  0 \\
\hline
 &  &  &  &  &  &  &  &  &  & 0\\
 &  &  &  &  &  &  &  &   &  & 0\\
  & & \multicolumn{2}{c}{\scalebox{1.2}{{0}}}   &  &  &  \multicolumn{2}{c}{I_{k}} &   &  & 0\\
 & & &  &  &  &   & &  &  & \vdots\\
 &  &  &  &  &  &  &  &   &  & 0 \\
\hline
0 & 1 & 2 & \cdots & k-1 & -k & -(k-1) & \cdots & -2 & -1 & k\\
 0& 0 & \ddots & \ddots & \vdots & -(k-1) & -(k-1) & \ddots & -2  & -1 & k-1\\
 \vdots & & \ddots& \ddots & 2 & \vdots & \ddots & \ddots & \vdots & \vdots & \vdots\\
\vdots & & & 0 & 1 & -2 & -2  & \cdots& -2 & -1 & 2\\
0 & \cdots\phantom{\vdots} & \cdots & 0 & 0 & -1 & -1 & \cdots &  -1 & -1 & 1 \\
\hline
-1 & -2 & \cdots & -(k-1) & -k & k & k-1 & \cdots & 2 & 1 & -(k+1) \\
1 & 2 & \cdots & k-1 & k & -k & -(k-1)  & \cdots & -2 & -1 & k+1 
\end{array}\right) \left(
\begin{array}{c}
    x_1\\
    x_2\\
    \vdots\\
    x_{k}\\
    \hline
    y_1\\
    y_2\\
    \vdots\\
    y_{k}\\
    \hline
    z
\end{array}
\right) 
\geq 
\left(
\begin{array}{c}
    0\\
    0\\
    0\\
    \vdots\\
    0\\
    \hline
    0\\
    0\\
    0\\
    \vdots\\
    0\\
    \hline
     \\
    \\
    \\
    -b_{k}\\
    \\
    \\
    \\
    \hline
    0\\
    0\\
\end{array}
\right) }
\end{equation}

Next, we will perform several row operations to system \eqref{eq: new system after U}. The operations that preserve the solution set of the system of inequalities are
\begin{itemize}
    \item Scaling a row by a positive real number;
    \item Adding (a multiple of) a row corresponding to an equality to another row.
\end{itemize}
In system \eqref{eq: new system after U}, the bottom 2 rows correspond to the equality
\begin{equation}\label{eq: equality bottom rows}
\begin{pmatrix}
    1 & \cdots & k \mid -k &\cdots & -1 \mid k+1
\end{pmatrix}
\begin{pmatrix}
    x_1 & \cdots  &x_k \mid y_1 &\cdots & y_k \mid z
\end{pmatrix}^\top = 0.
\end{equation}

For $i \in [k]$, we multiply the $(2k + i)$-th row by the scalar $(k+1)$ and add $(k-i+1)$ times the equality \eqref{eq: equality bottom rows} to it. Thus, the polytope $\poly_d'$ can be equivalently described by the following system of inequalities.

\begin{equation}
\text{}\hspace{-0.4cm}\poly_d': {\small\left(\begin{array}{cccc|cccc|c}
& & & & & & & & 0\\
& & & & & & & & 0\\
& & I_{k} & & & 0 & & & 0\\
& & & & & & & & \vdots\\
& & & & & & & & 0\\
\hline 
& & & & & & & & 0\\
& & & & & & & & 0\\
& & 0 & & & I_{k} & & & 0\\
& & & & & & & & \vdots\\
& & & & & & & & 0\\
\hline
& & & & & & & & 0\\
& & & & & & & & 0\\
& & -B_{k} & & & -B_{k} & & & 0\\
& & & & & & & & \vdots\\
& & & & & & & & 0\\
\hline
-1 & -2 & \cdots & -k & k & k-1 & \cdots & 1 & -(k+1) \\
1 & 2 & \cdots & k & -k & -(k-1) & \cdots & -1 & k+1 \\
\end{array}\right)
\left(
\begin{array}{c}
    x_1\\
    x_2\\
    \vdots\\
    x_{k}\\
    \hline
    y_1\\
    y_2\\
    \vdots\\
    y_{k}\\
    \hline
    z
\end{array}
\right) 
\geq 
\left(
\begin{array}{c}
    0\\
    0\\
    0\\
    \vdots\\
    0\\
    \hline
    0\\
    0\\
    0\\
    \vdots\\
    0\\
    \hline
     \\
    \\
    -(k+1)b_{k}\\
    \\
    \\
    \hline
    0\\
    0\\
\end{array}
\right) }
\end{equation}

\newpage

\begin{lemma}\label{lem:transfer_poly_properties}
Let $m,n$ be two natural numbers and let $L$ be a linear map $L:\R^m \to \R^{n}$. Consider the injective linear map  $E:= \operatorname{Id}_m \oplus L : \R^m \to \R^m \oplus \R^{n} = \R^{m+n}$, with $\operatorname{Id}_m:\R^m \to \R^{m}$ the identity map. Then for any polytope $\wp \subset  \R^m$:
\begin{enumerate}[label = (\arabic*)]
    \item $\wp$ is linearly isomorphic to $E(\wp)$;
    \item $E(\mathcal V(\wp))=\mathcal V(E(\wp))$;
    \item If $E(\wp)$ is integral, then so is $\wp$;
    \item If $\wp$ is integral, then $E(\wp)$ is integral if and only if $L(\mathcal{V}(\wp)) \subset \Z^{n}$. 
\end{enumerate}
\end{lemma}

We will define two embeddings of the above form:\\
$$ E_1=\operatorname{Id}_{2k} \oplus L_1 :  \R^{2k} \longrightarrow  \R^{2k+1}, \qquad E_2=\operatorname{Id}_{2k} \oplus L_2 :  \R^{2k} \longrightarrow  \R^{3k},$$ 
with 
$$\begin{array}{rccc}
      L_1 : & \R^{2k} &\longrightarrow & \R \\
     & (x_1,\ldots, x_k, \ y_1, \ldots, y_k)& \longmapsto & 
     \dfrac{1}{k+1}\left(\displaystyle\sum\limits_{i=1}^k i (y_{k-i+1} - x_i) \right),
\end{array}    $$
\[
\begin{array}{rccc}
     L_2 : & \R^{2k} &\longrightarrow  & \R^{k}  \\
    & (x_1, \ldots, x_k, \ y_1,\ldots, y_k) & \longmapsto &\dfrac{1}{k+1}(2-x_1-y_1,\ldots,2-x_k-y_k)  B_k^\top .
\end{array} 
\]
We observe that $\poly_d'\subset\operatorname{Im}(E_1)$ and define the polytope $$\spoly_d:=E_1^{-1}(\poly_d')\subset \R^{2k},$$ 
so that $E_1(\spoly_d) = \poly_d'$. Moreover, we define the polytope $\lpoly_d$ as the image of $\spoly_d$ via the map $E_2$, that is,
\[
\lpoly_d:=E_2(\spoly_d) \subset \R^{3k}.
\]
In summary, the linearly isomorphic polytopes $\poly_d, \poly_d', \spoly_d, \lpoly_d$ are related to each other as follows:
\begin{center}
\begin{tikzcd}
& & \poly_d' \subset \R^{2k+1} \arrow[r, "U"] & \poly_d \subset \R^{2k+1} \\
\spoly_d \subset \R^{2k} \arrow[rru, "E_1"] \arrow[rrd, "E_2"'] &  & &\\
 &  & \lpoly_d \subset \R^{3k} &
\end{tikzcd}
\end{center}

Our strategy will be the following. We derive the desired properties on the polytope $\lpoly_d$, and then, we deduce the same properties on $\poly_d'$, and hence on $\poly_d$. This will be done with the help of Lemma \ref{lem:transfer_poly_properties}, which will allow us to first transfer these properties to the polytope $\spoly_d$ and then to $\poly_d'$.

The polytope $\spoly_d$ is described by the inequalities

\begin{equation}\label{eq:spoly_d}
\text{}\hspace{-0cm}\spoly_d: {\left(\begin{array}{ccc|ccc}
	& & & & &  \\
	& I_{k} & & & 0 & \\
	& & & & &   \\
	\hline 
	& & & & &   \\
	& 0 & & & I_{k} &  \\
	& & & & &   \\
	\hline
	& & & & &   \\
	& -B_k & & & -B_{k} &  \\
	& & & & &   \\
	\end{array}\right)
	\left(
	\begin{array}{c}
	x_1\\
	x_2\\
	\vdots\\
	x_{k}\\
	\hline
	y_1\\
	y_2\\
	\vdots\\
	y_{k} 
	\end{array}
	\right) 
	\geq 
	\left(
	\begin{array}{c}
	0\\
	\vdots\\
	0\\
	\hline
	0\\
	\vdots\\
	0\\
	\hline
	\\
	-(k+1)b_k\\
	\text{}
	\end{array}
	\right) }
\end{equation}

and $\lpoly_d$ is described by

\begin{equation}
\text{}\hspace{-0cm}\lpoly_d: {\left(\begin{array}{ccc|ccc|ccc}
	& & & & & & & & \\
	& I_{k} & & & 0 & & & 0 &\\
	& & & & & \\
	\hline 
	& & & & & \\
	& 0 & & & I_{k} & & & 0 & \\
	& & & & & \\
	\hline
	& & & & & \\
	& -B_k & & & -B_k & & & 0 & \\
	& & & & & \\
	\hline
	& & & & & \\
	& I_k & & & I_k  & & & A_k & \\
	& & & & & \\
	\hline
	& & & & & \\
	& -I_k & & & -I_k  & & & -A_k & \\
	& & & & & 
	\end{array}\right)
	\left(
	\begin{array}{c}
	x_1\\
	x_2\\
	\vdots\\
	x_{k}\\
	\hline
	y_1\\
	y_2\\
	\vdots\\
	y_{k} \\
	\hline
	z_1\\
	z_2\\
	\vdots\\
	z_{k}
	\end{array}
	\right) 
	\geq 
	\left(
	\begin{array}{c}
	0\\
	\vdots\\
	0\\
	\hline
	0\\
	\vdots\\
	0\\
	\hline
	\\
-(k+1)b_k\\
\text{}\\
	\hline
	2\\
	\vdots\\
	2\\
	\hline
	-2\\
	\vdots\\
	-2
	\end{array}
	\right). }
\end{equation}

By adding rows corresponding to equalities in the fourth block row of to the third block row, and using the equations in Lemma \ref{lem:AkBkbk}, we obtain the equivalent representation of $\lpoly_d$

\begin{equation}\label{eq:lpoly_d}
\text{}\hspace{-0cm}\lpoly_d:{\left(\begin{array}{ccc|ccc|ccc}
& & & & & & & & \\
& I_{k} & & & 0 & & & 0 &\\
& & & & & \\
\hline 
& & & & & \\
 & 0 & & & I_{k} & & & 0 & \\
& & & & & \\
\hline
& & & & & \\
& 0 & & & 0 & & & I_{k} & \\
& & & & & \\
\hline
& & & & & \\
& I_k & & & I_k  & & & A_k & \\
& & & & & \\
\hline
& & & & & \\
& -I_k & & & -I_k  & & & -A_k & \\
& & & & & 
\end{array}\right)
\left(
\begin{array}{c}
    x_1\\
    x_2\\
    \vdots\\
    x_{k}\\
    \hline
    y_1\\
    y_2\\
    \vdots\\
    y_{k} \\
        \hline
    z_1\\
    z_2\\
    \vdots\\
    z_{k}
\end{array}
\right) 
\geq 
\left(
\begin{array}{c}
    0\\
    \vdots\\
    0\\
    \hline
    0\\
    \vdots\\
    0\\
    \hline
    0\\
    \vdots\\
    0\\
    \hline
        2\\
    \vdots\\
    2\\
        \hline
        -2\\
    \vdots\\
    -2
\end{array}
\right).}
\end{equation}

\newpage

\begin{proposition}\label{prop;lpoly_integral}
    The set of vertices $\mathcal V(\lpoly_d)$ of the polytope $\lpoly_d$ consists of all the  vectors of the form $(x_1,\ldots,x_k,y_1,\ldots,y_k,z_1,\ldots,z_k)$ such that for each $i\in[k]$ exactly two out of $x_i,y_i$ and $z_i$ are equal to $0$, and
    \begin{equation}\label{eq:equality_lpoly}\left(\begin{array}{ccc|ccc|ccc}
& & & & & \\
& I_k & & & I_k  & & & A_k & \\
& & & & & 
\end{array}\right)
\left(
\begin{array}{c}
    x_1\\
    \vdots\\
    x_{k}\\
    \hline
    y_1\\
    \vdots\\
    y_{k} \\
        \hline
    z_1\\
    \vdots\\
    z_{k}
\end{array}
\right) 
=
\left(
\begin{array}{c}
        2\\
    \vdots\\
    2\\
\end{array}
\right).  
\end{equation}
    Moreover, $\lpoly_d$ is integral and has $3^k=3^{d-2}$ vertices.
\end{proposition}

\begin{proof}
   First, observe that a vertex of $\lpoly_d$ must necessarily have at least $2k$ entries equal to $0$, since it must satisfy $3k$ of the inequalities in \eqref{eq:lpoly_d} with equality, $k$ of which are given by the last equalities. The remaining $2k$ equalities have to be chosen from the first $3k$ inequalities, which then necessarily impose at least $2k$ entries to be $0$.
   
    On the other hand, we can show that most $2k$ entries of a vertex are $0$: let $v=(x \, |\, y \,|\, z)\in \mathcal V(\lpoly_d)$ with $x,y,z\in\R^k$ and let us define
    \begin{align*}
        \mathcal X(v)&=\{i \in[k] \,:\, x_i=0\},\\
        \mathcal Y(v)&=\{i \in[k] \,:\, y_i=0\},\\
        \mathcal Z(v)&=\{i \in[k] \,:\, z_i=0\}.
    \end{align*}
    Now we claim that $\mathcal X(v)\cap \mathcal Y(v)\cap\mathcal Z(v)=\emptyset$. Indeed, if $i\in \mathcal X(v)\cap \mathcal Y(v)\cap\mathcal Z(v)$
    then the $i$-th equality of the last block becomes
    $$x_i+y_i-z_{i-1}+2z_i-z_{i+1} = 2  \quad \Leftrightarrow \quad -z_{i-1}-z_{i+1}=2,$$
     where we consider by convention $z_0=0$ if $i=1$ and $z_{k+1}=0$ if $i=k$.  This implies that at least one between $z_{i-1}$ and $z_{i+1}$ is strictly negative, which is incompatible with the system \eqref{eq:lpoly_d}.  Thus, $\mathcal X(v)\cap \mathcal Y(v)\cap\mathcal Z(v)=\emptyset$ and hence, we get $|\mathcal X(v)|+|\mathcal Y(v)|+|\mathcal Z(v)|\le 2k$, which means that at most $2k$ entries of $v$ are $0$.
    
    Together with the first observation, we then obtain that a vertex of $\lpoly_d$ must have exactly $2k$ entries equal to $0$. Moreover, the fact that $\mathcal X(v)\cap \mathcal Y(v)\cap\mathcal Z(v)=\emptyset$ implies that for each $i\in[k]$ we must have exactly two between $x_i,y_i$ and $z_i$ equal to $0$. 
    
    It remains to show that, for each of these possible choices, we obtain a vertex of $\lpoly_d$. This happens if and only if there is a unique vector defined by these $2k$ entries chosen to be $0$ together with the $k$ equalities of \eqref{eq:equality_lpoly}, and this vector belongs to $\lpoly_d$, that is, the remaining $k$ nonzero entries are positive. Formally, consider $\mathcal X,\mathcal Y,\mathcal Z\subseteq [k]$ such that, for every $i \in [k]$,
    $$|\{i\}\cap\mathcal X|+|\{i\}\cap\mathcal Y|+|\{i\}\cap\mathcal Z|=2.$$
    This is equivalent to say that $$ [k]= \mathcal X^c\sqcup \mathcal Y^c\sqcup \mathcal Z^c,$$
    where $\cdot^c$ denotes the complement in $[k]$. Let 
   $v=( x\,|\,y\,|\,z)\in \R^{3k}$ be a vector such that
   $$x_i=0  \;\;\forall \,i \in \mathcal X, \; y_i=0  \;\;\forall \,i \in \mathcal Y, \; z_i=0  \;\;\forall \,i \in \mathcal Z,  $$
   and satisfying \eqref{eq:equality_lpoly}.
    If $i\in \mathcal X$, then $x_i=0$, while if $i \notin \mathcal X$, then $i \in \mathcal Y$ and $i \in \mathcal Z$, that is, $y_i=0$ and $z_i=0$, and, combining them with the $i$th equation in \eqref{eq:equality_lpoly}, we obtain
 \begin{equation}\label{eq:aux_xi} x_i=2+z_{i-1}+z_{i+1}.
 \end{equation}
  If $i \in \mathcal Y$, then $y_i=0$, while the same argument as above shows that, if $i \notin \mathcal Y$, then
  \begin{equation}\label{eq:aux_yi}y_i=2+z_{i-1}+z_{i+1}.\end{equation}
   Thus, we only need to show that there is a unique solution in the $z_i$'s, and that it is nonnegative, so that, using \eqref{eq:aux_xi} and \eqref{eq:aux_yi}, we automatically get that the $x_i$'s and the $y_i$'s are unique and nonnegative. To do so, we will only use the equations of \eqref{eq:equality_lpoly} corresponding to the rows indexed by $\mathcal Z^c$. Let us write $\mathcal Z^c$ as union of maximal  subsets of consecutive integers, that is, let $i_1,\ldots,i_r,j_1,\ldots,j_r \in[k]$ be such that  $i_\ell<j_\ell$ for every $\ell\in[r]$, $j_\ell<i_{\ell+1}-1$ for every $\ell \in[r-1]$, and
   $$\mathcal Z^c=\{i_1,\ldots,j_1\}\cup \cdots \cup \{i_r,\ldots,j_r\}.$$
    For each $\ell \in [r]$ let $\mathcal I_\ell:=\{i_\ell,\ldots,j_\ell\}$ and denote by $s_\ell$ its cardinality, that is, $s_{\ell}=j_\ell-i_\ell+1$. Observe that, if $i \in \mathcal I_\ell$, then $i\in \mathcal X\cap \mathcal Y$, and hence, $x_i=y_i=0$. Moreover, by definition of $\mathcal I_\ell$, $i_\ell-1, j_\ell+1\notin \mathcal Z^c$, thus $z_{i_\ell-1}=z_{j_\ell+1}=0$ (where by convention we put $z_0=z_{k+1}=0$). Combining these equalities with the equations of \eqref{eq:equality_lpoly} corresponding to the rows indexed by $\mathcal I_\ell$, we obtain the linear system
    $$A_{s_{\ell}}\begin{pmatrix}
        z_{i_\ell} \\ \vdots \\ z_{j_\ell}
    \end{pmatrix}=\begin{pmatrix}
        2 \\ \vdots \\ 2
    \end{pmatrix}.$$
    Since by Lemma \ref{lem:AkBkbk}(1) $\det(A_{s_\ell})=s_{\ell}+1\neq 0$, this linear system has a unique solution, and, always by Lemma \ref{lem:AkBkbk}(4), such a solution is 
    \begin{equation}\label{eq:aux_zi} 
    (z_{i_\ell} \, \cdots \, z_{j_\ell})=b_{s_\ell}^\top,\end{equation} whose entries are all positive. This shows the first part of the statement and the fact that the number of vertices is exactly $3^{k}=3^{d-2}$. 
    
    The integrality of $\lpoly_d$ follows by observing that all the vertices found in the previous argument have either $0$ entries, or are uniquely determined by \eqref{eq:aux_xi}, \eqref{eq:aux_yi} and \eqref{eq:aux_zi}, which clearly implies that they belong to $\Z^{3k}$. 
\end{proof}

\begin{corollary}\label{cor: vertices Kd}
The set of vertices $\mathcal{V}(\spoly_d)$ of the polytope $\spoly_d$ is the image of the injective map

\[
\tau : \{(\mathcal{X},\mathcal{Y}) \mid \mathcal{X}, \mathcal{Y} \subset [k], \ \mathcal{X}\cup\mathcal{Y} = [k]\} \quad \to \quad \Z^{2k}
\]
(with $k = d-2$) given by
\[
(\mathcal{X},\mathcal{Y}) \mapsto (x_1,\ldots,x_k,y_1,\ldots,y_k),
\]
where $(x_i, y_i) = (0,0)$ for all $i \in \mathcal{X} \cap \mathcal{Y}$,  $(x_i, y_i) = (t_i,0)$ for all $i \in \mathcal{X}^c \cap \mathcal{Y}$, and $(x_i, y_i) = (0,t_i)$ for all $i \in \mathcal{X} \cap \mathcal{Y}^c$, with
\[
t_i := 2 + \max\{s \in \N \mid i-j \in   \mathcal{X}\cap\mathcal{Y} \  \ \forall j \in [s]\} + \max\{s' \in \N \mid i+j \in   \mathcal{X}\cap\mathcal{Y} \ \ \forall j \in [s']\}.
\]
Moreover, the sets $\{(\mathcal{X},\mathcal{Y}) \mid \mathcal{X}, \mathcal{Y} \subset [k], \ \mathcal{X}\cup\mathcal{Y} = [k]\}$ has cardinality $3^k = 3^{d-2}$, and so has $\mathcal{V}(\spoly_d)$.

\end{corollary}
\begin{proof}
First, via the map $E_2$ we have the bijection between the vertices of $\lpoly_d$ and $\spoly_d$ given by $\mathcal{V}(\lpoly_d) \to \mathcal{V}(\spoly_d): (x_1,\ldots,x_k, y_1,\ldots,y_k, z_1,\ldots,z_k) \mapsto (x_1,\ldots,x_k, y_1,\ldots,y_k)$. Let us take $(x_1,\ldots,x_k, y_1,\ldots,y_k, z_1,\ldots,z_k) \in \mathcal{V}(\lpoly_d)$. 
By Proposition \ref{prop;lpoly_integral}, for every $i \in [k]$, exactly one of $x_i$, $y_i$, $z_i$ is zero. For the positions $j \in \mathcal{Z}^c = \mathcal{X} \cap \mathcal{Y}$, the values of $z_j \neq 0$ are given by Equation \eqref{eq:aux_zi} and satisfy for each $i \in \mathcal{Z}$ with $i-1 \in \mathcal{Z}^c$:
\[
z_{i-1} = \max\{s \in \N \mid i-j \in   \mathcal{Z}^c\  \ \forall j \in [s]\}.
\]
Similarly for $i \in \mathcal{Z}$ with $i+1 \in \mathcal{Z}^c$, we have
\[
z_{i+1} = \max\{s' \in \N \mid i+j \in   \mathcal{Z}^c\  \ \forall j \in [s']\}.
\]
Together with Equations \eqref{eq:aux_xi} and \eqref{eq:aux_yi}, this shows that $(x_i,y_i) = (t_i,0)$ for $i \in \mathcal{X}^c$ and $(x_i,y_i) = (0,t_i)$ for $i \in \mathcal{Y}^c$, and hence $\mathcal{V}(\spoly_d)$ is contained in the image $\tau$, and by Proposition \ref{prop;lpoly_integral} it is clear that $\mathcal{V}(\spoly_d)$ is the full image.
\end{proof}

We are now ready to prove Theorem \ref{thm:main_poly}.

\begin{proof}[Proof of Theorem \ref{thm:main_poly}] By Proposition \ref{prop;lpoly_integral}, the polytope $\lpoly_d$ is an integral polytope with $3^{k}=3^{d-2}$ vertices. In addition, by definition, we have $\lpoly_d=E_2(\spoly_d)$ and by Lemma \ref{lem:transfer_poly_properties}(3) also $\spoly_d$ is an integral polytope. Using Lemma \ref{lem:transfer_poly_properties}(2) we deduce that  $\spoly_d$ has $3^{d-2}$ vertices, and, if $(x_1,\ldots,x_k,y_1,\ldots,y_k)\in\mathcal V(\spoly_d)$, then it satisfies \eqref{eq:equality_lpoly}, that is,
\begin{equation}\label{eq:xi+yi}
\left(
\begin{array}{c}
    x_1\\
    \vdots\\
    x_{k}\\
\end{array}
\right) +\left(
\begin{array}{c}
    y_1\\
    \vdots\\
    y_{k}\\
\end{array}
\right) 
=
\left(
\begin{array}{c}
        2\\
    \vdots\\
    2\\
\end{array}
\right)  -A_k\left(
\begin{array}{c}
        z_1\\
    \vdots\\
    z_k\\
\end{array}
\right),
\end{equation}
for some $z_1,\ldots,z_k\in\mathbb Z$. Since $\poly_d'=E_1(\spoly_d)$,  by Lemma \ref{lem:transfer_poly_properties}(4) we have that $\poly_d'$ is integral if and only if $L_1(\mathcal V(\spoly_d))$ is an integer. Let $(x_1,\ldots,x_k,y_1,\ldots,y_k)\in\mathcal V(\spoly_d)$, then 
\begin{align*}
L_1(x_1,\ldots,x_k,y_1,\ldots,y_k)&=\dfrac{1}{k+1}\left(\displaystyle\sum\limits_{i=1}^k (k+1-i) y_{i} - \sum\limits_{i=1}^kix_i) \right)\\
&=\sum\limits_{i=1}^k y_{i}-\dfrac{1}{k+1}\left(\displaystyle\sum\limits_{i=1}^k i y_{i} + \sum\limits_{i=1}^kix_i \right),
\end{align*}
Since $(y_1,\ldots,y_k)\in\mathbb Z^k$, we have $L_1(x_1,\ldots,x_k,y_1,\ldots,y_k)\in \mathbb Z$  if and only if   
$$\left(\displaystyle\sum\limits_{i=1}^k i y_{i} + \sum\limits_{i=1}^kix_i \right)\in (k+1)\mathbb Z.$$
On the other hand,
\begin{align*}\left(\displaystyle\sum\limits_{i=1}^k i y_{i} + \sum\limits_{i=1}^kix_i \right)&=\begin{pmatrix}
    1 & 2 &\cdots & k 
\end{pmatrix}\left(\left(
\begin{array}{c}
    x_1\\
    \vdots\\
    x_{k}\\
\end{array}
\right) +\left(
\begin{array}{c}
    y_1\\
    \vdots\\
    y_{k}\\
\end{array}
\right)  \right)  \\
&= \begin{pmatrix}
    1 & 2 &\cdots & k 
\end{pmatrix}\left(\left(
\begin{array}{c}
        2\\
    \vdots\\
    2\\
\end{array}
\right)  -A_k\left(
\begin{array}{c}
        z_1\\
    \vdots\\
    z_k\\
\end{array}
\right) \right)  
 \\
&=k(k+1) -\begin{pmatrix}
    0 &\cdots & 0 & k+1 
\end{pmatrix}\left(
\begin{array}{c}
        z_1\\
    \vdots\\
    z_k\\
\end{array}
\right)\\
&=(k+1)(k-z_k),\end{align*}
where the second equality follows from \eqref{eq:xi+yi}, and the third equality is due to Lemma \ref{lem:AkBkbk}(5). Since $z_k\in\mathbb Z$, this shows that $\poly_d'$ is an integral polytope with $3^{d-2}$ vertices, and thus the same holds also for $\poly_d$, because it is obtained from $\poly_d'$ by a unimodular map. 
\end{proof}

\subsection{Integer points and face structure of $\poly_d$}
With Theorem \ref{thm:main_poly} and Corollary \ref{cor: vertices Kd}, we concluded integrality of $\poly_d$ and characterized its $3^{d-2}$ vertices, with $\mathcal{V}(\poly_d) = U(E_1(\mathcal{V}(\spoly_d)))$. 

\begin{example}
	The vertices of $\spoly_3$ ($k = 1$) are
	\[
	\mathcal{V}(\spoly_3) = \{(0,0), (2,0), (0,2)\}
	\]
	and hence, by applying $E_1$ and $U$, we derive
	\[
	\mathcal{V}(\poly_3') = \{(0,0,0), (2,0,-1), (0,2,1)\} =  \mathcal{V}(\poly_3).
	\]
	The vertices of $\spoly_4$  ($k = 2$) are given by
	\begin{align*}
	\mathcal{V}(\spoly_4) = \{(0,0, 0&,0), (3,0,0,0), (0,3,0,0), (0,0,3,0), (0,0,0,3),\\
	&(2,2,0,0), (2,0,0,2), (0,2,2,0), (0,0,2,2)\}
	\end{align*}
	and hence, by applying $E_1$ and $U$, we derive
	\begin{align*}
	\mathcal{V}(\poly_4') = \{(0,0, 0&,0,0), (3,0,0,0,-1), (0,3,0,0,-2), (0,0,3,0,2), (0,0,0,3,1),\\
	&(2,2,0,0,-2), (2,0,0,2,0), (0,2,2,0,0), (0,0,2,2,2)\}
	\end{align*}
	and
	\begin{align*}
	\mathcal{V}(\poly_4) = \{(0,0, 0&,0,0), (3,0,0,0,-1), (3,3,0,0,-2), (0,0,3,3,2), (0,0,3,0,1),\\
	&(4,2,0,0,-2), (2,0,2,0,0), (2,2,2,2,0), (0,0,4,2,2)\}.
	\end{align*}
\end{example}

Now, recall from Theorem \ref{thm bijection mu with polytope} that we have a natural bijection between the set $\irr_d^\mu(\Z)$ of irreducible diagram pairs $(\mD,d)$ and the integer points $\poly_d(\Z)$.
For $d \leq 8$ we computed with \textsc{SageMath} \cite{sagemath} the number of integer points $|\poly_d(\Z)|$ for $\mu \geq d$.  The code can be found at \url{https://github.com/HBeelooSauerbierCouvee/irreducible-ferrers-diagrams}.

\begin{table}[H]
\centering
     \begin{tabular}{|c||c|}
        \hline
        $d$ &   $|\poly_d(\Z)|$\\
        \hline
        \hline
        3 & 4 \\
        4 & 22\\
        5 & 155 \\
        6 & 1301\\
        7 & 12330 \\
        8 & 127275 \\
        \hline
    \end{tabular}
    \caption{For $d \leq 8$: the number of integer points $|\poly_d(\Z)|$.}
\end{table}

For general $d \geq 3$, we do not have a formula for $|\poly_d(\Z)|$, other than the lower bound $3^{d-2}$ and (trivial)  upper bound $(2d-3)^{2d-3}$ by Corollary \ref{cor: poly d bounded}.  However, since the polytopes $\poly_d$ are integral, we suggest that their integer points $\poly_d(\Z)$ could be further studied with tools from Ehrhart theory (see \cite{beck2007computing, rehberg2025}). 

We can partition the integer points by their last coordinate $z$, giving us the integer points $\poly_d^{(a,a+\Delta)}(\Z)$ (under the correspondence $z = \Delta$) for any $a$ with $a \geq d$ or $a+\Delta \geq d$, and $\min(a,a+\Delta) \geq d-1$. The number of these points for $d \leq 6$ is given in Table \ref{tab: num int points by delta}, and the integer points themselves for $d \leq 4$ are shown in Table \ref{tab: int points by delta}. 

\begin{table}[H]
\centering
     \begin{tabular}{|c||c|c|c|c|c|c|c|c|c||c|}
        \hline
        $d \ \backslash \ \Delta $ & $-4$ & $-3$ & $-2$ & $-1$ & 0 & 1 & 2 & 3 & 4 & Total\\
        \hline
        \hline
        3 & & & & 1 & 2 & 1 & & & & 4\\
        \hline
        4 & & & 2 & 5 & 8 & 5 & 2 & & & 22\\
        \hline
        5 & & 5 & 17 & 34 & 43 & 34 & 17 & 5 & & 155 \\
        \hline
        6 & 16 & 66  &  159 & 257 & 305 & 257 & 159 & 66 & 16 & 1301\\
        \hline
    \end{tabular}
    \caption{For $d\leq 6$: the number of integer points $|\poly_d^{(a,a+\Delta)}(\Z)|$ for any $a$ with $a \geq d$ or $a+\Delta \geq d$, and $\min(a,a+\Delta) \geq d-1$, and row wise totals equal to $|\poly_d(\Z)|$.}
    \label{tab: num int points by delta}
\end{table}

\begin{table}[H]
\centering
     \begin{tabular}{|c||c|c|c|c|c|}
        \hline
        $d \ \backslash \ \Delta$ &  $-2$ &  $-1$ &  $0$ &  $1$ &  $2$\\
        \hline
        \hline
        & &  &  & & \\
        3 & & $\bm{(2,0)}$ & $(1,1)$ & $\bm{(0,2)}$ & \\
        & &  & $\bm{(0,0)}$ & & \\
        & &  &  & & \\
        \hline
        & &  &  & & \\
        4 & $\bm{(4, 2, 0, 0)}$ & $(3, 2, 1, 1)$ & $\bm{(2, 2, 2, 2)}$ & $(1,1,3,2)$ & $\bm{(0, 0, 4, 2)}$\\
        & $\bm{(3, 3, 0, 0)}$ & $(2, 2, 1, 0)$ & $(1, 1, 1, 1)$ & $(1, 0, 2, 2)$& $\bm{(0, 0, 3, 3)}$ \\
        & & $(3, 1, 1, 0)$ & $(2, 1, 2, 1)$ & $(1,0,3,1)$& \\
        & & $(2, 1, 0, 0)$ & $(2, 0, 1, 1)$ & $(0,0,2,1)$& \\
        & & $\bm{(3, 0, 0, 0)}$ & $(1, 1, 2, 0)$ & $\bm{(0,0,3,0)}$& \\
        & &  & $\bm{(2, 0, 2, 0)}$ & & \\
        & &  & $(1, 0, 1, 0)$ & & \\
        & &  & $\bm{(0, 0, 0, 0)}$ & & \\
        & &  &  & & \\
        \hline
    \end{tabular}
    \caption{For $d \leq 4$: the integer points $\poly_d^{(a,a+\Delta)}(\Z)$ for any $a$ with $a \geq d$ or $a+\Delta \geq d$, and $\min(a,a+\Delta) \geq d-1$. The points corresponding to the vertices of $\poly_d$ (without last coordinate $z$) are indicated in bold.}
    \label{tab: int points by delta}
\end{table}

\noindent As an example, the five integer points in $\poly_4^{(a,a+1)}(\Z)$ for $a \geq 3$, together with the corresponding five irreducible diagrams in $\irr_4^{(a,a+1)}$ are shown in Example \ref{ex: d=4 and delta=1}.

\begin{comment}
\begin{table}[H]
\centering
     \begin{tabular}{|c|c|c|}
        \hline
        $d$ &  Ehrhart polynomial & Integer points\\
        \hline
        3 & $(t + 1)(t + 1)
$& 4 \\
        4 &  $\frac{1}{3} (t + 1) (7 t^3 + 14 t^2 + 9 t + 3)$ & 22\\
        5 &  $\frac{1}{90} (t + 1) (694 t^5 + 2042 t^4 + 2348 t^3 + 1357 t^2 + 444 t + 90)$ & 155 \\
        6 & $\frac{1}{1260}(t + 1) (38784 t^7 + 150908 t^6 + 246027 t^5 $& \\
       & $ + 218948 t^4 + 116373 t^3 + 38810 t^2 + 8520 t + 1260)$ & 1301\\
        7 & $\frac{1}{226800}(t + 1) (31359872 t^9 + 152106688 t^8 + 320970557 t^7 + 387074923 t^6 $ & \\ & $+ 294761393 t^5 + 148053007 t^4 + 50126298 t^3 + 11619522 t^2 + 1922940 t + 226800)$ & 12330 \\
         8 & $\frac{1}{...}(t + 1) ...$ & \\ & $...$ & 127275 \\
        \hline
    \end{tabular}
    \caption{Caption}
    \label{tab:placeholder}
\end{table}
\todo{add caption, explain f vector and ehrhart polynomial, remove integral column}   
\end{comment}

In addition to the number of vertices (0-dimensional faces), we can computationally find the number of $i$-dimensional faces $f_i$ of the polytopes $\poly_d$, for $i \geq 0$, and obtain the corresponding $f$-vectors $(f_0, f_1,f_2,\ldots)$.

\begin{table}[H]
    \centering
    \begin{tabular}{|c|c|c|}
        \hline
        $d$ & $f$-vector  $(f_{0}, f_1,\ldots,f_{2d-4})$ of $\poly_d$ &  $\dim(\poly_d) = 2d-4$ \\
        \hline
        3  & (3, 3, 1) & 2   \\
        4  & (9, 18, 15, 6, 1)
 & 4 \\
        5  & (27, 81, 108, 81, 36, 9, 1)
& 6  \\
        6 & (81, 324, 594, 648, 459, 216, 66, 12, 1)
 & 8 \\
 7 & (243, 1215, 2835, 4050, 3915, 2673, 1305, 450, 105, 15, 1)
& 10  \\
        \hline
    \end{tabular}
    \caption{The $f$-vectors of $\poly_d$ for $d = 3,4,\ldots,7$.}
\end{table}

The $f$-vectors coincide with the $f$-vectors of Cartesian products of $d-2$ triangles (2-simplices). In fact, by computation of the face lattices, we verified that the polytopes $\poly_d$ are \textit{combinatorially equivalent} to products of $d-2$ triangles for all $d \leq 7$. This also implies that the $f$-vector of $\poly_d$ is the coefficient vector of $(3+3x+x^2)^{d-2}$, i.e., the $(d-2)$-th power of the face generating function of a triangle. The \textsc{SageMath} code used for the verification can be found at \url{https://github.com/HBeelooSauerbierCouvee/irreducible-ferrers-diagrams}. In addition to the verification for $d \leq 7$, the system of linear inequalities \eqref{eq:lpoly_d} also closely resembles that of a product of triangles of all $d$, and thus we conjecture this combinatorial equivalence for all $d$.

\begin{conjecture}\label{conj: product triangles}
For all $d \geq 3$, the polytope $\poly_d$ is combinatorially equivalent to the Cartesian product of $d-2$ triangles.
\end{conjecture}

\section{The case $d=3$}\label{sec: d=3}

In this section, we restrict ourselves to the case $d=3$. Using Theorem \ref{thm:main}, we will show that there are only four families of irreducible pairs $(\mD,3)$. 

For each $n \in \N$, define the following Ferrers diagrams of order $n$:
\begin{align*}
    \mA_n&=(n,n,\ldots, n)=[n]^2\\
    \mathcal G_n&=(n,n-1,n-1,\ldots,n-1,1)\\
    \mathcal E_n&=(n-1,n-1,\ldots,n-1,1,1) \\
    \mathcal F_n&=(n,n-2,n-2,\ldots n-2,0)
\end{align*}

\begin{figure}[h]
    \centering
\ytableausetup{baseline}
  \scalebox{0.7}{\ydiagram[*(white!80!green) \bullet]{7,7,7,7,7,7,7} 
\qquad \qquad  \ydiagram[*(white!80!blue) \bullet]
{6+1}
*[*(white!80!green) \bullet]{6,6,6,6,6,6} *[*(white!80!red) \bullet] {0,0,0,0,0,0,1}
\qquad \qquad      \ydiagram[*(white!80!blue) \bullet]
{5+2}
*[*(white!80!green) \bullet]{5,5,5,5,5,5} 
\qquad \qquad     \ydiagram[*(white!80!red) \bullet]
{0,0,0,0,0,0+1,0+1}
*[*(white!80!green) \bullet]{6,6,6,6,6}} 
    \caption{A graphical illustration of the $4$ Ferrers diagrams $\mA_7,\mG_7,\mE_7,\mF_7$, with their standard forms highlighted.}
    \label{fig:AGEF}
\end{figure}

\begin{theorem}\label{thm:irreducible_d=3}
    Let $\mD$ be a nonempty Ferrers diagram. Then, the pair $(\mD,3)$ is irreducible if and only if
    $$\mD \in \{\mA_n,\mG_n,  \,:\, n \ge 3\}\cup \{\mE_n,\mF_n : n \ge 4\}.$$
\end{theorem}

\begin{proof}
    Let $\mD$ be a Ferrers diagram such that $(\mD,3)$ is an irreducible pair.
    
    First, assume that $\mD\cap \mB_{n,2}\neq \emptyset$. Then, $(\mD,3)$ is in $(a,b)$-standard form with $a,b\ge d-1$ and $X(\mD,3)$, $Y(\mD,3)$ Ferrers diagrams with only one column. Using Theorem \ref{thm:main}(3), we deduce
    $c_1(Y)-c_1(X)=|Y|-|X|=2(b-a)$ and
    $(b-a+1)-c_1(Y)\ge 0$, which combined give
    $a-b+1-c_1(X)\ge 0$. Thus,
    $$\begin{cases}
        0\le c_1(X)\le a-b+1 \\
        0 \le c_1(Y) \le b-a+1
    \end{cases}$$
    In particular, we deduce $|a-b|\le 1$. We have now three cases:

    \noindent \textbf{Case I:} If $a=b$, then $c_1(X)=c_1(Y)\le 1$. and hence $\mD \in \{\mA_{b},\mG_{b+1}\}$.

    \noindent \textbf{Case II:} If $a=b-1$, then $c_1(Y)=c_1(X)+2$ and $c_1(Y)\leq 2$, which implies $X=\emptyset$ and $c_1(Y)=2$. Hence, $\mD=\mF_{b+1}$.

    \noindent \textbf{Case III:} If $a=b+1$, then we get the adjoint of the previous case, that is $\mD=\mE_{a+1}$.

    Now, assume that $\mD\cap\mB_{n,2}=\emptyset$. Then $(2,2)\notin \mD$, and hence $\nu_1(\mD,3)=0$. This implies that $\nu_{\min}(\mD,3)=0$ and thus, by Lemma \ref{lem:irreducible_empty}, $\mD=\emptyset$, contradicting the hypothesis of $\mD$ being nonempty.

\end{proof}

Note that the ES-conjecture for $\mathcal A_n$ coincides with the existence of MRD codes in $\F^{n\times n}$, which was proved already by Delsarte in 1978 \cite{delsarte1978bilinear}. Furthermore, note that $\mE_n$ and $\mF_n$ are adjoint to each other. Thus, proving the existence of an $[\mE_n,\nu_{\min}(\mE_n,d),d]_{\F}$ MFD code is equivalent to proving the existence of an $[\mF_n,\nu_{\min}(\mF_n,d),d]_{\F}$ MFD code.

\begin{lemma}\label{lem:matrices1-0}
   Let $(\mD,d)$ be an irreducible pair with $\mD\cap \mB_{n,d-1} \neq \emptyset$ in $(a,b)$-standard form. Let $\C$ be an $[\mD,\nu_{\min}(\mD,d),d]_{\F}$ MFD code. Then, the following hold:
    \begin{enumerate}[label = (\arabic*)]
        \item For every $(i,j)\in X(\mD,d)^\top$ there exists $M^{(i,j)}\in \C$ such that 
$$M^{(i,j)}_{a,b}=\begin{cases}
    1 & \mbox{ if } (a,b)=(i,j), \\
    0 & \mbox{ if } (a,b) \in X(\mD,d)^\top\setminus\{(i,j)\}.
\end{cases}$$
        
        \item For every $(i,j)\in Y(\mD,d)$ there exists $M^{(i,j)}\in \C$ such that 
        $$M^{(i,j)}_{a,b}=\begin{cases}
    1 & \mbox{ if } (a,b)=(i,j), \\
    0 & \mbox{ if } (a,b) \in Y(\mD,d)\setminus\{(i,j)\}.
\end{cases}$$
    \end{enumerate}
\end{lemma}

\begin{proof}
    We prove only the first statement, since the second one can be directly obtained via its adjoint Ferrers diagram. 

    Observe that it is enough to show that $X(\mD,d)^\top$ is contained in an \emph{information set}, that is a set $Z\subset \mD$ of cardinality $|Z|=\dim_{\F}(\C)$ such that,  $\dim_\F(\pi_Z(\C))=\dim_\F(\C)$, where $\pi_Z$ is the projection map from $\F^\mD$ to $\F^Z$. 

    Since $(\mD,d)$ is an irreducible pair with $\mD\cap \mB_{n,d-1}\neq \emptyset$, then by Theorem \ref{thm:main}, $\nu_{\min}(\mD,d)=\nu_0(\mD,d)=\nu_{d-1}(\mD,d)$. Thus, $\nu_{d-1}(\mD,d)=\dim_\F(\C)$, which, since $\C$ has minimum rank distance $d$, is also equal to $\dim_\F(\pi_Z(\C))$, where $Z=\mD\cap (\mathbb N \times (\mathbb N\setminus [d-1]))$. In particular, $Z$ is an information set for $\C$ and it contains $X(\mD,d)^\top$.
\end{proof}

\begin{remark}
    With the same exact proof of Lemma \ref{lem:matrices1-0}, one can show that such matrices $M^{(i,j)}$ do exist for every $(i,j) \in \mD\cap (\mathbb N \times (\mathbb N\setminus[d-1]))$. And the same can be done for every $(i,j) \in \mD\cap ((\mathbb N\setminus[d-1])\times \mathbb N )$. Indeed, these two sets are always information sets for an MFD code of distance $d$ over $\mD$, when $(\mD,d)$ is an irreducible pair with $\mD\cap\mB_{n,d-1}\neq \emptyset$. 
\end{remark}

\subsection{The Ferrers diagrams $\mG_n$}

We now exhibit a construction of $[\mG_n,\nu_{\min}(\mG_n,3),3]_{\F}$ MFD codes. First of all, we observe that
$$\nu_{\min}(\mG_n,3)=\nu_0(\mG_n,3)=\nu_{d-1}(\mG_n,3)=(n-2)^2.$$
Let us now consider an $[n-1\times n-1,(n-1)(n-3),3]_{\F}$ MRD code $\mC_0$. Since $[n-1]^2\subseteq \mG_n$, this is also a $[\mG_n,(n-1)(n-3),3]_{\F}$ code. Let us call $\mC_0'$ the code obtained from $\C_0$ by erasing the first row and first column of every matrix. By erasing one row and one column to each element in $\C_0$, the rank decreases at most by $2$, and this means that $\dim_\F(\mC_0')=\dim_\F(\mC_0)$, due to the fact that the minimum rank distance of $\C_0$ is $3$.  Since $$\dim_\F(\C_0')=\dim_\F(\C_0)=(n-1)(n-3)<(n-2)^2=\dim_{\F}(\F^{(n-2)\times(n-2)}),$$
then there exists a matrix $A'=(A'_{i,j})\in \F^{(n-2)\times(n-2)}\setminus\mC_0'$. We now construct a matrix $A=(A_{i,j})$ defined as follows:
$$A_{i,j}=\begin{cases}
  1 & \mbox{ if } (i,j)\in\{(1,n),(n,1)\}, \\
  A'_{i-1,j-1} & \mbox{ if } 2\le i,j \le n-1, \\
  0 & \mbox{ otherwise. }
\end{cases}$$
In other words, the matrix $A$ looks like 

$$   A= \left(
    \begin{array}{ccccc}0 \cellcolor{blue!10} & 0 \cellcolor{blue!10}& \ldots \cellcolor{blue!10} & 0 \cellcolor{blue!10}& 1\cellcolor{blue!10} \\
    \cline{2-5}
     0\cellcolor{blue!10} & \multicolumn{1}{|c}{\phantom{0}\cellcolor{blue!10}} & \multicolumn{1}{c}{\phantom{0} \cellcolor{blue!10}} & \multicolumn{1}{c|}{\phantom{0} \cellcolor{blue!10}} &  \\
     \vdots \cellcolor{blue!10} & \multicolumn{1}{|c}{\phantom{0}\cellcolor{blue!10}} & \multicolumn{1}{c}{A'\cellcolor{blue!10}} & \multicolumn{1}{c|}{\phantom{0}\cellcolor{blue!10}} &  \\ 
         0\cellcolor{blue!10} & \multicolumn{1}{|c}{\phantom{0}\cellcolor{blue!10}} & \multicolumn{1}{c}{\phantom{0}\cellcolor{blue!10}}& \multicolumn{1}{c|}{\phantom{0}\cellcolor{blue!10}} &  \\ \cline{2-4}
      \multicolumn{1}{c|}{1\cellcolor{blue!10}} & & \\ \cline{1-1}
    \end{array}
    \right)$$
%%%

We can prove the existence of $[\mG_n,\nu_{\min}(\mG_n,3),3]_{\F}$ MFD codes over any field $\F$ over which there is a $[n-1\times n-1,(n-1)(n-3),3]_{\F}$ MRD code.

\begin{theorem}\label{thm:diagrams_Gn}
    The code $\C_0\oplus \langle A \rangle_{\F}$ is a $[\mG_n,\nu_{\min}(\mG_n,3),3]_{\F}$ MFD code.
\end{theorem}

\begin{proof}
    We only need to prove that every nonzero matrix $M\in \C_0\oplus \langle A\rangle_{\F}$ has rank at least $3$. Let us write $M=\lambda A+T$, for some $\lambda \in \F$, and $T\in \C_0$. If $\lambda =0$ then $T\in \C_0$ must be nonzero, and by definition of $\C_0$ which is a rank-metric code of minimum rank $3$, we have $\rk(M)=\rk(T)\ge 3$. If $\lambda \neq 0$, then we can consider $\lambda=1$, up to considering the scalar multiple $\lambda^{-1}M$. In the central $(n-2)\times(n-2)$ square block, the matrix $M$ is equal to $A'+T'$, where $T'\in \C_0'$ is the matrix obtained by $T$ after removing the first row and the first column. By the assumption on $A'$, we have that $A' \notin \C_0'$, and thus neither $-A'$ does. In particular, the matrix $A'+T'\neq 0$ and 
    $$M=A+T= \left(
    \begin{array}{ccccc}* \cellcolor{blue!10} & * \cellcolor{blue!10}& \ldots \cellcolor{blue!10} & * \cellcolor{blue!10}& 1\cellcolor{blue!10} \\
    \cline{2-5}
     *\cellcolor{blue!10} & \multicolumn{1}{|c}{\phantom{0}\cellcolor{blue!10}} & \multicolumn{1}{c}{\phantom{0} \cellcolor{blue!10}} & \multicolumn{1}{c|}{\phantom{0} \cellcolor{blue!10}} &  \\
     \vdots \cellcolor{blue!10} & \multicolumn{1}{|c}{\phantom{0}\cellcolor{blue!10}} & \multicolumn{1}{c}{A'+T'\cellcolor{blue!10}} & \multicolumn{1}{c|}{\phantom{0}\cellcolor{blue!10}} &  \\ 
         *\cellcolor{blue!10} & \multicolumn{1}{|c}{\phantom{0}\cellcolor{blue!10}} & \multicolumn{1}{c}{\phantom{0}\cellcolor{blue!10}}& \multicolumn{1}{c|}{\phantom{0}\cellcolor{blue!10}} &  \\ \cline{2-4}
      \multicolumn{1}{c|}{1\cellcolor{blue!10}} & & \\ \cline{1-1}
    \end{array}
    \right).$$
    By the structure of $M$, we have
    $$\rk(M)=2+\rk(A'+T')\geq 3,$$
    which concludes the proof.
\end{proof}

{\color{black}
\begin{remark}
    We must note that a construction for $[\mG_n,(n-2)^2,3]_{\F}$ MFD codes can be obtained also by using \cite[Construction 2]{etzion2016optimal}; see also \cite[Theorem 9]{etzion2016optimal}. Their construction still relies on an auxiliary $[(n-1)\times(n-1),(n-1)(n-3),3]_{\F}$ MRD code, but it does not look equivalent to ours at first sight. Our proof is instead based on \cite[Example II.16]{antrobus2019maximal}, where the authors only exhibit this construction for the diagrams $\mG_4$.
\end{remark}}

\subsection{The Ferrers diagrams $\mE_n$ and $\mF_n$}\label{ssec:XnYn}

We conclude the study of the Etzion-Silberstein conjecture for $d=3$, by analyzing the remaining two families of irreducible pairs in this case. Since $\mE_n^\top=\mF_n$, we can just restrict to study the pairs $(\mE_n,3)$. Moreover, due to Theorem \ref{thm:irreducible_d=3} and the fact that the Etzion-Silberstein conjecture holds over the diagrams $\mA_n$ (proved by Delsarte \cite{delsarte1978bilinear}) and $\mG_n$ with $d=3$ (see Theorem \ref{thm:diagrams_Gn}), the Etzion-Silberstein conjecture for $d=3$ turns out to be equivalent to the Etzion-Silberstein conjecture for the pairs $\{(\mE_{n},d)\,:\, n\ge 4\}$. We can say more, and relate the existence of MFD codes over these diagrams to the existence of  classical MRD codes with special additional features relating to puncturing on a row\footnote{Although the definition of puncturing in the rank-metric setting is a more general concept, by \emph{puncturing} a code on one row we simply indicate the rank-metric code obtained by deleting that row from each matrix in the code.}.

\begin{theorem}\label{thm: MRD conj d=3}Let $\F$ be a finite field. The following are equivalent.
\begin{enumerate}[label = (\arabic*)]
    \item  The Etzion-Silberstein conjecture holds for $d=3$ over any Ferrers diagram $\mD$.
    \item For every $n\ge 3$ there exists an
    $[n\times (n-1),n(n-3),3]_\F$ MRD code $\C$ such that its puncturing on one row is contained in an $[(n-1)\times (n-1),(n-1)(n-2),2]_\F$ MRD code.
\end{enumerate}
   \begin{proof}
       As explained above,  by Theorem \ref{thm:irreducible_d=3}, for proving the Etzion-Silberstein conjecture for $d=3$ it is enough to show it for the irreducible pairs with $d=3$, that are $\{(\mA_n,3), (\mG_n,3)\,:\, n\ge 3\}\cup \{(\mE_n,3), (\mF_n,3)\,:\, n\ge 4\}$. Now, the Etzion-Silberstein conjecture over the full diagrams $\mA_n=[n]^2$ is equivalent to the existence of square MRD codes, which was settled by Delsarte \cite{delsarte1978bilinear}. Also, for the pairs $\{(\mG_n,3)\,:\, n\ge 3\}$, Theorem \ref{thm:diagrams_Gn} ensures that the Etzion-Silberstein conjecture is settled. Since $\mF_n^\top=\mE_n$, the only pairs left are $\{(\mE_n,3)\,:\, n\ge 4\}$.
       Thus, it remains to show that constructing an $[\mE_{n+1},\nu_{\min}(\mE_{n+1},3),3]_{\F}$ MFD code is equivalent to constructing an $[n\times (n-1),n(n-3),3]_\F$ MRD code $\C$ such that its puncturing on one row is contained in an $[(n-1)\times (n-1),(n-1)(n-2),2]_\F$ MRD code.
       {\color{black} This part of the proof follows from the more general result of Theorem \ref{thm:equiv_conj_Ekdr}. which we will see in Section \ref{sec:punct_ext_conjecture}.}

   \end{proof}
\end{theorem}

\section{The puncturing-inclusion MRD conjecture}\label{sec:punct_ext_conjecture}

Let us consider a generalization of the Ferrers diagrams $\mE_n$ and $\mF_n$ defined in Section \ref{ssec:XnYn}. For three positive integers $k,d,r$, define the Ferrers diagrams
\begin{align*}\mE_{k,d,r}&:=(\underbrace{k+r,k+r,\ldots,k+r}_{\text{$k$ times}},\underbrace{r,r,\ldots,r}_{\text{$(d-1)$ times}})\\
\mF_{k,d,r}&:=(\underbrace{k+d-1,\ldots k+d-1}_{\text{$r$ times}},\underbrace{k,k,\ldots,k}_{\text{$k$ times}}),\end{align*}
where we have omitted the $0$ columns.

\begin{figure}[h]
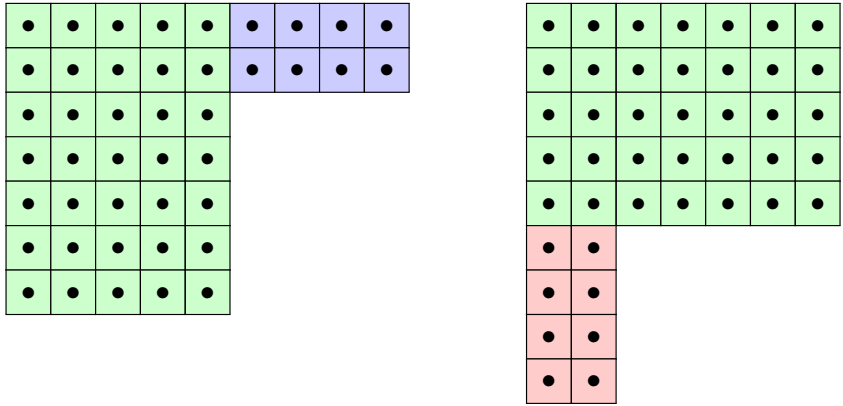

    \centering
\ytableausetup{baseline}
    \ydiagram[*(white!80!blue) \bullet]
{5+4,5+4}
*[*(white!80!green) \bullet]{5,5,5,5,5,5,5} 
\qquad \qquad     \ydiagram[*(white!80!red) \bullet]
{0,0,0,0,0,0+2,0+2,0+2,0+2}
*[*(white!80!green) \bullet]{7,7,7,7,7} 
    \caption{The two diagrams $\mE_{5,5,2}$ and $\mF_{5,5,2}$. The green background represents the $[a]\times[b]$ rectangle of their standard form, while the blue and the red ones are respectively the $X$ and the $Y$ part.}
    \label{fig:secondtry}
\end{figure}

Comparing them with the definition of $\mE_n$ and $\mF_n$ of Section \ref{ssec:XnYn}, we have
$$\mE_{n}=\mE_{n-2,3,1}, \qquad \mF_n=\mF_{n-2,3,1}.$$
Note that also in this case we have that $\mE_{k,d,r}$ and $\mF_{k,d,r}$ are adjoint to each other.

\begin{proposition}
    If $k\ge d$ and $r\le d-2$, then $(\mE_{k,d,r},d)$ and $(\mF_{k,d,r},d)$ are irreducible.
\end{proposition}

\begin{proof} Since $\mF_{k,d,r}^\top=\mE_{k,d,r}$, it is enough to prove the statement for $\mE_{k,d,r}$. 
    Let us compute the values $\nu_i(\mE_{k,d,r},d)$ for $i \in \{0,\ldots,d-1\}$. Direct computations shows that
   \begin{equation}\label{eq:numin_Ekdr}\nu_0(\mE_{k,d,r},d)=\nu_{d-1}(\mE_{k,d,r},d)=k(k+r-d+1)=(k+r)(k-d+1)+r(d-1).\end{equation}
    Moreover, for $i\in[d-1]$ we have
    $$\nu_{d-1-i}(\mE_{k,d,r},d)=\begin{cases}
        (k+r-i)(k-d+1+i)+(d-1)(r-i) & \mbox{ if } 1\le i \le r \\
        (k+r-i)(k-d+1+i) & \mbox{ if } r\le i \le d-1.
    \end{cases}$$
    
    \noindent \textbf{Case I:} If $0\le i\le r$, then we have
    \begin{align*}\nu_{d-1-i}(\mE_{k,d,r},d)&=(k+r-i)(k-d+1+i)+(d-1)(r-i)\\
    &=(k+r)(k-d+1)+r(d-1)+i(r-i)\\
    &=\nu_{d-1}(\mE_{k,d,r})+i(r-i)\\
    &\ge \nu_{d-1}(\mE_{k,d,r},d). 
    \end{align*}

    \noindent \textbf{Case II:} If $r\le i \le d-1$, then
    \begin{align*}\nu_{d-1-i}(\mE_{k,d,r},d)&=(k+r-i)(k-d+1+i)\\
    &=(k+r)(k-d+1)+i(d+r-1-i)\\
    &\ge (k+r)(k-d+1)+r(d-1) \\&=\nu_{d-1}(\mE_{k,d,r},d),
    \end{align*}
    where the last inequality follows from the fact that the map
    $$i \in\{r,\ldots,d-1\}\longmapsto i(d+r-1-i)$$
    is concave, and attains the minimum value in the extremal points $r$ and $d-1$, in which is equal to $r(d-1)$.
    The claim now follows directly from part 2. of the classification result of Theorem \ref{thm:main}.
\end{proof}

\begin{lemma}\label{lem:intersection_MRD}
    Let $\C$ be an $[\mE_{k,r,d}, \nu_{\min}(\mE_{k,d,r},d),d]_{\F}$ MFD code. Then, 
    $$\C \cap \F^{[k+r]\times [k]}$$
    is a $[(k+r)\times k,(k+r)(k-d+1),d]_{\F}$ MRD code.
\end{lemma}

\begin{proof}
    The code $\C \cap \F^{[k+r]\times [k]}$ is a $[(k+r)\times k, K, D]_\F$ code, with $D\geq d$ and $K\geq  \nu_{\min}(\mE_{k,d,r},d)-r(d-1)$. The statement follows by observing that
    $$\nu_{\min}(\mE_{k,d,r}),d)=(k+r)(k-d+1)+r(d-1),$$
    as computed in \eqref{eq:numin_Ekdr}.
\end{proof}

\begin{corollary}\label{cor:E_MFD_MRD}
    Let $\C$ be an $[\mE_{k,r,d}, \nu_{\min}(\mE_{k,d,r}),d),d]_{\F}$ MFD code. Then, 
    there exists a $[(k+r)\times k,(k+r)(k-d+1),d]_{\F}$ MRD code $\C_0$, and matrices $M^{(i,j)}\in \F^{\mE_{k,d,r}}$, for $i\in [r], j \in [d-1]$ with 
    $$(M^{(i,j)})_{a,b}=\begin{cases}
        1 & \mbox{ if } (a,b)=(i,j+k) \\
        0 & \mbox{ if } (a,b)\neq (i,j+k), a \in [r], b \in \{k+1,\ldots k+d-1\} \\
        m^{(i,j)}_{a,b} & \mbox{ 
 otherwise,}
    \end{cases}$$
    such that 
    $$\C=\C_0 \oplus \langle M^{(i,j)}\,:\, i \in [r], j \in [d-1]\rangle_{\F}.$$
\end{corollary}

\begin{proof}
    By Lemma \ref{lem:intersection_MRD} we have that $\C_0=\C\cap \F^{[k+r]\times [k]}$ is a $[(k+r)\times k,(k+r)(k-d+1),d]_{\F}$ MRD code.
    Moreover, Lemma \ref{lem:matrices1-0} ensures the existence of such matrices $M^{(i,j)}$. These matrices are linearly independent, and they are also independent from $\C_0$. For a dimension argument, the statement follows.
\end{proof}

We are now ready to introduce an open problem in the theory of rank-metric codes, which ends up being equivalent to a smaller instance of the Etzion-Silberstein conjecture.

\begin{conjecture}\label{conj:punct_ext_MRD}
    Let $n,d \in \N$ with $1<d<n$, and let $\F$ be a finite field. There exists an
    $[n\times (n-1),n(n-d),d]_\F$ MRD code $\C$ such that its puncturing on one row  is contained in an $[(n-1)\times (n-1),(n-1)(n-d+1),d-1]_\F$ MRD code.
\end{conjecture}

Concerning Conjecture \ref{conj:punct_ext_MRD}, we can observe that when we puncture a $[n\times (n-1),n(n-d),d]_\F$ MRD code on one row, the minimum rank must decrease by one, since otherwise it would violate the  bound of Theorem \ref{thm:singES_bound}. Thus, we obtain a $[(n-1)\times (n-1),n(n-d),d-1]_\F$ code. In order to extend it to an $[(n-1)\times (n-1),(n-1)(n-d+1),d-1]_\F$ MRD code, we need to add a $(d-1)$-dimensional subspace to it.

\begin{theorem}\label{thm:equiv_conj_Ekdr}
    Conjecture \ref{conj:punct_ext_MRD} is equivalent to the existence of $[\mE_{n-1,d,1},\nu_{\min}(\mE_{n-1,d,1},d),d]_{\F}$ MFD codes.
\end{theorem}

\begin{proof}
  Let us fix $d$ and $n$ with $1<d<n$.    We show the two implications separately. 

  Assume that Conjecture \ref{conj:punct_ext_MRD} holds true, and hence, there exists an $[n\times (n-1),n(n-d),d]_{\F}$ MRD code $\C$ such that its puncturing on the first row $\pi(\C)$ is contained in a $[(n-1)\times (n-1),(n-1)(n-d+1),d-1]_\F$ MRD code $\tilde{\C}$. Let us write $\tilde{\C}=\pi(\C)\oplus \C'$, with $\dim(\C')=d-1$. We choose a basis $A^{(1)},\ldots, A^{(d-1)}$ for $\C'$, and, for each $\ell \in [d-1]$, we define $B^{(\ell)}\in \F^{\mE_{n-1,d,1}}$ as the matrix $B^{(\ell)}=(B^{(\ell)}_{i,j})$ given by
  $$B^{(\ell)}_{i,j}=\begin{cases}
   1 & \mbox{ if } (i,j)=(1,n-1+\ell) \\
      A^{(\ell)}_{i-1,j} & \mbox{ if } 2\le i \le n-1 \\
      0 & \mbox{ if } i=1, j \in [n-1]
  \end{cases}$$
  We claim that the code $\C\oplus \langle B^{(1)},\ldots,B^{(d-1)}\rangle_{\F}$ is a $[\mE_{n-1,d,1},\nu_{\min}(\mE_{n-1,d,1},d),d]_{\F}$ MFD code. First, it is clear that $\dim_\F (\C\oplus \langle B^{(1)},\ldots,B^{(d-1)}\rangle_{\F})=\nu_{\min}(\mE_{n-1,d,1},d)=(n-1)(n-d+1)+(d-1)$, since the $B^{(i)}$'s are linearly independent from $\C$, by how they are defined on the first row. We only need to show that every nonzero element in $\C\oplus \langle B^{(1)},\ldots,B^{(d-1)}\rangle_{\F}$ has rank at least $d$. Clearly, if $B\in \C$ is nonzero, then $\rk(B)\ge d$, since $\C$ is a  $[n\times (n-1),n(n-d),d]_{\F}$ MRD code. Thus, assume that $B\in \C\oplus \langle B^{(1)},\ldots,B^{(d-1)}\rangle_{\F}\setminus \C$. The matrix $B$ can be written as 
  $B=B_0+\sum_{i=1}^{d-1}\lambda_iB^{(i)}$, for some $\lambda_1,\ldots,\lambda_{d-1}\in \F$ not all zero. If we write
  $$B_0=\begin{pmatrix} \beta_1 & \cdots & \beta_{n-1} \\ \cline{1-3}
  \\
  & A_0 & \\
  \\
  \end{pmatrix}.$$
  Therefore, we have
    $$B=B_0+\sum_{i=1}^{d-1}\lambda_iB^{(i)}= \left(
    \begin{array}{cccccc}   \beta_1 \cellcolor{blue!10}& \ldots \cellcolor{blue!10} & \beta_{n-1} \cellcolor{blue!10}& \lambda_1\cellcolor{blue!10} &\cellcolor{blue!10} \cdots & \lambda_{d-1}\cellcolor{blue!10} \\
    \cline{1-6}
      \multicolumn{1}{|c}{\phantom{0}\cellcolor{blue!10}} & \multicolumn{1}{c}{\phantom{0} \cellcolor{blue!10}} & \multicolumn{1}{c|}{\phantom{0} \cellcolor{blue!10}} & \\
       \multicolumn{1}{|c}{\phantom{0}\cellcolor{blue!10}} & \multicolumn{1}{c}{\phantom{0} \cellcolor{blue!10}} & \multicolumn{1}{c|}{\phantom{0} \cellcolor{blue!10}} &  \\
      \multicolumn{1}{|c}{\phantom{0}\cellcolor{blue!10}} & \multicolumn{1}{c}{A_0+\sum\limits_{i=1}^{d-1}\lambda_iA^{(i)}\cellcolor{blue!10}} & \multicolumn{1}{c|}{\phantom{0}\cellcolor{blue!10}} & & \mathbf{0}  \\ 
        \multicolumn{1}{|c}{\phantom{0}\cellcolor{blue!10}} & \multicolumn{1}{c}{\phantom{0}\cellcolor{blue!10}}& \multicolumn{1}{c|}{\phantom{0}\cellcolor{blue!10}} &  \\ 
         \multicolumn{1}{c}{\phantom{0}\cellcolor{blue!10}} & \multicolumn{1}{c}{\phantom{0} \cellcolor{blue!10}} & \multicolumn{1}{c|}{\phantom{0} \cellcolor{blue!10}} & \\\cline{1-3}   
    \end{array}
    \right).$$
  Note that, $A_0\in \pi(\C)$ and $A^{(i)}\in \C'$, hence the matrix 
  $$A_0+\sum\limits_{i=1}^{d-1}\lambda_iA^{(i)}$$ is a noznero element in the $[(n-1)\times (n-1),(n-1)(n-d+1),d-1]_\F$ MRD code $\tilde{\C}$,  thus its rank is at least $d-1$.
    By the block structure of $B$ and the fact that not all the $\lambda_i$'s are zero, we have
$$\rk(B)=1+\rk\left(A_0+\sum\limits_{i=1}^{d-1}\lambda_iA^{(i)}\right)\geq 1+d-1,$$
    which concludes the first implication.

    Suppose now that there exists an $[\mE_{n-1,d,1},\nu_{\min}(\mE_{n-1,d,1},d),d]_{\F}$ MFD code $\C$. By Corollary \ref{cor:E_MFD_MRD}, there exists an $[n\times (n-1),n(n-d),d]_{\F}$ MRD code $\C_0$ and matrices $M^{(1)},\ldots,M^{(d-1)}$ of the form
    $$(M^{(i)})_{a,b}=\begin{cases}
        1 & \mbox{ if } (a,b)=(1,i+n-1) \\
        0 & \mbox{ if } a=1, b\in \{n,\ldots,n+d-1\}\setminus\{j\} \\ A^{(i)}_{a-1,b}
         & \mbox{ 
 otherwise,}
    \end{cases}$$
    for some matrices $A^{(i)}\in \F^{\F^{(n-1)\times (n-1)}}$,
    such that $\C=\C_0\oplus \langle M^{(1)},\ldots, M^{(d-1)}\rangle_{\F}$.  The proof goes similarly. We have that for every $B_0\in\C_0$ and every $\lambda_1,\ldots,\lambda_{d-1}$, the matrix $$B=B_0+\sum_{i=1}^{d-1}\lambda_i M_i$$ belongs to $\C$, and hence its rank is at least $d$, if it is a nonzero matrix.
    If we write   $$B_0=\begin{pmatrix} \beta_1 & \cdots & \beta_{n-1} \\ \cline{1-3}
  \\
  & A_0 & \\
  \\
  \end{pmatrix}.$$
  Therefore, we have
    $$B=B_0+\sum_{i=1}^{d-1}\lambda_iM^{(i)}= \left(
    \begin{array}{cccccc}   \beta_1 \cellcolor{blue!10}& \ldots \cellcolor{blue!10} & \beta_{n-1} \cellcolor{blue!10}& \lambda_1\cellcolor{blue!10} &\cellcolor{blue!10} \cdots & \lambda_{d-1}\cellcolor{blue!10} \\
    \cline{1-6}
      \multicolumn{1}{|c}{\phantom{0}\cellcolor{blue!10}} & \multicolumn{1}{c}{\phantom{0} \cellcolor{blue!10}} & \multicolumn{1}{c|}{\phantom{0} \cellcolor{blue!10}} & \\
       \multicolumn{1}{|c}{\phantom{0}\cellcolor{blue!10}} & \multicolumn{1}{c}{\phantom{0} \cellcolor{blue!10}} & \multicolumn{1}{c|}{\phantom{0} \cellcolor{blue!10}} & \\
      \multicolumn{1}{|c}{\phantom{0}\cellcolor{blue!10}} & \multicolumn{1}{c}{A_0+\sum\limits_{i=1}^{d-1}\lambda_iA^{(i)}\cellcolor{blue!10}} & \multicolumn{1}{c|}{\phantom{0}\cellcolor{blue!10}} &  & \mathbf{0} \\ 
        \multicolumn{1}{|c}{\phantom{0}\cellcolor{blue!10}} & \multicolumn{1}{c}{\phantom{0}\cellcolor{blue!10}}& \multicolumn{1}{c|}{\phantom{0}\cellcolor{blue!10}} &  \\ 
         \multicolumn{1}{c}{\phantom{0}\cellcolor{blue!10}} & \multicolumn{1}{c}{\phantom{0} \cellcolor{blue!10}} & \multicolumn{1}{c|}{\phantom{0} \cellcolor{blue!10}} & \\\cline{1-3}   
    \end{array}
    \right).$$
    Note that, by the block matrix structure we have, when $B$ is nonzero, that
    $$d\le \rk(B)=\begin{cases}1+\rk\left(A_0+\sum\limits_{i=1}^{d-1}\lambda_iA^{(i)}\right) & \mbox{ if } (\lambda_1,\ldots,\lambda_{d-1})\neq(0,\ldots,0) \\
    \rk(B_0) &  \mbox{ if } \lambda_1=\ldots=\lambda_{d-1}=0\end{cases}$$
implying that, for every 
$B_0 \in \C_0$ and every $\lambda_1,\ldots,\lambda_{d-1}$, we have
$$\rk\left(A_0+\sum\limits_{i=1}^{d-1}\lambda_iA^{(i)}\right)\ge d-1$$
whenever the matrix $B$ is nonzero. In particular, the spaces 
$$\pi(\C_0):=\langle A_0 \,:\, B_0 \in \C_0\rangle_{|F}\quad \mbox{and} \quad  \langle A^{(1)},\ldots,A^{(d-1)}\rangle_{\F}$$
must intersect trivially, otherwise there would be a matrix of rank $1<d$ in $\C$. 
Thus, their span is a $[(n-1)\times(n-1),(n-1)(n-d),d-1]_{\F}$ MRD code, which contains $\pi(\C_0)$.

\end{proof}

\begin{remark}
    Theorem \ref{thm:equiv_conj_Ekdr} gives an idea on how difficult is proving the whole Etzion-Silberstein conjecture. Indeed, proving the puncturing-inclusion MRD conjecture -- that is, Conjecture \ref{conj:punct_ext_MRD} -- is equivalent to  show the Etzion-Silberstein conjecture only for one family of irreducible instances of the Etzion--Silberstein conjecture, the one given by $\{(\mE_{n-1,d,1},d)\,:\, {d,n \in \mathbb N}\}$. On the contrary, we have shown that there are many more irreducible Ferrers diagrams families for which the Etzion-Silberstein conjecture is still open.
\end{remark}

\bigskip 

\bibliographystyle{abbrv}
\bibliography{biblio}

\end{document}

%% file: TikZ_pictures/BulletsDigraph_n=3d=3.tex
\begin{tikzcd}[every arrow/.append style={dash},column sep={3.5cm,between origins},row sep={3cm,between origins}] 
 &  & \ydiagram[*(red!50) \bullet]{3,3,3}\ar[->>,thick]{d} & &\\
 &  & {\ydiagram[\bullet]{3,3,2}}\ar[->>,thick]{ld}\ar[->>,thick]{rd} & & \\
 &\ydiagram[\bullet]{3,3,1}\ar[->>,thick]{ld}\ar[<<-,thick]{rd} & & \ydiagram[\bullet]{3,2,2}\ar[<<-,thick]{ld}\ar[->>,thick]{rd}&  \\
 \ydiagram[\bullet]{3,3} & & \ydiagram[*(red!50) \bullet]{3,2,1}\ar[->>,thick]{ld}\ar[->>,thick]{d}\ar[->>,thick]{rd} & & \ydiagram[\bullet]{2,2,2}  \\
  &\ydiagram[\bullet]{3,2}\ar[->>,thick]{lu} & \ydiagram[\bullet]{3,1,1} & \ydiagram[\bullet]{2,2,1}\ar[->>,thick]{ru}&  \\
   & \ydiagram[\bullet]{3,1}\ar[->>,thick]{u}\ar[->>,thick]{ru} & \ydiagram[\bullet]{2, 2}\ar[->>,thick]{lu}\ar[->>,thick]{ru} & \ydiagram[\bullet]{2, 1, 1}\ar[->>,thick]{lu}\ar[->>,thick]{u} & & \\
                \ydiagram[\bullet]{3}\ar[->>,thick]{ru} & & \ydiagram[\bullet]{2, 1}\ar[->>,thick]{lu}\ar[->>,thick]{ru}\ar[->>,thick]{u} & & \ydiagram[\bullet]{1, 1, 1}\ar[->>,thick]{lu}\\
                &\ydiagram[\bullet]{2}\ar[->>,thick]{ru}\ar[->>,thick]{lu}& &\ydiagram[\bullet]{1, 1} \ar[->>,thick]{lu}\ar[->>,thick]{ru}\\
                & &\ydiagram[\bullet]{1}  \ar[->>,thick]{ru}\ar[->>,thick]{lu}  & &\\   
                & & \mathlarger{\mathlarger{\mathlarger{\mathlarger{\mathlarger{\colorbox{red!50}{$\emptyset$}}}}}}\ar[->>,thick]{u}  
\end{tikzcd}

%% file: TikZ_pictures/BulletsDigraph_n=4_d=3.tex
\adjustbox{scale=0.3,center}{
\begin{tikzcd}[every arrow/.append style={dash},column sep={2cm,between origins},row sep={4.2cm,between origins}] 
& & &&&& & \ydiagram[*(red!50) \bullet]{4,4,4,4}\ar[->>,thick]{d} & & & && &&\\
& & &&&& & {\ydiagram[\bullet]{4,4,4,3}}\ar[->>,thick]{lld}\ar[->>,thick]{rrd} & & \\
&& & & &\ydiagram[\bullet]{4,4,4,2}\ar[->>,thick]{lld}\ar[->>,thick]{rrd} & &&& \ydiagram[\bullet]{4,4,3,3}\ar[->>,thick]{lld}\ar[->>,thick]{rrd}&  \\&&
 &\ydiagram[\bullet]{4,4,4,1}\ar[->>,thick]{d}\ar[->>,thick]{rrd} & & & &\ydiagram[\bullet]{4,4,3,2}\ar[->>,thick]{d}\ar[->>,thick]{lld}\ar[->>,thick]{rrd} & & &&\ydiagram[\bullet]{4,3,3,3}\ar[->>,thick]{lld}\ar[->>,thick]{d}  \\ &&
& \ydiagram[\bullet]{4,4,4}\ar[->>,thick]{d}&&\ydiagram[\bullet]{4,4,3,1}\ar[->>,thick]{lld}\ar[->>,thick]{d}\ar[<<-,thick]{rrd} && \ydiagram[\bullet]{4,4,2,2}\ar[->>,thick]{lld}\ar[->>,thick]{rrd}& &\ydiagram[\bullet]{4,3,3,2}\ar[->>,thick]{rrd}\ar[->>,thick]{d}\ar[<<-,thick]{lld}& &\ydiagram[\bullet]{3,3,3,3}\ar[->>,thick]{d}\\ & &
& \ydiagram[\bullet]{4,4,3}\ar[->>,thick]{lld}\ar[<<-,thick]{rrd}   & &\ydiagram[\bullet]{4,4,2,1}\ar[->>,thick]{lllld}\ar[->>,thick]{lld}\ar[<<-,thick]{rrd} & & \ydiagram[*(red!50) \bullet]{4,3,3,1}\ar[->>,thick]{lld}\ar[->>,thick]{d}\ar[->>,thick]{rrd} && \ydiagram[\bullet]{4,3,2,2}\ar[->>,thick]{rrrrd}\ar[->>,thick]{rrd}\ar[<<-,thick]{lld} & &\ydiagram[\bullet]{3,3,3,2}\ar[->>,thick]{rrd}\ar[<<-,thick]{lld}  & \\ &
 \ydiagram[\bullet]{4,4,2}\ar[->>,thick]{d}\ar[<<-,thick]{rrd}   &  &   \ydiagram[\bullet]{4,4,1,1} \ar[->>,thick]{lld}\ar[<<-,thick]{rrd}   & & \ydiagram[\bullet]{4,3,3} \ar[->>,thick]{lld}\ar[<<-,thick]{rrd} & &  \ydiagram[\bullet]{4,3,2,1}\ar[->>,thick]{lllld}\ar[->>,thick]{lld}\ar[->>,thick]{rrd}\ar[->>,thick]{rrrrd}& & \ydiagram[\bullet]{3,3,3,1}\ar[->>,thick]{rrd}\ar[<<-,thick]{lld} &  & \ydiagram[\bullet]{4,2,2,2}\ar[->>,thick]{rrd}\ar[<<-,thick]{lld} & &  \ydiagram[\bullet]{3,3,2,2}\ar[->>,thick]{d}\ar[<<-,thick]{lld} \\ &
 \ydiagram[\bullet]{4,4,1} \ar[->>,thick]{ld}\ar[<<-,thick]{rd} & &\ydiagram[\bullet]{4,3,2}\ar[->>,thick]{ld}\ar[<<-,thick]{rd}\ar[<<-,thick]{rrrd} & & \ydiagram[\bullet]{4,3,1,1}\ar[->>,thick]{llld}\ar[->>,thick]{rrrd}\ar[<<-,thick]{rrrrrd}   & &\ydiagram[*(red!50) \bullet]{3,3,3}\ar[->>,thick]{ld}& &\ydiagram[\bullet]{4, 2,2,1}\ar[<<-,thick]{llllld}\ar[->>,thick]{ld}\ar[->>,thick]{rrrd} & &\ydiagram[\bullet]{3,3,2,1}\ar[<<-,thick]{llllld}\ar[<<-,thick]{ld}\ar[->>,thick]{rd} & & \ydiagram[\bullet]{3,2,2,2} \ar[->>,thick]{rd}\ar[<<-,thick]{ld}   \\
\ydiagram[\bullet]{4,4}\ar[<<-,thick]{rd} & & \ydiagram[\bullet]{4,3,1}\ar[->>,thick]{ld}\ar[<<-,thick]{rd}\ar[<<-,thick]{rrrd} & & \ydiagram[*(red!50) \bullet]{4,2,2}\ar[->>,thick]{ld}\ar[->>,thick]{rrrrrd}  & & \ydiagram[\bullet]{3,3,2}\ar[->>,thick]{ld}\ar[->>,thick]{rrrd} & &\ydiagram[\bullet]{4,2,1,1}\ar[<<-,thick]{llllld}\ar[->>,thick]{ld}\ar[<<-,thick]{rrrd} & & \ydiagram[*(red!50) \bullet]{3,3,1,1}\ar[->>,thick]{llllld}\ar[->>,thick]{rd} & & \ydiagram[\bullet]{3,2,2,1}\ar[<<-,thick]{llld}\ar[<<-,thick]{ld}\ar[->>,thick]{rd} & &\ydiagram[\bullet]{2,2,2,2}\ar[<<-,thick]{ld}
 \\&
 \ydiagram[\bullet]{4,3}\ar[<<-,thick]{d}\ar[<<-,thick]{rrd} & & 
\ydiagram[\bullet]{4,2,1}\ar[->>,thick]{d}\ar[->>,thick]{rrd}\ar[<<-,thick]{rrrrd} & & \ydiagram[\bullet]{3,3,1}\ar[->>,thick]{lllld}\ar[<<-,thick]{rrd} & & \ydiagram[\bullet]{4,1,1,1}\ar[<<-,thick]{lld}\ar[<<-,thick]{rrd} & & \ydiagram[\bullet]{3,2,2}\ar[->>,thick]{rrrrd}\ar[<<-,thick]{lld} & & \ydiagram[\bullet]{3,2,1,1}\ar[->>,thick]{d}\ar[->>,thick]{lld}\ar[<<-,thick]{lllld} & & \ydiagram[\bullet]{2,2,2,1}\ar[<<-,thick]{d}\ar[<<-,thick]{lld} \\
& \ydiagram[\bullet]{3,3}\ar[<<-,thick]{rrd} & &\ydiagram[\bullet]{4,2}\ar[<<-,thick]{d}\ar[<<-,thick]{rrd} &&  \ydiagram[\bullet]{4,1,1}\ar[<<-,thick]{d}\ar[<<-,thick]{rrd} & &\ydiagram[*(red!50) \bullet]{3,2,1}\ar[->>,thick]{lllld}\ar[->>,thick]{d}\ar[->>,thick]{rrrrd} & &\ydiagram[\bullet]{3,1,1,1}\ar[<<-,thick]{d}\ar[<<-,thick]{lld}  & &\ydiagram[\bullet]{2,2,1,1}\ar[<<-,thick]{d}\ar[<<-,thick]{lld}& &\ydiagram[\bullet]{2,2,2}\ar[<<-,thick]{lld} \\& &&
 \ydiagram[\bullet]{3,2}\ar[<<-,thick]{rrd}\ar[<<-,thick]{rrrrd} & &\ydiagram[\bullet]{4,1}\ar[<<-,thick]{lld}\ar[<<-,thick]{d}  && \ydiagram[\bullet]{3,1,1}\ar[<<-,thick]{lld}\ar[<<-,thick]{rrd} & &\ydiagram[\bullet]{2,1,1,1}\ar[<<-,thick]{d}\ar[<<-,thick]{rrd}&&\ydiagram[\bullet]{2,2,1}\ar[<<-,thick]{lllld}\ar[<<-,thick]{lld} &\\ &&
& \ydiagram[\bullet]{4} &&\ydiagram[\bullet]{3,1}  && \ydiagram[\bullet]{2,2}  &&\ydiagram[\bullet]{2,1,1} && \ydiagram[\bullet]{1,1,1,1}  \\
&& &  \ydiagram[\bullet]{3}\ar[->>,thick]{rru}\ar[->>,thick]{u} &&& & \ydiagram[\bullet]{2, 1}\ar[->>,thick]{llu}\ar[->>,thick]{rru}\ar[->>,thick]{u} &&& & \ydiagram[\bullet]{1, 1, 1}\ar[->>,thick]{llu}\ar[->>,thick]{u}\\
     & &&  &        &\ydiagram[\bullet]{2}\ar[->>,thick]{rru}\ar[->>,thick]{llu}& && &\ydiagram[\bullet]{1, 1} \ar[->>,thick]{llu}\ar[->>,thick]{rru}\\
&& &&         &       & &\ydiagram[\bullet]{1}  \ar[->>,thick]{rru}\ar[->>,thick]{llu}\\
&& &&         &       & & \mathlarger{\mathlarger{\mathlarger{\mathlarger{\mathlarger{\colorbox{red!50}{$\emptyset$}}}}}}\ar[->>,thick]{u}  
\end{tikzcd}
}\hspace*{-2cm}

%% file: TikZ_pictures/BulletsDigraph_n=4_d=4.tex
\adjustbox{scale=0.3,center}{
\begin{tikzcd}[every arrow/.append style={dash},column sep={2cm,between origins},row sep={4.2cm,between origins}] 
& & &&&& & \ydiagram[*(red!50) \bullet]{4,4,4,4}\ar[->>,thick]{d} & & & && &&\\
& & &&&& & {\ydiagram[\bullet]{4,4,4,3}}\ar[->>,thick]{lld}\ar[->>,thick]{rrd} & & \\
&& & & &\ydiagram[\bullet]{4,4,4,2}\ar[->>,thick]{lld}\ar[<<-,thick]{rrd} & &&& \ydiagram[\bullet]{4,4,3,3}\ar[<<-,thick]{lld}\ar[->>,thick]{rrd}&  \\&&
 &\ydiagram[\bullet]{4,4,4,1}\ar[->>,thick]{d}\ar[<<-,thick]{rrd} & & & &\ydiagram[\bullet]{4,4,3,2}\ar[<<-,thick]{d}\ar[->>,thick]{lld}\ar[->>,thick]{rrd} & & &&\ydiagram[\bullet]{4,3,3,3}\ar[<<-,thick]{lld}\ar[->>,thick]{d}  \\ &&
& \ydiagram[\bullet]{4,4,4}\ar[<<-,thick]{d}&&\ydiagram[\bullet]{4,4,3,1}\ar[->>,thick]{lld}\ar[<<-,thick]{d}\ar[<<-,thick]{rrd} && \ydiagram[*(red!50) \bullet]{4,4,2,2}\ar[->>,thick]{lld}\ar[->>,thick]{rrd}& &\ydiagram[\bullet]{4,3,3,2}\ar[->>,thick]{rrd}\ar[<<-,thick]{d}\ar[<<-,thick]{lld}& &\ydiagram[\bullet]{3,3,3,3}\ar[<<-,thick]{d}\\ & &
& \ydiagram[\bullet]{4,4,3}\ar[<<-,thick]{lld}\ar[<<-,thick]{rrd}   & &\ydiagram[\bullet]{4,4,2,1}\ar[->>,thick]{lllld}\ar[->>,thick]{lld}\ar[<<-,thick]{rrd} & & \ydiagram[\bullet]{4,3,3,1}\ar[->>,thick]{lld}\ar[<<-,thick]{d}\ar[->>,thick]{rrd} && \ydiagram[\bullet]{4,3,2,2}\ar[->>,thick]{rrrrd}\ar[->>,thick]{rrd}\ar[<<-,thick]{lld} & &\ydiagram[\bullet]{3,3,3,2}\ar[<<-,thick]{rrd}\ar[<<-,thick]{lld}  & \\ &
 \ydiagram[\bullet]{4,4,2}\ar[<<-,thick]{d}\ar[<<-,thick]{rrd}   &  &   \ydiagram[\bullet]{4,4,1,1} \ar[<<-,thick]{lld}\ar[<<-,thick]{rrd}   & & \ydiagram[\bullet]{4,3,3} \ar[<<-,thick]{lld}\ar[<<-,thick]{rrd} & &  \ydiagram[*(red!50) \bullet]{4,3,2,1}\ar[->>,thick]{lllld}\ar[->>,thick]{lld}\ar[->>,thick]{rrd}\ar[->>,thick]{rrrrd}& & \ydiagram[\bullet]{3,3,3,1}\ar[<<-,thick]{rrd}\ar[<<-,thick]{lld} &  & \ydiagram[\bullet]{4,2,2,2}\ar[<<-,thick]{rrd}\ar[<<-,thick]{lld} & &  \ydiagram[\bullet]{3,3,2,2}\ar[<<-,thick]{d}\ar[<<-,thick]{lld} \\ &
 \ydiagram[\bullet]{4,4,1} \ar[<<-,thick]{ld}\ar[<<-,thick]{rd} & &\ydiagram[\bullet]{4,3,2}\ar[<<-,thick]{ld}\ar[<<-,thick]{rd}\ar[<<-,thick]{rrrd} & & \ydiagram[\bullet]{4,3,1,1}\ar[<<-,thick]{llld}\ar[<<-,thick]{rrrd}\ar[<<-,thick]{rrrrrd}   & &\ydiagram[\bullet]{3,3,3}\ar[<<-,thick]{ld}& &\ydiagram[\bullet]{4, 2,2,1}\ar[<<-,thick]{llllld}\ar[<<-,thick]{ld}\ar[<<-,thick]{rrrd} & &\ydiagram[\bullet]{3,3,2,1}\ar[<<-,thick]{llllld}\ar[<<-,thick]{ld}\ar[<<-,thick]{rd} & & \ydiagram[\bullet]{3,2,2,2} \ar[<<-,thick]{rd}\ar[<<-,thick]{ld}   \\
\ydiagram[\bullet]{4,4}\ar[<<-,thick]{rd} & & \ydiagram[\bullet]{4,3,1}\ar[<<-,thick]{ld}\ar[<<-,thick]{rd}\ar[<<-,thick]{rrrd} & & \ydiagram[\bullet]{4,2,2}\ar[<<-,thick]{ld}\ar[<<-,thick]{rrrrrd}  & & \ydiagram[\bullet]{3,3,2}\ar[<<-,thick]{ld}\ar[<<-,thick]{rrrd} & &\ydiagram[\bullet]{4,2,1,1}\ar[<<-,thick]{llllld}\ar[<<-,thick]{ld}\ar[<<-,thick]{rrrd} & & \ydiagram[\bullet]{3,3,1,1}\ar[<<-,thick]{llllld}\ar[<<-,thick]{rd} & & \ydiagram[\bullet]{3,2,2,1}\ar[<<-,thick]{llld}\ar[<<-,thick]{ld}\ar[<<-,thick]{rd} & &\ydiagram[\bullet]{2,2,2,2}\ar[<<-,thick]{ld}
 \\&
 \ydiagram[\bullet]{4,3}\ar[<<-,thick]{d}\ar[<<-,thick]{rrd} & & 
\ydiagram[\bullet]{4,2,1}\ar[<<-,thick]{d}\ar[<<-,thick]{rrd}\ar[<<-,thick]{rrrrd} & & \ydiagram[\bullet]{3,3,1}\ar[<<-,thick]{lllld}\ar[<<-,thick]{rrd} & & \ydiagram[\bullet]{4,1,1,1}\ar[<<-,thick]{lld}\ar[<<-,thick]{rrd} & & \ydiagram[\bullet]{3,2,2}\ar[<<-,thick]{rrrrd}\ar[<<-,thick]{lld} & & \ydiagram[\bullet]{3,2,1,1}\ar[<<-,thick]{d}\ar[<<-,thick]{lld}\ar[<<-,thick]{lllld} & & \ydiagram[\bullet]{2,2,2,1}\ar[<<-,thick]{d}\ar[<<-,thick]{lld} \\
& \ydiagram[\bullet]{3,3}\ar[<<-,thick]{rrd} & &\ydiagram[\bullet]{4,2}\ar[<<-,thick]{d}\ar[<<-,thick]{rrd} &&  \ydiagram[\bullet]{4,1,1}\ar[<<-,thick]{d}\ar[<<-,thick]{rrd} & &\ydiagram[\bullet]{3,2,1}\ar[<<-,thick]{lllld}\ar[<<-,thick]{d}\ar[<<-,thick]{rrrrd} & &\ydiagram[\bullet]{3,1,1,1}\ar[<<-,thick]{d}\ar[<<-,thick]{lld}  & &\ydiagram[\bullet]{2,2,1,1}\ar[<<-,thick]{d}\ar[<<-,thick]{lld}& &\ydiagram[\bullet]{2,2,2}\ar[<<-,thick]{lld} \\& &&
 \ydiagram[\bullet]{3,2}\ar[<<-,thick]{rrd}\ar[<<-,thick]{rrrrd} & &\ydiagram[\bullet]{4,1}\ar[<<-,thick]{lld}\ar[<<-,thick]{d}  && \ydiagram[\bullet]{3,1,1}\ar[<<-,thick]{lld}\ar[<<-,thick]{rrd} & &\ydiagram[\bullet]{2,1,1,1}\ar[<<-,thick]{d}\ar[<<-,thick]{rrd}&&\ydiagram[\bullet]{2,2,1}\ar[<<-,thick]{lllld}\ar[<<-,thick]{lld} &\\ &&
& \ydiagram[\bullet]{4} &&\ydiagram[\bullet]{3,1}  && \ydiagram[\bullet]{2,2}  &&\ydiagram[\bullet]{2,1,1} && \ydiagram[\bullet]{1,1,1,1}  \\
&& &  \ydiagram[\bullet]{3}\ar[->>,thick]{rru}\ar[->>,thick]{u} &&& & \ydiagram[\bullet]{2, 1}\ar[->>,thick]{llu}\ar[->>,thick]{rru}\ar[->>,thick]{u} &&& & \ydiagram[\bullet]{1, 1, 1}\ar[->>,thick]{llu}\ar[->>,thick]{u}\\
     & &&  &        &\ydiagram[\bullet]{2}\ar[->>,thick]{rru}\ar[->>,thick]{llu}& && &\ydiagram[\bullet]{1, 1} \ar[->>,thick]{llu}\ar[->>,thick]{rru}\\
&& &&         &       & &\ydiagram[\bullet]{1}  \ar[->>,thick]{rru}\ar[->>,thick]{llu}\\
&& &&         &       & & \mathlarger{\mathlarger{\mathlarger{\mathlarger{\mathlarger{\colorbox{red!50}{$\emptyset$}}}}}}\ar[->>,thick]{u}  
\end{tikzcd}
}

%% file: TikZ_pictures/BulletsSubDigraph_n=3_d=3.tex
\begin{tikzcd}[every arrow/.append style={dash},column sep={3.5cm,between origins},row sep={3cm,between origins}] 
 &  & \ydiagram[*(red!50) \bullet]{3,3,3}\ar[->>,thick]{d} & &\\
 &  & {\ydiagram[\bullet]{3,3,2}}\ar[->>,thick]{ld}\ar[->>,thick]{rd} & & \\
 &\ydiagram[\bullet]{3,3,1}\ar[<<-,thick]{rd} & & \ydiagram[\bullet]{3,2,2}\ar[<<-,thick]{ld}&  \\
  & & \ydiagram[*(red!50) \bullet]{3,2,1} & &   
\end{tikzcd}

%% file: TikZ_pictures/BulletsSubDigraph_n=4_d=3.tex
\adjustbox{scale=0.3,center}{
\begin{tikzcd}[every arrow/.append style={dash},column sep={2cm,between origins},row sep={4.2cm,between origins}] 
& & &&&& & \ydiagram[*(red!50) \bullet]{4,4,4,4}\ar[->>,thick]{d} & & & && &&\\
& & &&&& & {\ydiagram[\bullet]{4,4,4,3}}\ar[->>,thick]{lld}\ar[->>,thick]{rrd} & & \\
&& & & &\ydiagram[\bullet]{4,4,4,2}\ar[->>,thick]{lld}\ar[->>,thick]{rrd} & &&& \ydiagram[\bullet]{4,4,3,3}\ar[->>,thick]{lld}\ar[->>,thick]{rrd}&  \\&&
 &\ydiagram[\bullet]{4,4,4,1}\ar[->>,thick]{d}\ar[->>,thick]{rrd} & & & &\ydiagram[\bullet]{4,4,3,2}\ar[->>,thick]{d}\ar[->>,thick]{lld}\ar[->>,thick]{rrd} & & &&\ydiagram[\bullet]{4,3,3,3}\ar[->>,thick]{lld}\ar[->>,thick]{d}  \\ &&
& \ydiagram[\bullet]{4,4,4}\ar[->>,thick]{d}&&\ydiagram[\bullet]{4,4,3,1}\ar[->>,thick]{lld}\ar[->>,thick]{d}\ar[<<-,thick]{rrd} && \ydiagram[\bullet]{4,4,2,2}\ar[->>,thick]{lld}\ar[->>,thick]{rrd}& &\ydiagram[\bullet]{4,3,3,2}\ar[->>,thick]{rrd}\ar[->>,thick]{d}\ar[<<-,thick]{lld}& &\ydiagram[\bullet]{3,3,3,3}\ar[->>,thick]{d}\\ & &
& \ydiagram[\bullet]{4,4,3}\ar[->>,thick]{lld}\ar[<<-,thick]{rrd}   & &\ydiagram[\bullet]{4,4,2,1}\ar[->>,thick]{lllld}\ar[->>,thick]{lld}\ar[<<-,thick]{rrd} & & \ydiagram[*(red!50) \bullet]{4,3,3,1}\ar[->>,thick]{lld}\ar[->>,thick]{d}\ar[->>,thick]{rrd} && \ydiagram[\bullet]{4,3,2,2}\ar[->>,thick]{rrrrd}\ar[->>,thick]{rrd}\ar[<<-,thick]{lld} & &\ydiagram[\bullet]{3,3,3,2}\ar[->>,thick]{rrd}\ar[<<-,thick]{lld}  & \\ &
 \ydiagram[\bullet]{4,4,2}\ar[->>,thick]{d}\ar[<<-,thick]{rrd}   &  &   \ydiagram[\bullet]{4,4,1,1} \ar[->>,thick]{lld}\ar[<<-,thick]{rrd}   & & \ydiagram[\bullet]{4,3,3} \ar[->>,thick]{lld}\ar[<<-,thick]{rrd} & &  \ydiagram[\bullet]{4,3,2,1}\ar[->>,thick]{lllld}\ar[->>,thick]{lld}\ar[->>,thick]{rrd}\ar[->>,thick]{rrrrd}& & \ydiagram[\bullet]{3,3,3,1}\ar[->>,thick]{rrd}\ar[<<-,thick]{lld} &  & \ydiagram[\bullet]{4,2,2,2}\ar[->>,thick]{rrd}\ar[<<-,thick]{lld} & &  \ydiagram[\bullet]{3,3,2,2}\ar[->>,thick]{d}\ar[<<-,thick]{lld} \\ &
 \ydiagram[\bullet]{4,4,1}\ar[<<-,thick]{rd} & &\ydiagram[\bullet]{4,3,2}\ar[->>,thick]{ld}\ar[<<-,thick]{rd}\ar[<<-,thick]{rrrd} & & \ydiagram[\bullet]{4,3,1,1}\ar[->>,thick]{llld}\ar[->>,thick]{rrrd}\ar[<<-,thick]{rrrrrd}   & &\ydiagram[*(red!50) \bullet]{3,3,3}\ar[->>,thick]{ld}& &\ydiagram[\bullet]{4, 2,2,1}\ar[<<-,thick]{llllld}\ar[->>,thick]{ld}\ar[->>,thick]{rrrd} & &\ydiagram[\bullet]{3,3,2,1}\ar[<<-,thick]{llllld}\ar[<<-,thick]{ld}\ar[->>,thick]{rd} & & \ydiagram[\bullet]{3,2,2,2}\ar[<<-,thick]{ld}   \\
  & & \ydiagram[\bullet]{4,3,1}\ar[<<-,thick]{rd}\ar[<<-,thick]{rrrd} & & \ydiagram[*(red!50) \bullet]{4,2,2}\ar[->>,thick]{ld}\ar[->>,thick]{rrrrrd}  & & \ydiagram[\bullet]{3,3,2}\ar[->>,thick]{ld}\ar[->>,thick]{rrrd} & &\ydiagram[\bullet]{4,2,1,1}\ar[<<-,thick]{llllld}\ar[<<-,thick]{rrrd} & & \ydiagram[*(red!50) \bullet]{3,3,1,1}\ar[->>,thick]{llllld}\ar[->>,thick]{rd} & & \ydiagram[\bullet]{3,2,2,1}\ar[<<-,thick]{llld}\ar[<<-,thick]{ld} & & 
 \\&
 & & 
\ydiagram[\bullet]{4,2,1}\ar[<<-,thick]{rrrrd} & & \ydiagram[\bullet]{3,3,1}\ar[<<-,thick]{rrd} & &  & & \ydiagram[\bullet]{3,2,2}\ar[<<-,thick]{lld} & & \ydiagram[\bullet]{3,2,1,1}\ar[<<-,thick]{lllld} & &   \\
&   & &  &&    & &\ydiagram[*(red!50) \bullet]{3,2,1}  & &  & & & & 
\end{tikzcd}
}\hspace*{-2cm}

%% file: TikZ_pictures/BulletsSubDigraph_n=4_d=4.tex
\adjustbox{scale=0.3,center}{
\begin{tikzcd}[every arrow/.append style={dash},column sep={2cm,between origins},row sep={4.2cm,between origins}] 
& & &&&& & \ydiagram[*(red!50) \bullet]{4,4,4,4}\ar[->>,thick]{d} & & & && &&\\
& & &&&& & {\ydiagram[\bullet]{4,4,4,3}}\ar[->>,thick]{lld}\ar[->>,thick]{rrd} & & \\
&& & & &\ydiagram[\bullet]{4,4,4,2}\ar[->>,thick]{lld}\ar[<<-,thick]{rrd} & &&& \ydiagram[\bullet]{4,4,3,3}\ar[<<-,thick]{lld}\ar[->>,thick]{rrd}&  \\&&
 &\ydiagram[\bullet]{4,4,4,1}\ar[<<-,thick]{rrd} & & & &\ydiagram[\bullet]{4,4,3,2}\ar[<<-,thick]{d}\ar[->>,thick]{lld}\ar[->>,thick]{rrd} & & &&\ydiagram[\bullet]{4,3,3,3}\ar[<<-,thick]{lld} \\ &&
&  &&\ydiagram[\bullet]{4,4,3,1}\ar[<<-,thick]{d}\ar[<<-,thick]{rrd} && \ydiagram[*(red!50) \bullet]{4,4,2,2}\ar[->>,thick]{lld}\ar[->>,thick]{rrd}& &\ydiagram[\bullet]{4,3,3,2}\ar[<<-,thick]{d}\ar[<<-,thick]{lld}& &\\ & &
&    & &\ydiagram[\bullet]{4,4,2,1}\ar[<<-,thick]{rrd} & & \ydiagram[\bullet]{4,3,3,1}\ar[<<-,thick]{d} && \ydiagram[\bullet]{4,3,2,2}\ar[<<-,thick]{lld} & &  & \\ &
   &  &      & & & &  \ydiagram[*(red!50) \bullet]{4,3,2,1}& &  &  &  & &  \\ &
 
\end{tikzcd}
}

%% file: TikZ_pictures/BulletsSubDigraph_n=5_d=5.tex
\adjustbox{scale=0.6,center}{

\begin{tikzcd}[every arrow/.append style={dash},column sep={2.5cm,between origins},row sep={5.6cm,between origins}] 
& & &&&& & \ydiagram[*(red!50) \bullet]{5,5,5,5,5}\ar[->>,thick]{d} & & & && &&\\& & &&&& & \ydiagram[\bullet]{5,5,5,5,4}\ar[->>,thick]{ld}\ar[->>,thick]{rd} & & & && &&\\ & & &  &  & & \ydiagram[\bullet]{5,5,5,5,3}\ar[->>,thick]{ld}\ar[<<-,thick]{rd} &&  \ydiagram[\bullet]{5,5,5,4,4}\ar[->>,thick]{rd}\ar[<<-,thick]{ld} & &\\ & & &  & &\ydiagram[\bullet]{5,5,5,5,2}\ar[->>,thick]{lld}\ar[<<-,thick]{d} & & \ydiagram[\bullet]{5,5,5,4,3}\ar[->>,thick]{lld}\ar[<<-,thick]{d}\ar[->>,thick]{rrd} &&  \ydiagram[\bullet]{5,5,4,4,4}\ar[->>,thick]{rrd}\ar[<<-,thick]{d} & & \\ & & & \ydiagram[\bullet]{5,5,5,5,1}\ar[<<-,thick]{d} & &\ydiagram[\bullet]{5,5,5,4,2}\ar[->>,thick]{lld}\ar[<<-,thick]{d}\ar[<<-,thick]{rrd} & & \ydiagram[*(red!50) \bullet]{5,5,5,3,3}\ar[->>,thick]{lld}\ar[->>,thick]{rrd} &&  \ydiagram[\bullet]{5,5,4,4,3}\ar[->>,thick]{rrd}\ar[<<-,thick]{d}\ar[<<-,thick]{lld}  & & \ydiagram[\bullet]{5,4,4,4,4}\ar[<<-,thick]{d} \\
 & & & \ydiagram[\bullet]{5,5,5,4,1}\ar[<<-,thick]{lld}\ar[<<-,thick]{rrd}   & &\ydiagram[\bullet]{5,5,5,3,2}\ar[->>,thick]{lllld} \ar[<<-,thick]{lld}\ar[<<-,thick]{rrd}   & & \ydiagram[\bullet]{5,5,4,4,2}\ar[->>,thick]{lld}\ar[<<-,thick]{d}\ar[->>,thick]{rrd} &&  \ydiagram[\bullet]{5,5,4,3,3}\ar[->>,thick]{rrrrd} \ar[<<-,thick]{rrd}\ar[<<-,thick]{lld} & & \ydiagram[\bullet]{5,4,4,4,3}\ar[<<-,thick]{rrd}\ar[<<-,thick]{lld}  \\ 
 &\ydiagram[\bullet]{5,5,5,3,1}\ar[<<-,thick]{d}\ar[<<-,thick]{rrd}  & & \ydiagram[\bullet]{5,5,5,2,2}\ar[->>,thick]{lld}\ar[<<-,thick]{rrd} & &\ydiagram[\bullet]{5,5,4,4,1}\ar[<<-,thick]{lld}\ar[<<-,thick]{rrd} & & \ydiagram[\bullet]{5,5,4,3,2}\ar[->>,thick]{lllld}\ar[<<-,thick]{lld}\ar[<<-,thick]{rrd}\ar[->>,thick]{rrrrd} &&\ydiagram[\bullet]{5,4,4,4,2}\ar[<<-,thick]{rrd}\ar[<<-,thick]{lld} & & \ydiagram[\bullet]{5,5,3,3,3}\ar[->>,thick]{rrd}\ar[<<-,thick]{lld} &&\ydiagram[\bullet]{5,4,4,3,3}\ar[<<-,thick]{d}\ar[<<-,thick]{lld} & &  \\
 &\ydiagram[\bullet]{5,5,5,2,1}\ar[<<-,thick]{rd}  & & \ydiagram[\bullet]{5,5,4,3,1}\ar[<<-,thick]{ld}\ar[<<-,thick]{rd}\ar[<<-,thick]{rrrd} & &\ydiagram[*(red!50) \bullet]{5,5,4,2,2}\ar[->>,thick]{llld}\ar[->>,thick]{rrrd}\ar[->>,thick]{rrrrrd} & & \ydiagram[\bullet]{5,4,4,4,1}\ar[<<-,thick]{ld} & & \ydiagram[*(red!50) \bullet]{5,5,3,3,2}\ar[->>,thick]{llllld}\ar[->>,thick]{ld}\ar[->>,thick]{rrrd} &&\ydiagram[\bullet]{5,4,4,3,2}\ar[<<-,thick]{llllld}\ar[<<-,thick]{rd}\ar[<<-,thick]{ld}    & & \ydiagram[\bullet]{5,4,3,3,3}\ar[<<-,thick]{ld} \\& &
\ydiagram[\bullet]{5,5,4,2,1}\ar[<<-,thick]{rrd}\ar[<<-,thick]{rrrrd}  &    & \ydiagram[\bullet]{5,5,3,3,1}\ar[<<-,thick]{d}\ar[<<-,thick]{rrrrd} & &  \ydiagram[\bullet]{5,4,4,3,1}\ar[<<-,thick]{d}\ar[<<-,thick]{rrd} & &\ydiagram[\bullet]{5,5,3,2,2}\ar[<<-,thick]{lllld}\ar[<<-,thick]{rrd} && \ydiagram[\bullet]{5,4,4,2,2}\ar[<<-,thick]{lllld}\ar[<<-,thick]{d} & & \ydiagram[\bullet]{5,4,3,3,2}\ar[<<-,thick]{lllld}\ar[<<-,thick]{lld} \\  & &
&    & \ydiagram[\bullet]{5,5,3,2,1}\ar[<<-,thick]{rrrd} & & \ydiagram[\bullet]{5,4,4,2,1}\ar[<<-,thick]{rd} && \ydiagram[\bullet]{5,4,3,3,1}\ar[<<-,thick]{ld} & & \ydiagram[\bullet]{5,4,3,2,2}\ar[<<-,thick]{llld} & & \\ &
   &  &      & & & &  \ydiagram[*(red!50) \bullet]{5,4,3,2,1}& &  &  &  & &  
\end{tikzcd}
}

%% file: TikZ_pictures/Polytope_P3.tex
\begin{center}
		\tdplotsetmaincoords{60}{115}
		\begin{tikzpicture}[tdplot_main_coords, scale =1.2]
		\draw[thick,->] (0,0,0)--(4,0,0) node[anchor=north east]{\scalebox{0.8}{$x$}};
		\draw[thick,->] (0,0,0)--(0,3,0) node[anchor=north west]{\scalebox{0.8}{$y$}};
		\draw[thick,->] (0,0,0)--(0,0,2) node[anchor=south]{\scalebox{0.8}{$z$}};

		\draw[dashed,gray] (0,0,0)--(0,0,-1);

		\draw[dashed,gray] (2,0,0)--(2,0,-1);
		\draw[dashed,gray] (0,0,-1)--(2,0,-1);
		
		\draw[dashed,gray] (1,0,0)--(1,1,0);
		\draw[dashed,gray] (0,1,0)--(1,1,0);
		
		\draw[dashed,gray] (0,2,0)--(0,2,1);	\draw[dashed,gray] (0,0,1)--(0,2,1);
	
	\filldraw[
        draw=blue,%
        fill=blue!50, opacity=0.2%
    ]          (0,0,0)
            -- (2,0,-1)
            -- (0,2,1)
            -- cycle;

		\draw[thick,blue,-] (1.99,0,-0.99)--(0, 1.99,0.99);
		
		\fill[fill=black] (2,0,-1) circle (2.5pt) 
		node[below right] {\contour{white}{\scalebox{0.8}{$\mathcal{E}$}}};

        \draw[]
        (1.9,0.2,-1.4)  
		node[below] {\contour{white}{\scalebox{0.7}{$(2,0,-1)$}}};

        \draw[]
        (1.5,1.6,-0)  
		node[below] {\contour{white}{\scalebox{0.7}{$(1,1,0)$}}};

        \draw[]
        (0,2,0.9)  
		node[right] {\contour{white}{\scalebox{0.7}{$(0,2,1)$}}};

        \draw[]
        (0,0,-0.1)  
		node[left] {\contour{white}{\scalebox{0.7}{$(0,0,0)$}}};
        
		\fill[fill=black] (1,1,0) circle (2.5pt) node[below right] {\contour{white}{\scalebox{0.8}{$\mathcal{G}$}}};
		
		\fill[fill=black] (0,2,1) circle (2.5pt) node[above right] {\contour{white}{\scalebox{0.8}{$\mathcal{F}$}}};

		\draw[thick,blue,-] (0,0,0)--(1.99, 0, -0.99);
		\draw[thick,blue,-] (0,0,0)--(0.99, 0.99,0);
		\draw[thick,blue,-] (0,0,0)--(0, 1.99,0.99);
		
		\fill[fill=black] (0,0,0) circle (2.5pt)
		node[above left] {\contour{white}{\scalebox{0.8}{$\mathcal{A}$}}};

\end{tikzpicture}
\end{center}